\documentclass[a4paper]{amsart}
\usepackage{enumitem}  
\usepackage{cite}
\usepackage{tikz-cd}
\usetikzlibrary{er,positioning}
\usepackage{amsmath,amssymb,amsthm,graphicx,yfonts,mathrsfs,color, mathtools}
\usepackage{amscd}  
\usepackage{stmaryrd,colonequals}
\usepackage[all]{xy} 
\usepackage{multirow} 
\usepackage{array}
\usepackage[utf8]{inputenc}
\usepackage{braket}
\usepackage{rotating}
\usepackage{setspace,comment}
\usepackage{multirow}
\usepackage{rotating}
\usepackage{pifont}
\usepackage{listings}
\PassOptionsToPackage{hyphens}{url}
\definecolor{darkgreen}{rgb}{0,0.5,0}
\usepackage[colorlinks, breaklinks]{hyperref}
\hypersetup{%
  colorlinks = true,
  linkcolor  = black,
  citecolor = darkgreen
}

\usepackage{xcolor}

\definecolor{codegreen}{rgb}{0,0.6,0}
\definecolor{codegray}{rgb}{0.5,0.5,0.5}
\definecolor{codepurple}{rgb}{0.58,0,0.82}
\definecolor{backcolour}{rgb}{0.95,0.95,0.92}

\lstdefinestyle{mystyle}{
  backgroundcolor=\color{white},   commentstyle=\color{codegreen},
  keywordstyle=\color{magenta},
  numberstyle=\tiny\color{codegray},
  stringstyle=\color{codepurple},
  basicstyle=\ttfamily\footnotesize,
  breakatwhitespace=false,         
  breaklines=true,                 
  captionpos=b,                    
  keepspaces=true,                 
  numbers=left,                    
  numbersep=5pt,                  
  showspaces=false,                
  showstringspaces=false,
  showtabs=false,                  
  tabsize=2
}
\usepackage[colorinlistoftodos]{todonotes}
\usepackage{thmtools}
\usepackage{thm-restate}
\theoremstyle{plain}
\newtheorem{thm}{Theorem}[section]
\newtheorem{lemma}[thm]{Lemma}

\newtheorem{prop}[thm]{Proposition}
\newtheorem{cor}[thm]{Corollary}

\newtheorem{claim*}{Claim}

\numberwithin{equation}{section}

\theoremstyle{definition}

\newtheorem{rmk}[thm]{Remark}

\newtheorem{example}[thm]{Example}

\newtheorem{ass}[thm]{Assumption}
\newtheorem{mydef}[thm]{Definition}

\makeatletter
\def\th@plain{%
  \thm@notefont{}
  \itshape 
}
\def\th@definition{%
  \thm@notefont{}
  \normalfont 
}
\makeatother

\DeclareMathOperator{\Is}{Is}

\newcommand{\ordnop}{\mathop{\mathrm{ord}}\nolimits}

\DeclareMathOperator{\st}{st}

\DeclareMathOperator{\tors}{tors}
\newcommand{\C}{\mathbb{C}}
\newcommand{\G}{\mathbb{G}}
\newcommand{\N}{\mathbb{N}}
\newcommand{\Q}{\mathbb{Q}}
\newcommand{\R}{\mathbb{R}}
\newcommand{\Z}{\mathbb{Z}}
\newcommand{\F}{\mathbb{F}}

\newcommand{\Zp}{\mathbb{Z}_{p}}

\DeclareMathOperator{\Tr}{Tr}

\renewcommand{\div}{\operatorname{div}}
\newcommand{\Log}{\mathop{\mathrm{Log}}\nolimits}

\newcommand{\A}{\mathbb{A}}
\newcommand{\PP}{\mathbb{P}}

\DeclareMathOperator{\Res}{Res}

\newcommand{\dr}{\textup{dR}}
\newcommand{\MW}{\textup{MW}}

\newcommand{\OO}{\mathcal{O}}

\newcommand{\cyc}{\textup{cyc}}

\newcommand{\Div}{\operatorname{Div}}

\newcommand{\Spec}{\operatorname{Spec}}

\renewcommand{\div}{\operatorname{div}}

\newcommand{\inv}{\textup{inv}}
\DeclareMathOperator{\Der}{Der}
\DeclareMathOperator{\Frac}{Frac}
\DeclareMathOperator{\Supp}{Supp}
\DeclareMathOperator{\dR}{\operatorname{dR}}
\DeclareMathOperator{\cris}{\operatorname{cris}}
\DeclareMathOperator{\naive}{\textup{naive}}
\DeclareMathOperator{\Disc}{\operatorname{Disc}}

\author{Francesca Bianchi}
\title{$p$-adic sigma functions and heights on Jacobians of genus $2$ curves}
\date{\today}
\begin{document}
\maketitle

\begin{abstract}
Let $C$ be a genus $2$ hyperelliptic curve over a number field $K$, with a Weierstrass point $\infty$ at infinity, let $J$ be its Jacobian, let $\Theta$ be the theta divisor with respect to $\infty$, and let $p$ be any prime number. We give an explicit construction of a $p$-adic height $h_p\colon J(\overline{\Q})\to \Q_p$ by means of $p$-adic analogues of N\'eron functions of divisor $2\Theta$. We define such N\'eron functions using division polynomials and a generalisation of Blakestad's $p$-adic sigma function on the formal group of $J$. We prove that our $p$-adic N\'eron function $\lambda_v$ at a non-archimedean place $v$ of $K$ is the image, under a suitable trace map, of a symmetric $v$-adic Green function of divisor $\Theta$ \`a la Colmez. We use this to relate $\lambda_v$ and $h_p$ to local and global extended Coleman--Gross (and hence Nekov\'a\v{r}) $p$-adic height pairings. We provide examples of our implementation, including one for a prime $p$ greater than $10^6$, and explain how similar techniques can be used to compute $p$-adic integrals of differentials of the first, second and third kind on $C$ independently of the reduction type. As an application, we also give an explicit quadratic Chabauty function vanishing on the rational points on certain genus $4$ bihyperelliptic curves.
\end{abstract}
\tableofcontents
\section{Introduction}
Let $C$ be a genus $2$ curve over a number field $K$ and let $J$ be its Jacobian. We assume that $C$ is given by an affine equation of the form $y^2 = f(x)$, where $f(x)$ is a monic polynomial of degree $5$ with coefficients in the ring of integers $\OO$ of $K$ and no repeated roots.  Let $\iota\colon C\xhookrightarrow{} J$ be the embedding with respect to the unique point at infinity of $C$, and let $\Theta$ be the corresponding theta divisor of $J$. Fix a (any) rational prime $p$ and a continuous idele class character $\chi\colon \A_K^{\times}/K^{\times}\to \Q_p$. With respect to this data, we explicitly construct, at every non-archimedean place $v$ of $K$, $p$-adic analogues of real-valued N\'eron functions with respect to $2\Theta$ (\S \ref{subsec:local_Neron}). 
 Let 
\begin{equation*}
\lambda_v\colon J(K_v)\setminus\Supp(\Theta)\to \Q_p
\end{equation*}
be such a $p$-adic N\'eron function at $v$. We obtain a quadratic form $h_p\colon J(K)\to \Q_p$ by extending quadratically to $J(K)$ the following function on $ J(K)\setminus\Supp(\Theta)$:
\begin{equation*}
h_p(P) =\frac{1}{[K:\Q]} \sum_v n_v\lambda_v(P), 
\end{equation*}
where $n_v = [K_v:\Q_{\ell}]$ if $v\mid \ell$; we call $h_p$ a global $p$-adic height.

To define $\lambda_v$, we use the explicit embedding of $J$ into $\PP^8$ due to Grant \cite{Grant1990} and genus $2$ hyperelliptic division polynomials of Kanayama \cite{Kanayama, Kanayama_corrections} and Uchida \cite{Uchida}, a generalisation of the standard division polynomials on elliptic curves (see \S \ref{subsec:jac}). 
If $\chi$ is unramified at $v$, i.e.\ the local component $\chi_v$ is identically zero on $\OO_v^{\times}$ \footnote{Note that the places of ramification of $\chi$ are a subset of the places of $K$ above $p$.}, this is all that is needed for the definition of the N\'eron function at $v$ (cf.\ Definition \ref{def:Neron_fct_away_p}). Moreover, up to a multiplicative constant, this is the same as the real-valued N\'eron function at $v$, as explicitly constructed by Uchida \cite{Uchidacanloc}.

At the places $v$ at which $\chi$ is ramified, the $p$-adic N\'eron function $\lambda_v$ depends on the choice of a subspace $W_v$ of $H^1_{\dR}(C/K_v)$ that is complementary to the space of holomorphic forms and isotropic with respect to the cup product pairing. In particular, we attach to $W_v$ a unique $v$-adic analogue $\sigma_v$ of the complex hyperelliptic sigma function for $C$ (see Section \ref{sec:sigma_functions}). This extends work of Blakestad. Indeed, suppose that $K_v$ is an unramified extension of $\Q_p$, that $v$ is of good reduction for $C$, ordinary reduction for $J$ and that $p\geq 5$. Then the $v$-adic sigma function attached to the unit root subspace of the Frobenius endomorphism of $H^1_{\dR}(C/K_v)$ is the ``canonical'' $v$-adic sigma function of \cite{blakestadsthesis} (see Theorem \ref{thm:Blakestad_main}, Remark \ref{rmk:inver_H1} and Proposition \ref{prop:Blakestad_space_unit_root}). In fact, Blakestad's thesis was the motivation for the work presented in this paper: for elliptic curves, the Mazur--Tate $p$-adic height  \cite{mazur-tate} can be described explicitly using $v$-adic analogues of the Weierstrass sigma function \cite{padicsigma,MST} (see \S \ref{subsec:ell_curves_analogues}). Our initial goal was to extend this to genus $2$. We will show in forthcoming work with Kaya and M\"uller \cite{BKM22} that our explicit definition of $h_p$ in the canonical case indeed recovers the canonical Mazur--Tate $p$-adic height for $J$.

 The complex sigma function is a holomorphic function $\C^2\to \C$, which depends on an analytic isomorphism $J(\C)\cong \C^2/\Lambda$ (\S \ref{subsec:jac}). In general, no uniformisation theorem is at our disposal in the $v$-adic setting; however, the kernel of reduction of $J$ modulo $v$, which we denote by $J_1(K_v)$, is the group associated to a commutative formal group law $F(T_1,T_2)$ of dimension $2$ over $\OO_v$. Upon base-changing to $K_v$, the formal group law $F$ becomes isomorphic to the additive group $\G_a^2$ (see \S \ref{subsec:formalgps_gen},\ref{subsec:formal}).
 
 Using a suitable isomorphism, we can mimic the complex theory and define the sigma function $\sigma_v$ as the solution of a  system of differential equations on the formal group\footnote{This will not be surprising to the reader familiar with $v$-adic sigma functions on elliptic curves (\S \ref{subsec:ell_curves_analogues}): see \cite{Perrin-Riou, bernardi, padicsigma}.}. 
All this can be made explicit thanks to Grant's description of formal group parameters $T_1$ and $T_2$ \cite{Grant1990}, and the formal power series $\sigma_v(T)\colonequals \sigma_v(T_1,T_2)$ converges on a finite index subgroup $H_v$ of $J_1(K_v)$, and hence of $J(K_v)$ (Proposition \ref{prop:finite_index_subgp}).  When it exists, Blakestad's is the unique $v$-adic sigma function that is given as a power series in $T_1$ and $T_2$ with $v$-adically integral coefficients.  

The connection between the complex sigma function and real-valued N\'eron functions was proved by Yoshitomi \cite{Yoshitomi} (see also \cite{Uchida} for a generalisation to arbitrary genus). By construction, the $v$-adic sigma function $\sigma_v$ satisfies similar properties to the complex one; therefore, it can be used to define a $p$-adic analogue of a N\'eron function: for $P\in J(K_v)\setminus \Supp(\Theta)$, we let
\begin{equation*}
n_v\lambda_v(P) = -\frac{2}{ m^2}\cdot \chi_{v}\left(\frac{\sigma_v(T(mP))}{\phi_m(P)}\right),
\end{equation*}
where $m=k[J(K_v):H_v]$ for some $k\in\{1,2\}$ and  $\phi_m$ is the $m$-th division polynomial (see Definition \ref{def:Neron_fct_above_p}).

 A natural question that arises is how the global height $h_p$ and the $p$-adic N\'eron functions $\lambda_v$ compare to well-known global and local $p$-adic height theories that apply to Jacobians of curves in general. We explore this question for the $p$-adic height pairing of Coleman--Gross \cite{ColemanGross} (and its extension to bad reduction due to Colmez \cite{Colm96} and Besser \cite{Besser_pArakelov, BesserPairing}); by \cite{Besser:CG_and_Nekovar, BesserPairing}, this is equal to the $p$-adic height of Nekov\'a\v{r} \cite{Nekovar} if the curve has semistable reduction at every prime above $p$. 
 
 The extended\footnote{i.e.\ ``extended to bad reduction''.} Coleman--Gross height for $C$ is defined in a crucially different way from our $h_p$. In particular, the local part of the theory is given by pairings on degree $0$ divisors on $C$ with disjoint support, and it is defined using Coleman--Colmez--Vologodsky integration on $C$ at the places of ramification of $\chi$ and arithmetic intersection theory at the other places.

Just for this introduction, we denote by $\langle D_1,D_2\rangle_v$ the extended Coleman--Gross local height pairing at $v$, evaluated on the pair of $K_v$-rational degree $0$ divisors $D_1,D_2$ with disjoint support.  As is the case for $h_p$, the local pairings $\langle D_1,D_2\rangle_v$ at the primes $v$ of ramification for $\chi$ depend on a choice of isotropic subspace of $H_{\dR}^1(C/K_v)$, complementary to $H^0(C/K_v,\Omega^1)$.
 
 We prove that, provided that we make the same choices of subspaces at the primes of ramification, the global height $h_p$ is the same as the extended Coleman--Gross height for $C$ (Corollary \ref{cor:global_CG_same_as_this}) and that the local Coleman--Gross height pairing at any $v$ is equal to a suitable linear combination of pullbacks of $\lambda_v$ to $C$: see Corollary \ref{cor:lambda_eq_CG} for a precise statement. For the reader's convenience, we state here a simplified version of the local and global comparison:
 \begin{thm}\label{thm:intro_1}
Let $P_1,P_2,Q_1,Q_2\in C(K_v)$ such that $Q_i\neq P_j$ for all $i,j\in\{1,2\}$. Then there exists a finite set of points $S\subset C(K_v)$, such that for all $R\in C(K_v)\setminus S$, we have:
 \begin{equation*}
\langle P_1-P_2,Q_1-Q_2\rangle_v = -\frac{n_v }{2}\sum_{1\leq i,j\leq 2} (-1)^{i+j}\lambda_v(\iota(Q_i)  -\iota(P_j)-\iota(R)). 
 \end{equation*}
 Moreover,
 \begin{equation*}
 \sum_{v}\langle P_1 - P_2, Q_1 - Q_2 \rangle_v =\frac{1}{2}( h_p([P_1+Q_1-P_2-Q_2])-h_p([P_1-P_2])-h_p([Q_1-Q_2])).
 \end{equation*}
 \end{thm}
 \begin{rmk}
 This is essentially a $p$-adic analogue of the real-valued theorem of Faltings \cite{faltings_calculus_arithm_surfaces} and Hriljac \cite{hriljac} (specialised to the setting of a genus $2$ odd degree hyperelliptic curve).
 \end{rmk}
 In order to prove these comparison results, we take a detour into Colmez's theory of $p$-adic integration on curves and symmetric $p$-adic Green functions of theta divisors on their Jacobians \cite{Colm96}, which we apply to $C$ and to the divisor $\Theta$ on $J$. The link to the above problem is the following. In the original paper of Coleman--Gross \cite{ColemanGross}, at a place $v$ of ramification for $\chi$, the local height paring $\langle D_1,D_2\rangle_v$ is given by
 \begin{equation*}
\langle D_1,D_2\rangle_v =  t_v\biggl(\int_{D_2}\omega_{D_1}\biggr),
 \end{equation*}
 where the integral is a $K_v$-valued (Coleman \cite{Col82, coleman, ColemandeShalit}) integral of a differential on $C$, and $t_v\colon K_v\to \Q_p$ is a trace map. 
The trace map is uniquely determined by $\chi$, and so is the branch of the $p$-adic logarithm that the integral depends on (for certain choices of divisors).

The dependency of $v$-adic integration on the branch of the $p$-adic logarithm is resolved in Colmez's theory (which, unlike Coleman's, makes no assumption on the reduction) in a different manner: by viewing the value $\Log{p}$ of the logarithm at $p$ as a variable. In fact, the Colmez integral of $\omega_{D_1}$ takes values in $\mathscr{L}(K_v)\colonequals K_v\oplus 	\Q\Log{p}$.  With such a convention, Colmez gives a unified theory of (extended) Coleman--Gross local height pairings: in terms of his integration theory, for any $v$ (including the unramified primes), we have
 \begin{equation}\label{eq:D1D2v_colmez}
\langle D_1,D_2\rangle_v = T_v\biggl(\int_{D_2}\omega_{D_1}\biggr),
 \end{equation}
 for some $\Q$-linear map $T_v\colon \mathscr{L}(K_v)\to \Q_p$ uniquely determined by $\chi$. At an unramified prime $v\mid q$, this trace map $T_v$ satisfies $T_v(a+b\Log{q}) = b\chi_v(q)$. In other words, the intersection multiplicity appearing in Coleman--Gross's definition of $\langle D_1,D_2\rangle_v$ at such a prime is replaced by the coefficient of $\Log{q}$ of a $v$-adic integral, constructed analogously to the one appearing at the primes of ramification for $\chi$.  
 
Moreover, Colmez constructs $v$-adic analogues of Green functions on abelian varieties. He shows that the integral appearing in \eqref{eq:D1D2v_colmez} can be expressed in terms of pullbacks to $C$ of a suitable symmetric\footnote{In our case, ``symmetric'' is a synonym of ``even'' (cf.\ \S\ref{subsec:comparison})} $v$-adic Green function $G_{\Theta}$ associated to the divisor $\Theta$ on $J$. The choice of $G_{\Theta}$ at a ramified prime $v$ depends on $W_v$. We prove (see Theorem \ref{thm:comparison} for a more precise statement): 
 \begin{thm}
There exists a symmetric $v$-adic Green function $G_{\Theta}$ of divisor $\Theta$ such that
 \begin{equation*}
 n_v\lambda_{v} = -2T_{v}(G_{\Theta}).
 \end{equation*}
 Moreover, if $\chi$ is ramified at $v$ and $\lambda_v$ corresponds to the choice of $W_v\subset H^1_{\dR}(C/K_v)$, so does $G_{\Theta}$.
\end{thm} 
  From this, Theorem \ref{thm:intro_1} and the more general Corollaries \ref{cor:lambda_eq_CG} and \ref{cor:global_CG_same_as_this} follow easily. 

We now turn to the computational aspects of the paper (Section \ref{sec:implementation}). Our \texttt{SageMath} \cite{sage} implementation is available at \cite{github_padic_g2}. In recent years, algorithms for $p$-adic heights have attracted considerable interest due to their crucial role in explicit versions of the Chabauty--Kim method \cite{KimP1, Kimunipotent} for determining rational points on curves: see especially the groundbreaking \cite{BDQCI} and \cite{SplitCartan} on quadratic Chabauty. We therefore believe that an implementation of the local $p$-adic N\'eron functions could have some applications beyond the ones that we already present in this paper.

Given a subspace $W_v$, computing the $v$-adic sigma function associated to it requires working explicitly with the formal group law of $J$. In particular, one needs to compute expansions of functions in the formal group parameters and a formal group isomorphism of the base-change of $J_1$ to $K$ (or $K_v$) with the formal additive group of dimension $2$. All this can be achieved essentially by implementing the explicit results of \cite{Grant1990}.

For instance, we can compute a ``naive'' $v$-adic sigma function (\S \ref{subsec:naive_sigma}), a genus $2$ analogue of Bernardi's sigma function for elliptic curves \cite{bernardi}. This sigma function is quite useful for computational purposes, because its coefficients as a series in $T_1,T_2$ belong to $K$ and are independent of the prime $v$, which allows us to consider large primes. 

However, if $J$ has good ordinary reduction at all places $v$ of ramification of $\chi$, there are instances where one might want to instead compute Blakestad's canonical $v$-adic sigma function. 
Indeed, as mentioned above, we prove that Blakestad's $v$-adic sigma function corresponds to the unit root subspace of Frobenius, and by the comparison result with the extended Coleman--Gross height (Corollary \ref{cor:global_CG_same_as_this}), the global height $h_p$ with respect to this choice is related to a $p$-adic analogue of the Birch and Swinnerton-Dyer conjecture \cite{BaMuSt12}. In addition, in light of its integrality, Blakestad's $v$-adic sigma function has better convergence properties than any other and this might make it a preferable choice in various computational situations. In \S \ref{subsec:implementation_canonical_sigma} we discuss how we can compute it when $K_v\cong \Q_p$ using Kedlaya's algorithm \cite{kedlaya}.

The final ingredients for the computation of $p$-adic N\'eron functions are division polynomials, continuous idele class characters and the index $[J(K_v):J_1(K_v)]$. For division polynomials, we improve an implementation of de Jong--M\"uller \cite{muller_de_jong} (\S \ref{subsec:division_polynomials}); for the characters we could use \cite{QCnfs} to address the general case (but have currently restricted the implementation to the cyclotomic character), and for the index we use work of Bruin--Stoll \cite{Bruin-Stoll:MWSieve} (see \S \ref{subsec:implementation_Neron_fcts} for details). 

To illustrate our implementation, we present two examples. In \S \ref{subsec:eg_large_p}, we consider the curve $C\colon y^2 = x^5-1$ over $\Q$  and we compute the canonical cyclotomic $p$-adic height of a non-torsion point in $J(\Q)$, for the good ordinary prime $p= 10^6 + 81$. What allows us to consider such a large prime and yet compute the canonical $p$-adic height (to $p$-adic precision 20, in approximately 15 seconds) 
 is that, in this case, we can find the unit root subspace without appealing to Kedlaya's algorithm, as a consequence of the fact that $C$ has extra automorphisms over $\Q(\zeta_5)$.

In \S \ref{subsec:example_hts_Neron_fcts}, we consider a curve for which some computations of canonical Coleman--Gross heights were carried out in \cite{BBM0}. We use our comparison results to replicate these computations using $p$-adic N\'eron functions in place of Coleman integrals and intersection multiplicities. 

The two examples that we have chosen to present are of curves over $\Q$ that have good ordinary reduction at $p$. The good and ordinary assumptions can both (or either one) be removed in our implementation if we replace the unit root subspace of Frobenius with any explicitly given isotropic complementary subspace. In forthcoming work with Kaya and M\"uller \cite{BKM22} we address the problem of computing a canonical sigma function also in the semistable ordinary case. Finally, the assumption that the curve is defined over $\Q$ is also not essential for the algorithms (however, if our choices require computing unit root eigenspaces of Frobenius using Kedlaya's algorithm, we need to assume that $K_v\cong \Q_p$ at every prime $v$ of ramification of $\chi$, due to current limitations in the \texttt{SageMath} \cite{sage} implementation). 

Compared to other explicit approaches to $p$-adic heights, such as the algorithms of \cite{Bes-Bal10, BaMuSt12} for Coleman--Gross heights, ours has the clear disadvantage that it assumes that $C$ is of genus $2$. On the other hand, we believe there are features that might make our approach preferable in the genus $2$ case. For example, its essential insensitivity to the reduction type at the primes $v\mid p$. 
This is inherited from the N\'eron functions definition. Namely, the N\'eron function at $v$ is first defined on the model-dependent kernel of reduction $J_1(K_v)$, and then extended to $J(K_v)$ by combining the facts that $[J(K_v):J_1(K_v)]$ is finite and that $\lambda_v$ is ``almost'' a quadratic function (Proposition \ref{prop:properties_neron_fcts}).

 By Theorem \ref{thm:intro_1} (and, more generally, Corollary \ref{cor:lambda_eq_CG}), we can then compute local Coleman--Gross heights for arbitrary reduction. In \S \ref{subsec:diff_first_2nd_third} we explain how, by replacing quasi-quadraticity with a suitable transformation property under multiplication on $J$, we may also express integrals of differentials of the first and second kind on $J$ (and hence on $C$, by pullback) in terms of integrals on $J_1(K_v)$. In summary, the same or similar techniques to the ones for heights can be applied to compute integrals of differentials of the first, second and third kind on $C$, without any assumptions on the reduction. 

The final section (Section \ref{sec:application_bihyper}) is devoted to an elementary quadratic Chabauty-type criterion for the rational points on certain genus $4$ hyperelliptic curves that admit degree two maps to two genus $2$ curves (as $C$), each with a rank $2$ Jacobian. See Proposition \ref{prop:QC_bihyper}.

\subsection{Elliptic curve analogues}\label{subsec:ell_curves_analogues}
Essentially all of the results of this paper admit an elliptic curve analogue, which either already appears in the literature or could be deduced in a similar way to the genus $2$ case.  The construction of the genus $2$ naive sigma function and the analysis of its convergence properties that we present in \S \ref{subsec:naive_sigma} are based on the work of Bernardi for elliptic curves \cite{bernardi}; Blakestad's sigma function (\S \ref{subsec:infty_sigma}) is an analogue of the canonical sigma function of Mazur--Tate \cite{padicsigma}. That the elliptic curve canonical sigma function for $E/K_v$ is related to the unit root eigenspace of $H^1_{\dR}(E/K_v)$ follows from \cite[Proposition 3]{padicsigma} and \cite{Katz, Katzinterpolation}
and is used in the algorithms of Mazur--Stein--Tate \cite[\S 3.2]{MST} and Harvey \cite[Sections 3,4]{harvey}. 

$p$-Adic heights on elliptic curves constructed using $p$-adic sigma functions are again subject of \cite{bernardi, padicsigma, MST,harvey} and they are $p$-adic heights in the sense of \cite{mazur-tate}. See also \cite{Perrin-Riou-comptes, Perrin-Riou, wuthrichheights}. Typically, such sigma functions are used to  define a global $p$-adic height. The point of view of decomposing this as a sum of local $p$-adic N\'eron functions with respect to the divisor given by twice the point at infinity is often not pursued. Perhaps this is because the natural domain of a $v$-adic sigma function (for $v\mid p$) is a subgroup of the formal group. Therefore, if we are only interested in values of the global $p$-adic height, it is more convenient to extend the latter using quadraticity, rather than extending each local N\'eron function using formulae involving division polynomials. 

However, the point of view of $p$-adic N\'eron functions is natural because it is a direct analogue of the real-valued setting (for which, see for instance \cite[Chapter VI]{silvermanadvancedtopics}), and it is convenient in the context of the Chabauty--Kim method. It is thus taken in \cite{nonabelianconjecture} at the primes away from $p$, and in \cite[\S\S 2.A, 4.A]{Bianchi20} at all primes. 

Direct analogues of the local comparison results of \S \ref{subsec:comparison} are,  to our knowledge, not explicitly stated in the literature, but they could be derived in an analogous manner. Some remarks are however in order. A relation between Colmez Green functions and the work of Mazur--Tate \cite{padicsigma} is hinted by Colmez himself before \cite[Proposition II.2.20]{Colm96}. The equality of the global Coleman--Gross and Mazur--Tate height is known, at least in the good ordinary reduction case \cite{Coleman:universal, Besser:CG_and_Nekovar} (see also the introduction to \cite{ColemanGrosspAdicSigma}). A different kind of local comparison in the good ordinary case is provided by \cite[Corollary 4.2]{ColemanGrosspAdicSigma}.

\subsection{Notation} $\N$ includes $0$. Given a number field $K$ and a non-archimedean place $v$, the notation $\ordnop_v$ is used for the $v$-adic valuation on $K_v^{\times}$, normalised so that $\ordnop_v(K_v^{\times}) = \Z$. The $v$-adic absolute value $|\cdot |_v$ on $K_v$ extends the standard $\ell$-adic absolute value, if $v$ lies above $\ell$.  Various ``logarithms'' appear in the paper. In order to help the reader overcome potential confusion, we compiled a list here:
\begin{itemize}
\item We use the symbol $\log$ in logarithmic derivatives. That is, if $f$ is a power series over some field and $D$ is a differential operator, by $D(\log(f))$ we mean $\frac{Df}{f}$. We also use $\log$ for the real logarithm.
\item $\Log$ and $\Log_v$ are logarithms in the sense of Colmez (Section \ref{sec:Colmez}, Equation \eqref{eq:Log_colmez} and \S\ref{subsubsec:notation_colm_heights}). 
\item $\log_p$ is the branch of the $p$-adic logarithm vanishing at $p$ (\S\ref{subsec:implementation_Neron_fcts}).
\item $\mathcal{L} = (\mathcal{L}_1,\dots, \mathcal{L}_n)$ is a formal group logarithm and, usually, the strict one (Definition \ref{def:strict_log}).
\item A \emph{one-logarithm} of an abelian variety $X$ over a $p$-adic field $K$ is a locally analytic group homomorphism $X(K)\to K$ (\S \ref{subsec:Colm_int}).
\item Given a smooth, projective, geometrically irreducible curve $X$ over a $p$-adic field $K$ with Jacobian $J$, the logarithm $\Log_{\tilde{J}}$ is the canonical homomorphism $T(K)/T_{\ell}(K)\to H^1_{\dR}(X)$ from differentials of the third kind on $X$ modulo logarithmic differentials to the first de Rham cohomology of $X$ (Definition \ref{def:LogJtilde}, Remark \ref{rmk:logJtilde}).
\end{itemize}

\subsection*{Acknowledgements}
First and foremost, I am extremely grateful to Steffen M\"uller for suggesting that Blakestad's work on $p$-adic sigma functions could be used in the context of heights, as well as for helpful discussions and generous feedback at various stages of this project. I would also like to thank Clifford Blakestad for sharing insights about his work,  Stevan Gajovi\'c for coming up with Lemma \ref{lemma:stevan} and for helpful comments, Enis Kaya for useful feedback on an earlier draft and for inspiring conversations on $p$-adic heights, and Remco Wouts for assistance with the computations. This work was supported by an NWO Vidi grant.

\section{Preliminaries}
\subsection{Formal group laws}\label{subsec:formalgps_gen}
In this subsection we recall some definitions and properties of formal group laws. We are mainly interested in an explicit description of the isomorphisms of certain formal group laws with the formal additive group of a suitable dimension: see Theorem \ref{thm:log_iso_Ga}, Proposition \ref{prop:exp_exp_log} and Lemma \ref{lemma:log_and_inv_derivations} below. The main references are \cite{Freije, zink, cassels_flynn}. In particular, up to and including Definition \ref{def:strict_log} we follow \cite{Freije} closely; Proposition \ref{prop:exp_exp_log} is then a stronger version of results in \cite{cassels_flynn}; finally the discussion on invariant derivations is based on \cite{zink}.

Let $R$ be a commutative ring with identity, let $n\in \Z_{\geq 1}$, let $X = (x_1,\dots, x_n)$ and let $R[[X]]$ be the power series ring in $x_1,\dots,x_n$ over $R$. 

\begin{mydef}
A \emph{formal group law} $F(X,Y)$ of dimension $n$ over $R$ is an $n$-tuple of power series $F_i(X,Y)\in R[[X,Y]]$ such that
\begin{enumerate}[label = (\roman*)]
\item\label{it:formal_gp_1} $F(X,Y) = X+Y + O(X,Y)^2$, where $O(X,Y)^2$ means terms of total degree at least $2$;
\item\label{it:formal_gp_2} $F(X,F(Y,Z)) = F(F(X,Y),Z)$. 
\end{enumerate}
\end{mydef}
The formal group law $F(X,Y)$ is said to be \emph{commutative} if $F(Y,X) = F(X,Y)$. 

It follows from the formal group axioms that there exists a unique \emph{formal inverse}, that is an $n$-tuple of power series $i(X)\in R[[X]]^n$ such that $F(X,i(X)) = F(i(X),X) = 0$ and $i(X) =  -X +O(X)^2$. Moreover, $F(X,0) = F(0,X) = X$.

\begin{mydef}
Given formal group laws $F(X,Y)$ and $G(X,Y)$ of dimensions $n$ and $m$ over $R$, a \emph{formal group homomorphism} $\alpha\colon F\to G$ is an $m$-tuple of power series $\alpha_i(X)\in XR[[X]]$ such that $\alpha(F(X,Y)) = G(\alpha(X),\alpha(Y))$. The homomorphism $\alpha$ is an \emph{isomorphism} if there exists a formal group homomorphism $\beta\colon G\to F$ such that $\alpha(\beta(Y)) = Y$ and $\beta(\alpha(X)) = X$.
\end{mydef} 
An example of a commutative formal group law of dimension $n$ is the additive group $\G_a^{n}(X,Y) = X+Y$. If $R$ is a $\Q$-algebra, any commutative formal group law of dimension $n$  over $R$ is isomorphic to $\G_a^{n}$:  see Theorem \ref{thm:log_iso_Ga} below. 

In order to describe the isomorphism explicitly, we first recall the notion of invariant differentials on a commutative formal group law $F(X,Y)$ of dimension $n$ over an arbitrary commutative ring $R$ with identity.

We define the $R[[X]]$-module of \emph{differential $1$-forms} as $\Omega = \sum_{i=1}^n R[[X]] dx_i$, and a corresponding total derivative map $d\colon R[[X]]\to \Omega$, $d(f) = \sum_{i=1}^n \frac{\partial f}{\partial x_i} dx_i$. A differential in $\Omega$ is \emph{exact} if it is in the image of $d$. Given $\omega = \sum_{i=1}^n \varphi_i(X) dx_i \in \Omega$, with $\varphi_i(X)\in R[[X]]$, we say that $\omega$  is \emph{(translation) invariant} if
\begin{equation*}
\omega(F(X,T)) \colonequals \sum_{i=1}^n \varphi_i(F(X,T)) dF_i(X,T)  = \omega;
\end{equation*}
here in $dF_i(X,T)$ we consider $T$ fixed and take partial derivatives with respect to $X$.

The definition of the formal group law implies that the matrix $\left(\frac{\partial F_i}{\partial x_j}(0,X)\right) $ is invertible over $R[[X]]$ and we have: 
\begin{lemma}\label{lemma:inv_diff}
The differential $\omega = \sum_{i=1}^n \varphi_i(X)dx_i\in\Omega$ is invariant if and only if 
\begin{equation*}
\omega = (a_1,\dots,a_n)\left(\frac{\partial F_i}{\partial x_j}(0,X)\right)^{-1}\left(\begin{matrix}
           dx_{1} \\
           \vdots \\
           dx_{n}
         \end{matrix}\right)
\end{equation*}
for some $(a_1,\dots, a_n)\in R^n$. In particular, the invariant differentials of $F(X,Y)$ form an $R$-module of rank $n$. 
\end{lemma}

\begin{proof}
See for example \cite[p.\,243]{Freije}.
\end{proof}

\begin{lemma}[{\hspace{1sp}\cite[Proposition 1.3, Lemma 1.4]{honda}}]
\label{lemma:honda}
If $R$ is a $\Q$-algebra and $F$ is a commutative formal group law over $R$, then every invariant differential is exact. 
\end{lemma}

From these two lemmas, one deduces
\begin{thm}\label{thm:log_iso_Ga}
Let $F$ be a commutative formal group law over a $\Q$-algebra $R$ and let $\omega_1,\dots,\omega_n$ be an $R$-basis for the invariant differentials of $F$. Let $\mathcal{L}_i(X)\in R[[X]]$ be the unique power series such that
\begin{equation*}
d\mathcal{L}_i = \omega_i,\qquad \mathcal{L}_i(0) = 0.
\end{equation*}
Then $\mathcal{L} = (\mathcal{L}_1,\dots, \mathcal{L}_n)\colon F\to \mathbb{G}_a^n$ is a formal group isomorphism. Conversely, any isomorphism $F\to \G_a^{n}$ is of this form.
\end{thm}
\begin{proof}
For the first statement see for example \cite[Theorem 1]{Freije}. For the second statement, if $\alpha = (\alpha_1,\dots,\alpha_n)\colon F\to \G_a^n$ is a homomorphism, then $\alpha_i(F(X,Y)) = \alpha_i(X) + \alpha_i(Y)$, thus $d\alpha_i$ is an invariant differential.  By \cite[Lemma 1.4]{zink}, $\alpha$ is an isomorphism if and only if the Jacobian matrix $ \big(\frac{\partial \alpha_i}{\partial x_j}(0)\big)$ is invertible over $R$. By Lemma \ref{lemma:inv_diff} this happens if and only if $d\alpha_1,\dots,d\alpha_n$ is a basis for the $R$-module of invariant differentials.
\end{proof}

\begin{mydef}\label{def:strict_log}
The isomorphism of Theorem \ref{thm:log_iso_Ga} is called a \emph{logarithm} of $F$. The \emph{strict logarithm} is the logarithm corresponding to the basis $\omega_1,\dots,\omega_n$ where, for $k\in\{1,\dots,n\}$,
\begin{equation*}
\omega_k = (\delta_{ik})\left(\frac{\partial F_i}{\partial x_j}(0,X)\right)^{-1}\left(\begin{matrix}
           dx_{1} \\
           \vdots \\
           dx_{n}
         \end{matrix}\right).
\end{equation*}
Here $\delta_{ik}$ is the Kronecker delta. 
The inverse of a formal group logarithm is called a \emph{formal group exponential}; the inverse of the strict logarithm is the \emph{strict exponential}.
\end{mydef}
By Theorem \ref{thm:log_iso_Ga}, if $R$ is an integral domain of characteristic $0$ and $F$ is a commutative formal group law over $R$, we can consider an isomorphism to the additive group of a suitable dimension upon base-changing to the fraction field of $R$. For our intended applications, we would like to understand what denominators can occur in the series expansions of the strict formal logarithm and exponential.

\begin{prop}\label{prop:exp_exp_log}
Let $R$ be a GCD domain of characteristic $0$ and let $K$ be its field of fractions. Let $F$ be a commutative formal group law over $R$ and let $F_K$ be its base-change to $K$. Then the strict logarithm $\mathcal{L} = (\mathcal{L}_1,\dots, \mathcal{L}_n)$ and strict exponential $\mathcal{E} = (\mathcal{E}_1,\dots,\mathcal{E}_n)$ of $F_K$ are of the form: 
\begin{alignat*}{2}
\mathcal{L}_i &= x_i + \sum_{\substack{j_1,\dots, j_n\in \N\\
j_1+\cdots + j_n \geq 2}} \frac{a_{j_1,\dots, j_n}}{\gcd(j_1,\dots, j_n)} x_1^{j_1}\cdots x_n^{j_n}, \qquad &&\text{where } a_{j_1,\dots, j_n}\in R;\\
\mathcal{E}_i &=x_i + \sum_{\substack{j_1,\dots, j_n\in \N\\
j_1+\cdots + j_n \geq 2}} \frac{b_{j_1,\dots, j_n}}{j_1!\cdots j_n!} x_1^{j_1}\cdots x_n^{j_n}, \qquad &&\text{where } b_{j_1,\dots, j_n}\in R.
\end{alignat*}
\end{prop}

\begin{proof}
Write $\mathcal{L}_i = \sum c_{j_1,\dots, j_n} x_1^{j_1}\cdots x_n^{j_n}$ for $c_{j_1,\dots, j_n}\in K$. For every $k\in \{1,\dots, n\}$, we have
$\frac{\partial \mathcal{L}_i}{\partial x_k}\in R[[X]]$, so
\begin{equation*}
j_k c_{j_1,\dots, j_n} \in R.
\end{equation*} 
If $c_{j_1,\dots,j_n}$ is non-zero, writing $c_{j_1,\dots, j_n} = \frac{\alpha}{\beta}$ for coprime $\alpha,\beta\in R$, we deduce that $\beta\mid j_k$ for every $k$ (by Euclid's lemma for GCD domains) and so $\beta\mid \gcd(j_1,\dots,j_n)$.  The claim on the leading term follows from the fact that $\mathcal{L}$ is the strict logarithm. 
For the coefficients of $\mathcal{E}_i$, see \cite[p.\,68]{cassels_flynn}.
\end{proof}
Let $F$ be a commutative formal group law over a $\Q$-algebra $R$. Theorem \ref{thm:log_iso_Ga} characterises any formal group isomorphism $\mathcal{L}\colon F\to \G_a^n$ by the choice of a basis for the $R$-module of invariant differentials of $F$.  We will now see that the inverse of $\mathcal{L}$ admits an explicit description in terms of a suitable choice of basis of a related $R$-module: that of invariant derivations of $F$. 
A \emph{derivation} is an $R$-linear map $D\colon R[[X]]\to R[[X]]$ such that $D(fg) = gDf + fDg$, and a derivation $D$ is defined to be \emph{invariant} for $F$ if
\begin{equation*}
D(f(F(X,T))) = (Df)(F(X,T))\qquad \text{for all } f \in R[[X]].
\end{equation*}
 It follows from the definition that any derivation $D$ satisfies
\begin{equation}\label{eq:formula_D}
D = \sum_{i=1}^n Dx_i \cdot \frac{\partial}{\partial x_i};
\end{equation}
thus $D$ is invariant if and only if, for every $j\in \{1,\dots,n\}$, we have
\begin{equation*}
\sum_{i=1}^n Dx_i\cdot \frac{\partial F_j(X,T)}{\partial x_i} = (Dx_j)(F(X,T)).
\end{equation*}
By \cite[Theorem 1.15]{zink}, the invariant derivations form an $R$-module of rank $n$. Moreover, denoting by $\Der R[[X]]^{\inv}$ the $R$-module of invariant derivations, and by $\Omega^{\inv}$ the $R$-module of invariant differentials, by \cite[Theorem 1.19]{zink} there is a perfect pairing
\begin{equation}\label{eq:dual_inv_diff_der}
(\cdot, \cdot)\colon \Der R[[X]]^{\inv}\times \Omega^{\inv}\to R, \qquad (D,gdf) = gDf. 
\end{equation}
In particular, if $\mathcal{L}$ is the strict logarithm, the dual to $d\mathcal{L}_i$ is
\begin{equation}\label{eq:Di}
D_i = \sum_{j=1}^n \frac{\partial F_j}{\partial x_i}(0,X)\frac{\partial}{\partial x_j}.
\end{equation}
\begin{lemma}
\label{lemma:log_and_inv_derivations}
Let $R$ be a $\Q$-algebra, $F$ a commutative formal group law over $R$. Let $\mathcal{L}=(\mathcal{L}_1,\dots, \mathcal{L}_n)$ be a formal logarithm $F\to \G_a^n$, with inverse $\mathcal{E}\colon \G_a^n\to F$. Suppose $\mathcal{L}$ corresponds to the choice of basis for $\Omega^{\inv}$ dual to the basis $(D_1,\dots,D_n)$ for $\Der R[[X]]^{\inv}$. Then
\begin{equation*}
\frac{\partial f(\mathcal{E}(T))}{\partial t_i} = (D_i f)(\mathcal{E}(T))\qquad \text{for all } f\in R[[X]].
\end{equation*}
\end{lemma}
\begin{proof}
Let $\gamma^{(i)}(t_i) = \mathcal{E}(T)|_{t_j = 0, j\neq i}$. Then $\gamma^{(i)}$ defines a formal group homomorphism $\G_a\to F$. Therefore, by \cite[Theorem 1.26]{zink}, the following formula defines an invariant derivation
\begin{equation}\label{eq:Ditilde_formula}
(\tilde{D}_i x_j)(X)= \frac{\partial F_j(X,\gamma^{(i)}(t_i))}{\partial t_i}\bigg\vert_{t_i =0}= \sum_{\ell=1}^n\frac{\partial F_j}{\partial y_{\ell}}(X,0) \frac{\partial \mathcal{E}_{\ell}}{\partial t_i}(T)\bigg\vert_{T = 0} = \frac{\partial \mathcal{E}_j}{\partial x_i}(X)\bigg\vert_{X = \mathcal{L}(X)}
\end{equation}
(the last equality follows from differentiating $\mathcal{E}(X+Y) = F(\mathcal{E}(X),\mathcal{E}(Y))$ with respect to $Y$ and evaluating at $Y=0$ and $X=\mathcal{L}(X)$). Therefore, for every $k\in \{1,\dots,n\}$,
\begin{equation*}
(\tilde{D}_i, \omega_k) = \sum_{j=1}^n \frac{\partial \mathcal{L}_k}{\partial x_j}\tilde{D}_i x_j=\frac{\partial \mathcal{L}_k(\mathcal{E}(X))}{\partial x_i}\bigg\vert_{ X = \mathcal{L}(X)} = \delta_{ik},
\end{equation*}
since $\mathcal{L}(\mathcal{E}(X)) = \mathcal{E}(\mathcal{L}(X))= X$. This shows that $\tilde{D}_i = D_i$. The lemma follows from \eqref{eq:Ditilde_formula} and \eqref{eq:formula_D}. 
\end{proof}
\begin{rmk}
With the notation of Lemma \ref{lemma:log_and_inv_derivations} and its proof, we have, in particular, that $\gamma^{(i)}(t_i)$ is the integral curve of $D_i$ (cf. \cite[Theorems 1.23, 1.26]{zink}).
\end{rmk}

\subsection{Jacobians of genus 2 curves: algebraically, analytically and projectively}\label{subsec:jac}
Let $C$ be a hyperelliptic curve of genus $2$ over a field $K$ of characteristic different from $2$, described explicitly by the closure in the weighted projective space $\PP_{1,3,1}$ of a smooth affine curve
\begin{equation}
\label{eq:Grant}
y^2 = x^5 + b_1x^4 + b_2x^3 + b_3 x^2 + b_4 x + b_5, \qquad b_i\in K.
\end{equation}
If $K$ is either $\C$ or an algebraically closed field, any smooth, projective, geometrically irreducible curve of genus $2$ is birationally equivalent over $K$ to one in this form; for other fields $K$, 
this is not the case; see for example \cite[Chapter 1]{cassels_flynn}.

There is a unique point on $C$ that does not belong to the provided affine patch; we will denote this point by $\infty$. Given a point $P$ we denote by $P^{-}$ the image of $P$ under the hyperelliptic involution, which maps $(x,y)$ to $(x,-y)$ and $\infty$ to itself. 
The (algebraic) Jacobian $J$ of $C$ is an abelian variety of dimension $2$ over $K$ whose $K$-rational points are $K$-rational classes of degree zero divisors on $C$ modulo linear equivalence. Every such class can be represented by a $K$-rational divisor of the form
\begin{equation*}
P_1 + P_2 - 2\infty, \qquad \text{for some} \quad P_1,P_2\in C.
\end{equation*}
In fact, such a divisor is unique for every point in $J(K)\setminus \{0\}$, whereas the zero class is represented by 
\begin{equation*}
P + P^{-}  - 2\infty, \qquad \text{for every} \quad P\in C.
\end{equation*}
Therefore, there is a surjection from the symmetric square of $C$, denoted $C^{(2)}$, to $J$, and we can identify $J$ with $C^{(2)}$ blown down at the origin (see for instance \cite[IIIa, \S 2]{MumfordTata2} for the case $K=\C$). 
 We also consider the $\Theta$ divisor of $J$, whose support is given by
\begin{equation*}
\Supp(\Theta) = \{[P-\infty]:P\in C\}.
\end{equation*} 

While we are ultimately interested in working with Jacobians of curves over finite extensions of $\Q_p$, for a prime $p$, it is instructive and useful for multiple reasons to first review some of the complex analytic theory. So assume now that $K=\C$. Our main references are \cite{Grant1990, Uchida, Baker_multiply_period, MumfordTata1, MumfordTata2}.
A basis for the space $H^0(C/\C,\Omega^1)$ of holomorphic $1$-forms on $C$ is given by
\begin{equation}\label{eq:omega_12}
\omega_1 = \frac{dx}{2y},\qquad \omega_2 = \frac{xdx}{2y};
\end{equation}
using the Hodge filtration, we identify this with a $2$-dimensional subspace of the first algebraic de Rham cohomology of $C/\C$, which we denote by $H^1_{\dR}(C/\C)$. Under this identification, we have
\begin{equation*}
H^1_{\dR}(C/\C) \cong H^0(C/\C,\Omega^1) \oplus W_0,
\end{equation*}
where $W_0$ is the space spanned by the classes of the differentials of the second kind
\begin{equation}\label{eq:compl_sub_complex}
\eta_1 = (-3x^3 - 2b_1x^2 - b_2 x)\frac{dx}{2y}, \qquad \eta_2 =- \frac{x^2 dx}{2y}.
\end{equation}
The space $W_0$ is isotropic with respect to the algebraic cup product pairing; moreover, $[\eta_1]$ and $[\eta_2]$ is a dual basis to $[\omega_1],[\omega_2]$, in the sense that $[\eta_i]\cup [\omega_j] = \delta_{ij}$, where $\delta_{ij}$ is the Kronecker delta, so $[\omega_1], [\omega_2], [-\eta_1],[-\eta_2]$ is a symplectic basis for $H^1_{\dR}(C/\C)$ with respect to the cup product pairing. 

Analogously, we also pick a symplectic basis $A_1,A_2,B_1,B_2$ for $H_1(C,\Z)$ with respect to the intersection product of cycles. It is explained in \cite[IIIa, \S 5]{MumfordTata2} how to do this using the fact that $C$ is a double cover of $\PP^1$.

We then define the period matrices 
\begin{equation*}
\Omega =(\Omega_{ij}) =  \biggr(\int_{A_j} \omega_i\biggl), \qquad \Omega^{\prime} =(\Omega^{\prime}_{ij}) =  \biggr(\int_{B_j} \omega_i\biggl).
\end{equation*} 
The matrix $\tau\colonequals \Omega^{-1}\Omega^{\prime}$ is well-defined, symmetric and its imaginary part is positive definite; i.e.\ $\tau$ belongs to the Siegel upper half space of dimension $2$. Let $\Lambda = \Omega \Z^2  + \Omega^{\prime} \Z^2$. The analytic Jacobian of $C$ is $\C^2/\Lambda$.  The isomorphism between the algebraic and analytic Jacobian is induced by the integration map
\begin{equation}\label{eq:Phi}
\Phi\colon C^{(2)} \to \C^2/\Lambda, \qquad (P_1,P_2)\mapsto \int_{\infty}^{P_1}(\omega_1,\omega_2) + \int_{\infty}^{P_2}(\omega_1,\omega_2)  \bmod{\Lambda}.
\end{equation}
To $\tau$ as above and any $a,b\in \Q^2$, we can attach a so-called \emph{theta function with characteristic}:
\begin{equation*}
\theta\begin{bmatrix}
a\\
b
\end{bmatrix}(z) = \sum_{n\in\Z^2} \exp(\pi i (n+a)^{T}\tau(n+a) + 2\pi i (n+a)^T(z+b)), \qquad z\in \C^2.
\end{equation*}
By \cite[II, Proposition 1.1 and p.\,123]{MumfordTata1}, this is a holomorphic function on $\C^2$. Moreover, it is quasi-periodic with respect to $\Omega^{-1}\Lambda$ in the sense of \cite[p.\,123]{MumfordTata1}: for every $m\in \Z^2$, 
\begin{align}\label{al:quasi_period1}
\theta\begin{bmatrix}
a\\
b
\end{bmatrix}(z+m) &= \exp(2\pi i a^{T} m)\cdot \theta\begin{bmatrix}
a\\
b
\end{bmatrix}(z)\\
\theta\begin{bmatrix}
a\\
b
\end{bmatrix}(z+\tau m) &= \exp(-2\pi i b^{T} m -\pi i m^T \tau m - 2\pi i m^T z)\cdot \theta\begin{bmatrix}
a\\
b
\end{bmatrix}(z). \label{al:quasi_period2}
\end{align}
From the quasi-periodicity \eqref{al:quasi_period1}-\eqref{al:quasi_period2} it is immediate that, writing $z=(z_1,z_2)$, for every $i,j\in\{1, 2\}$, the following is a meromorphic function on $\C^2/\Lambda$:
\begin{equation}
\frac{\partial^2}{\partial z_i \partial z_j}\log\left(\theta\begin{bmatrix}
a\\
b
\end{bmatrix}(\Omega^{-1}z)\right).
\end{equation}
Other meromorphic functions on $\C^2/\Lambda$ arise in the same way if we replace $\theta\begin{bmatrix}
a\\
b
\end{bmatrix}(z)$ with itself multiplied by the exponential of a polynomial in $z_1,z_2$ of degree at most $2$. Amongst such choices of polynomials and vectors $a,b\in \Q^2$ there are some that are particularly useful. 

First, when
\begin{equation*}
a = \delta \colonequals (1/2,1/2)\qquad \text{and} \qquad b =\delta^{\prime}\colonequals (1,1/2), 
\end{equation*}
the theta function is odd \cite[II, Proposition 3.14]{MumfordTata1} and vanishes to order $1$ precisely on those $z\in \C^2$ such that $\Omega z \bmod{\Lambda} = \Phi(P,\infty)$ for some $P\in C$, that is, on the pullback of the (analytic) theta divisor under $\C^2 \to \C^2/\Omega^{-1}\Lambda \xrightarrow{\sim}\C^2/\Lambda$. See \cite[3.80-3.85, 3.89]{MumfordTata2}.

Secondly, if we set
\begin{equation*}
\tilde{\sigma}(z) = \exp\left(\frac{1}{2}z^T H\Omega^{-1}z\right)\theta\begin{bmatrix}
\delta\\
\delta^{\prime}
\end{bmatrix}(\Omega^{-1} z),\qquad \text{where}\quad  H = (H_{ij}) = \biggl(\int_{A_j} \eta_i\biggr),
\end{equation*}
then the second logarithmic derivatives of $\tilde{\sigma}$, viewed as functions on $C^{(2)}$, can be described as quotients of elements in $\Z[b_1,\dots, b_5][x_1,x_2,y_1,y_2]$; see \cite[p.\,38]{Baker_multiply_period}.

Finally, we define the \emph{hyperelliptic sigma function} of $C$ to be
\begin{equation}\label{eq:sigma}
\sigma(z) = c\tilde{\sigma}(z),
\end{equation}
where $c\in \C^{\times}$ is such that the Taylor expansion of $\sigma(z)$ around $0$ is of the form
\begin{equation*}
\sigma(z) = z_1 + O(z_1,z_2)^3.
\end{equation*}
With such a normalisation, the Taylor expansion has coefficients in $\Q[b_i]$ (see e.g.\ \cite[Proposition 2.1]{Uchida}).

 For $i,j,\dots,k\in \{1,2\}$, let 
\begin{equation}
\label{eq:diff_equation_cx_sigma}
\wp_{ij\cdots k}(z) = -\frac{\partial}{\partial z_i}\frac{\partial}{\partial z_j}\cdots\frac{\partial}{\partial z_k}\log(\sigma(z)) \qquad \text{and}\qquad \wp = \wp_{11}\wp_{22} - \wp_{12}^2.
\end{equation}
It follows from above that $\wp_{ij\cdots k}$ and $\wp$ are meromorphic functions on $\C^2/\Lambda$. On the image of $\Phi$ these are expressible as rational functions dependent on the coefficients of $C$, and these definitions apply more generally when $\C$ is replaced by any field $K$ of characteristic different from $2$.  For example, if $z = \Phi((x_1,y_1),(x_2,y_2))$, then
\begin{align}
\label{eq:p12_p22}
\wp_{12}(z) = -x_1x_2,\qquad \wp_{22}(z) = x_1+x_2.
\end{align}
See \cite[(1.4)]{Grant1990} for the other $\wp_{ij}$ and for $\wp_{ijk}$. 

Following Kanayama \cite{Kanayama, Kanayama_corrections} and Uchida \cite{Uchida}, we also introduce the following functions on the Jacobian: for $m\geq 1$, the \emph{$m$-th division polynomial} is
\begin{equation}
\label{eq:def_div_poly}
\phi_m(z) = \frac{\sigma(mz)}{\sigma(z)^{m^2}}.
\end{equation}
By \cite[Theorem 5.8, Example 5.9]{Uchida}, $\phi_m(z)$ is a polynomial in $\wp_{ij}$ and $\wp_{ijk}$ ($1\leq i,j,k\leq 2$) with coefficients in $\Z\left[\frac{1}{2},b_1,\dots, b_5\right]$ and, conjecturally, in $\Z[b_1,\dots, b_5]$ \cite[Conjecture 4.14]{Uchida}. In particular, also division polynomials make sense over arbitrary fields of characteristic different from $2$, and given a curve over such a field, we will view $\phi_m$ as a function on its Jacobian $J$, satisfying
\begin{equation}
\label{eq:divisor_div_poly}
\div(\phi_m) = [m]^{*}\Theta - m^2\Theta.
\end{equation}

Other properties of division polynomials that we need are (cf.\ \cite[Proposition 4.9, Example 2.16]{Uchida}):
\begin{align}
\phi_{mn}(z) &= \phi_m(nz)\phi_n(z)^{m^2} \label{eq:div_quad}\\
\frac{\phi_m(u+v)\phi_m(u-v)}{\phi_m(u)^2\phi_m(v)^2} &= \frac{-\wp_{11}(mu) + \wp_{11}(mv) - \wp_{12}(mu)\wp_{22}(mu) + \wp_{22}(mu)\wp_{12}(mv)}{(-\wp_{11}(u) + \wp_{11}(v) - \wp_{12}(u)\wp_{22}(v) + \wp_{22}(u)\wp_{12}(v))^{m^2}}. \label{eq:div_par}
\end{align}

The analytic theory that we have outlined is related to our goal of defining $p$-adic heights in at least two ways. First, suppose that $K$ is a number field and let $v$ be an archimedean place of $K$. Then the (unique up to an additive constant) \emph{real-valued} N\'eron function at the place $v$, associated with $2\Theta$, admits an explicit formula in terms of the hyperelliptic sigma function $\sigma$: see \cite[Corollary 2.5]{Yoshitomi}. Moreover, such a N\'eron function satisfies transformation properties under scalar multiplication and addition on the Jacobian that can be described using the division polynomials and the functions $\wp_{ij}$, respectively \cite[Theorem 7.5]{Uchida}. In Section \ref{sec:sigma_functions}, extending work of Blakestad \cite{blakestadsthesis},
we construct $p$-adic analogues of the genus $2$ complex sigma function, and in Section \ref{sec:padic_hts} we use these to define $p$-adic N\'eron functions.

\begin{rmk}
Hyperelliptic sigma functions, division polynomials and their applications to real-valued N\'eron functions all admit generalisations to hyperelliptic curves of arbitrary genus: see \cite{Uchida}. For elliptic curves, see \cite[Chapter VI]{silvermanadvancedtopics}.
\end{rmk}

The second application of the analytic theory that we are interested in is an explicit set of  equations for the Jacobian as a variety in $\PP^8$. This is worked out in \cite{Grant1990}: for any field $K$ of characteristic different from $2$, \cite[Corollary 2.15]{Grant1990} provides
a set of $13$ homogeneous equations\footnote{\label{note:typoGrant}There is a typo in $f_{10}$: the first term should be $X_{112}^2$, rather than $X_{122}^2$.} in $K[X_0,X_{11}, X_{12}, X_{22}, X_{111}, X_{112}, X_{122}, X_{222}, X]$, defining in $\mathbb{P}^8$ the Jacobian of a curve described by an equation of the form \eqref{eq:Grant} with coefficients in the field $K$. Under this embedding, the identity of $J$ is mapped to
\begin{equation*}
O = (0:0:0:0:1:0:0:0:0).
\end{equation*}
In \cite[Theorem 3.3, \S 4]{Grant1990} Grant also gives explicit formulae for the group law on $J$ with respect to this embedding in $\PP^8$. 

Note that the projective coordinates $X_0, X_{ij}, X_{ijk}, X$ are related to the functions $\wp_{ij},\wp_{ijk},\wp$ on $C^{(2)}$ by
\begin{align}
\label{eq:X_and_p}
\frac{X_{ij}}{X_0} = \wp_{ij},\qquad \frac{X_{ijk}}{X_0} = \frac{1}{2}\wp_{ijk}, \qquad \frac{X}{X_0} = \frac{1}{2}(\wp +b_2\wp_{12} - b_4);
\end{align}
for notational convenience from now on we dehomogenise with respect to $X_0$, i.e.\ we set
\begin{equation*}
X_0 \colonequals 1. 
\end{equation*}

We end this subsection with a remark about invariant differentials and derivations on $J$. The complex analytic isomorphism $J\xrightarrow{\sim}\C^2/\Lambda$ is defined using the basis \eqref{eq:omega_12} for the space of holomorphic differentials on $C$ and the embedding $\iota\colon C\xhookrightarrow{} J$ mapping $\infty$ to the identity of $J$. 

The differentials $\omega_1,\omega_2$ of \eqref{eq:omega_12} are a basis for $H^0(C/K,\Omega^1)$ (that is, not only when $K = \C$), and the embedding $\iota$ induces an isomorphism $\iota^{*}\colon H^0(J/K,\Omega^1)\to H^0(C/K,\Omega^1)$ between the space of holomorphic differentials on $J$ and on $C$. Holomorphic differentials on $J$ are translation-invariant, and they correspond, by duality, to translation-invariant derivations.

 \begin{mydef}\label{def:basis_inv_diff_inv_der}
 For each $i\in \{1,2\}$, let $\Omega_i$ be the invariant differential on $J$ satisfying 
 \begin{equation*}
 \iota^{*} \Omega_i = \omega_i,
 \end{equation*}
and let $\partial_i$ be the invariant derivation dual to $\Omega_i$.
 \end{mydef}

\pagebreak
 \begin{lemma}\label{lemma:inv_difder}\leavevmode
\begin{enumerate}
\item \label{lemma_part:inv_difder_1} The basis $\{\Omega_1,\Omega_2\}$ for $H^0(J/K,\Omega^1)$ of Definition \ref{def:basis_inv_diff_inv_der} is given explicitly by
 \begin{align*}
 \Omega_1 = \frac{1}{2(X_{111}X_{122} - X_{112}^2)} (X_{122}\cdot dX_{11} - X_{112} \cdot dX_{12});\\
 \Omega_2 = \frac{1}{2(X_{111}X_{122} - X_{112}^2)} (X_{111}\cdot dX_{12} - X_{112}\cdot dX_{11}).
 \end{align*}
\item \label{lemma_part:inv_difder_2} The dual invariant derivations $\partial_1$ and $\partial_2$ satisfy
 \begin{equation*}
 \partial_1(X_{ij}) = 2X_{1ij},\qquad \partial_2(X_{ij}) = 2X_{ij2}.
 \end{equation*}
 \end{enumerate}
 \end{lemma}
 \begin{proof}
Part \eqref{lemma_part:inv_difder_2} is proved in \cite[p.\,64]{blakestadsthesis}, Part \eqref{lemma_part:inv_difder_1} then follows from this and an explicit computation. 
 \end{proof}

\subsection{Grant's formal group law} \label{subsec:formal}
We now apply the general theory of formal group laws presented in \S \ref{subsec:formalgps_gen} to the setting of \S \ref{subsec:jac}. In particular, we recall Grant's description of a pair of local parameters at the origin of a Jacobian as in \S \ref{subsec:jac} (over a suitable field) and the induced formal group law. Moreover, we relate the corresponding strict logarithm to the basis of invariant differentials of Definition \ref{def:basis_inv_diff_inv_der}. Finally, we consider a group associated to the formal group law.

Let $R$ be either the ring of integers of a number field, or a ring of characteristic $0$, complete with respect to a non-archimedean valuation. In both cases, let $K$ be the fraction field of $R$. As in \S \ref{subsec:jac}, we consider a hyperelliptic curve $C$ over $K$ defined by an equation of the form \eqref{eq:Grant}, but we further assume that $b_1,\dots, b_5\in R$. We denote by $J$ the projective variety in $\mathbb{P}^8$ defined in \S \ref{subsec:jac}, and we consider the following functions 
\begin{equation}\label{eq:local_par}
T_1 = -\frac{X_{11}}{X_{111}}, \qquad T_2 = -\frac{X}{X_{111}}.
\end{equation}
\begin{thm}[{\hspace{1sp}\cite[Theorem 4.2]{Grant1990}}]\label{thm:exp_Xij_Xijk}
The functions $\frac{1}{X_{111}}=\frac{X_0}{X_{111}}, \frac{X_{ij}}{X_{111}}, \frac{X_{ijk}}{X_{111}}$ can be expanded as formal power series in $T_1,T_2$ with coefficients in $\Z[b_1,\dots, b_5]$. 
\end{thm}
In particular, we have
\begin{equation} \label{eq:Xij_exp}
\begin{aligned}
\frac{1}{X_{111}} &= \sum_{i,j}\alpha_{ij}T_1^iT_2^j = T_1^3\biggl(-1+\sum_{\substack{i\geq 3\\
i+j>3}}\alpha_{ij}T_1^{i-3}T_2^j\biggr)\\
\frac{X_{22}}{X_{111}} &=\sum_{i,j}\beta_{ij}T_1^iT_2^j =  T_1\biggl(-2T_1T_2 + \sum_{\substack{i\geq 1\\
i+j>3}} \beta_{ij}T_1^{i-1}T_2^j\biggr)\\
\frac{X_{12}}{X_{111}} &=\sum_{i,j}\gamma_{ij}T_1^iT_2^j = T_1\biggl(T_2^2 +\sum_{\substack{i\geq 1\\
i+j>3}}\gamma_{ij} T_1^{i-1}T_2^j\biggr), 
\end{aligned}
\end{equation}
where $\alpha_{ij},\beta_{ij},\gamma_{ij}\in \Z[b_1,\dots,b_5]$ can be found recursively as follows. Suppose we have computed $\alpha_{rs},\beta_{rs},\gamma_{rs}$ for all $r,s$ such that $r+s<m$. Then we can compute $\alpha_{ij}, \beta_{ij}, \gamma_{ij}$ for all $i,j$ such that $i+j = m$ using the defining equations of the projective variety $J$. A similar recursion allows one to recover the following expansions:
\begin{equation}\label{eq:Xijk_exp}
\begin{aligned}
\frac{X_{112}}{X_{111}} &=\sum_{i,j}\delta_{ij}T_1^i  T_2^j =  -T_2^2 +\sum_{i+j\geq 4}\delta_{ij}T_1^iT_2^j\\
\frac{X_{122}}{X_{111}} &=\sum_{i,j}\epsilon_{ij}T_1^i  T_2^j = T_1T_2+\sum_{i+j\geq 4}\epsilon_{ij}T_1^i T_2^j\\
\frac{X_{222}}{X_{111}} &= \sum_{i,j}\zeta_{ij}T_1^i  T_2^j =-T_1^2+\sum_{i+j\geq 4}\zeta_{ij}T_1^iT_2^j,
\end{aligned}
\end{equation}
for some $\delta_{ij},\epsilon_{ij},\zeta_{ij}\in \Z[b_1,\dots, b_5]$.

\begin{rmk}\label{rmk:odd_and_even_expansions}
The functions 
\begin{equation*}
T_1 = -\frac{X_{11}}{X_{111}}, \quad\frac{1}{X_{111}}, \quad  \frac{X_{22}}{X_{111}},\quad \frac{X_{12}}{X_{111}}, \quad T_2 = -\frac{X}{X_{111}}
\end{equation*}
are all odd functions on $J$, while
\begin{equation*}
\frac{X_{112}}{X_{111}}, \quad \frac{X_{122}}{X_{111}}, \quad \frac{X_{222}}{X_{111}}
\end{equation*}
are even. It follows that
\begin{align*}
&\alpha_{ij} = \beta_{ij} = \gamma_{ij} = 0\quad &&\text{for all } i+j\in 2\Z;\\
&\delta_{ij} = \epsilon_{ij} = \zeta_{ij} = 0 \quad &&\text{for all } i+j\in 1+2\Z.
\end{align*}
\end{rmk}

The group law on $J$ induces a two-dimensional formal group law in $T_1,T_2$ over $R$, as follows. Given points $u,v\in J$, \cite[Theorem 3.3]{Grant1990} expresses each of $X_{ij}(u+v), X_{ijk}(u+v), X(u+v)$ as a quotient of polynomials in $\Z[b_1,\dots,b_5][X_{ij}(u), X_{ij}(v), X_{ijk}(u), X_{ijk}(v)]$. 

Let $T = (T_1,T_2)$ and $S = (S_1,S_2)$ be the local parameters \eqref{eq:local_par} for $u$ and $v$, respectively. Then one can express each of
\begin{equation*}
F_1(u,v) = -\frac{X_{11}(u+v)}{X_{111}(u+v)}, \qquad F_2(u,v) = -\frac{X(u+v)}{X_{111}(u+v)}.
\end{equation*}
as elements in $\Frac(\Z[b_1,\dots, b_5][T_1,T_2,S_1,S_2])$. To see this, you first have to multiply each of the numerator and denominators of $F_i(u,v)$ by the same suitable rational function.

\begin{thm}[{\hspace{1sp}\cite[Theorem 4.6]{Grant1990}}]\label{thm:formal_gp_law}
For each $i\in \{1,2\}$, we have
\begin{equation*}
F_i(T,S) = T_i + S_i + O(T,S)^2\in R[[T,S]].
\end{equation*}
Therefore $F(T,S) = (F_1,F_2)(T,S)$ is a commutative formal group law over $R$.
\end{thm}
Since $T_1$ and $T_2$ are odd functions on the Jacobian, the formal group inverse is 
\begin{equation*}
i(T) = -T;
\end{equation*}
therefore, for every $i\in \{1,2\}$, we have that $F_i(-T,-S) = - F_i(T,S)$, and hence the expansion of $F_i$ only contains terms of total odd degree. 
\begin{rmk}
Technically speaking the proof in \cite[Theorem 4.6]{Grant1990} assumes that $R$ is a ring (of characteristic different from $2$), complete under a non-archimedean valuation. This assumption ensures that the Jacobian over $K = \Frac(R)$ is an analytic group, and hence, locally, a formal group over $K$. Then a version of Gauss's lemma is used to prove that the formal group law is actually defined over $R$. The proof still gives us the desired result when $R$ is the ring of integers of a number field, since if the quotient of two power series in $R[[T,S]]$ is in $K_v[[T,S]]$ for \emph{a} non-archimedean place $v$, then it is in $K[[T,S]]$, and if it is also in $R_v[[T,S]]$ for \emph{every} non-archimedean place $v$ of $K$, then it is in $R[[T,S]]$.

It should also be possible to show that the group law is in fact defined over $\Z[b_1,\dots, b_5]$. However, it is not immediately clear how to do this. For example, we could try to re-write $F_i(T,S)$ as $\frac{G}{H}$ where $G$ and $H$ are power series in $(T,S)$ with coefficients in $\Z[b_1,\dots, b_5]$, and $H = 1+ O(T,S)$.

It is worth mentioning, in this respect, the analogous situation for genus $2$ curves described by an equation $y^2 = f_6 x^6 + f_5 x^5 + f_4 x^4+ f_3 x^3 + f_2 x^2 + f_1 x + f_0$, considered by Flynn in \cite{Flynn2, Flynn5}. In this case, in \cite{Flynn2} Flynn describes an explicit embedding of the Jacobian in $\PP^{15}$ (in place of $\PP^8$), as well as a formal group law over a ring that is complete with respect to a non-archimedean valuation. Like Grant, Flynn appeals to the analytic group structure of the Jacobian over the fraction field for the formal group law and to Gauss's lemma. On the other hand, in the subsequent article \cite{Flynn5}, he gives a more direct proof that the group law is in fact defined over $\Z[f_1,\dots, f_6]$ as a corollary of an explicit description of the group law on the Jacobian by biquadratic forms. 
\end{rmk}

Let $F$ be the formal group law over $R$ of Theorem \ref{thm:formal_gp_law}. By the general theory of \S \ref{subsec:formalgps_gen}, its base-change to $K$ is isomorphic to the additive formal group of dimension $2$. In particular, let $\mathcal{L} = (\mathcal{L}_1,\mathcal{L}_2)$ be the strict logarithm with respect to $T_1,T_2$, let $\mathcal{E} = (\mathcal{E}_1,\mathcal{E}_2)$ be its inverse, and let $(D_1,D_2)$ be the basis of invariant derivations dual to $(d\mathcal{L}_1,d\mathcal{L}_2)$ (cf.\ Theorem \ref{thm:log_iso_Ga}, Definition \ref{def:strict_log} and \eqref{eq:dual_inv_diff_der}).

The expansions $\Omega_1(T)$ and $\Omega_2(T)$ in $T_1,T_2$ of the invariant differentials $\Omega_1$ and $\Omega_2$ of Definition \ref{def:basis_inv_diff_inv_der} are also invariant differentials for $F$. 

\begin{lemma}\label{lemma:dLi_eq_Omegai}
For each $i\in \{1,2\}$, we have
\begin{equation*}
d\mathcal{L}_i(T) = \Omega_i(T).
\end{equation*}
\end{lemma}
\begin{proof}
It suffices to show that the invariant derivation dual to $\Omega_i(T)$ is $D_i$. Upon extending $D_i$ to $\Frac{K[[T_1,T_2]]}$ in such a way that it keeps satisfying the Leibniz rule, it is then sufficient, in view of Lemma \ref{lemma:inv_difder}\thinspace{}\eqref{lemma_part:inv_difder_2}, to show that \begin{equation*}
\alpha D_1(X_{12}) + \beta D_2(X_{12})
\end{equation*}
is different from $2X_{112}$ for all $(\alpha,\beta)\neq (1,0)$ and that it is different from $2X_{122}$ for all $(\alpha,\beta)\neq (0,1)$. 
By \eqref{eq:Xij_exp} and \eqref{eq:Xijk_exp}, we have
\begin{align*}
X_{12} = -T_1^{-2}T_2^2 + O(T_1,T_2),\qquad 
X_{112} = T_1^{-3}T_2^{2} + O(1),\qquad
X_{122} =-T_1^{-2}T_2 + O(1);
\end{align*}
for an integer $i$, $O(T_1,T_2)^i$ means terms of total degree $\geq i$, where we are allowing negative valuation: for example, $T_1^{-1}T_2^{3}$ is $O(T_1,T_2)^{2}$.
Moreover, by \eqref{eq:Di},
\begin{equation*}
\left(\begin{matrix}D_1\\
D_2\end{matrix}\right) = \left(\begin{matrix}
1 + O(T_1,T_2)^2 & O(T_1,T_2)^2\\
O(T_1,T_2)^2 &  1 + O(T_1,T_2)^2 
\end{matrix}\right)\left(\begin{matrix}
\frac{\partial}{\partial T_1}\\
\frac{\partial}{\partial T_2}
\end{matrix}\right).
\end{equation*}
Therefore,
\begin{equation*}
D_1 (X_{12}) = 2T_1^{-3}T_2^2 + O(1), \qquad D_2(X_{12}) = -2T_1^{-2}T_2 + O(1),
\end{equation*}
which proves the lemma. 
\end{proof}

So far we have only been concerned with formal group laws. If $R$ is now a characteristic $0$ ring, complete under a non-archimedean absolute value $|\cdot |_v$, with maximal ideal $\mathfrak{m}$, we can consider a group $F(\mathfrak{m})$ associated to the formal group law of Theorem \ref{thm:formal_gp_law}:
\begin{equation*}
F(\mathfrak{m}) = \mathfrak{m}\times \mathfrak{m},\qquad t + s \colonequals F(t,s),\text{ for all } t = (t_1,t_2),s = (s_1,s_2)\in F(\mathfrak{m}).
\end{equation*}
Let $\tilde{J}$ be the variety in $\PP^{8}$ obtained by reducing modulo $\mathfrak{m}$ the coefficients of the equations defining $J$ in $\PP^8$; this may or may not be an abelian variety over the residue field $R/\mathfrak{m}$. Nevertheless, we have a reduction map $\tilde{\ }\colon J(K)\to \tilde{J}(R/\mathfrak{m})$. It follows from Theorem \ref{thm:exp_Xij_Xijk} and the definition of the group law on $F(\mathfrak{m})$ that we have the following isomorphism of groups
\begin{equation}\label{eq:kernel_red}
F(\mathfrak{m}) \cong J_1(K)\colonequals \{u\in J(K)\,:\, \tilde{u} = \tilde{O}\}.
\end{equation}
(see \cite[Corollary 4.5, Theorem 4.6]{Grant1990}). Note that this definition is model-dependent. 

\section{$v$-Adic sigma functions}\label{sec:sigma_functions}
\subsection{The naive sigma function}\label{subsec:naive_sigma}
In this subsection, we assume that our genus $2$ curve $C$ is defined over a number field $K$, with ring of integers $\mathcal{O}$, and that all the coefficients $b_1,\dots,b_5$ of the defining equation \eqref{eq:Grant} for $C$ lie in $\mathcal{O}$.

Let $\rho$ be an embedding of $K$ in $\C$ 
 and let $\sigma(z)$ be the hyperelliptic sigma function on $\C^2$ (in the sense of \S \ref{subsec:jac}, \eqref{eq:sigma}) associated to the curve over $\C$ obtained from $C$ via $\rho$.
In particular, $\sigma(z)$ is an odd function and it satisfies the differential equations \eqref{eq:diff_equation_cx_sigma}. Moreover, it has a Taylor expansion in $z=(z_1,z_2)$ around $0$ with coefficients in $\Q[\rho(b_i)]$.

The algebraicity of the coefficients of the Taylor expansion around $0$ of the \emph{complex} sigma function is what we would like to use to define a \emph{$v$-adic} sigma function for a non-archimedean place $v$, thereby extending to genus $2$ Bernardi's definition of a $v$-adic sigma function on elliptic curves \cite{bernardi}. We will call such a $v$-adic sigma function \emph{naive}, as it is constructed ad hoc from the complex one, and to distinguish it from other $v$-adic sigma functions that we will consider in \S \ref{subsec:infty_sigma}.

By \S \ref{subsec:formal}, we have an explicit description of formal group parameters $T_1,T_2$ and a resulting formal group law $F$ over $\OO$. Consider the base-change to $K$, the strict logarithm $\mathcal{L} = (\mathcal{L}_1,\mathcal{L}_2)$ with respect to $T_1,T_2$ and its inverse $\mathcal{E} = (\mathcal{E}_1,\mathcal{E}_2)$, and let $(D_1,D_2)$ be the basis of invariant derivations dual to $(d\mathcal{L}_1,d\mathcal{L}_2)$ (cf.\ Theorem \ref{thm:log_iso_Ga}, Definition \ref{def:strict_log} and \eqref{eq:dual_inv_diff_der}). We extend $D_i$ to $\Frac{K[[T_1,T_2]]}$ in such a way that it keeps satisfying the Leibniz rule, and for $f(T)\in K[[T]]$, we define $D_i(\log(f(T)))\in\Frac{K[[T_1,T_2]]} $ so that Lemma \ref{lemma:log_and_inv_derivations} holds for $\log(f(T))$. 

 We also consider the base-change of $F$ to $\C$ via $\rho$.  By \cite[\S 10.9.3]{Bost}\footnote{Technically speaking this reference is for elliptic curves, but as stated by the author the results ``actually admit straightforward generalizations concerning commutative algebraic groups over (local) fields of characteristic zero''.}, there is a formal group logarithm $\tilde{\mathcal{L}} \colon F_{\C} \to \G_a^2$ whose inverse $\tilde{\mathcal{E}} \colon \G_a^2\to F_{\C}$ extends to the surjective homomorphism $\C^2\to J$ inducing the inverse of $\Phi\colon J\to \C^2/\Lambda$ (see \eqref{eq:Phi}).  
 
\begin{lemma} \label{lemma:pijF}
The logarithm $\tilde{\mathcal{L}}$ and exponential $\tilde{\mathcal{E}}$ are the base-change to $\C$ of $\mathcal{L}$ and $\mathcal{E}$. Therefore,
\begin{equation*}
X_{ij}(\mathcal{E}(z)) = \wp_{ij}(z), \qquad X_{ijk}(\mathcal{E}(z)) = \frac{\wp_{ijk}}{2}(z),
\end{equation*}
where $\wp_{ij}(z)$ (resp.\ $\wp_{ijk}(z)$) is the quotient of series in $z_1,z_2$ obtained from the Taylor expansion of $\sigma$ and \eqref{eq:diff_equation_cx_sigma}.
\end{lemma}
\begin{proof}
This follows from the definition of $\Phi$ \eqref{eq:Phi} and Lemma \ref{lemma:dLi_eq_Omegai}.
\end{proof}

Let now $v$ be a prime of $K$ above a fixed rational prime $p$, let $K_v$ be the completion of $K$ with respect to $v$, $\OO_v$ its ring of integers with unique maximal ideal $\mathfrak{m}_v$. Let $\ordnop_v$ be the valuation on $K_v$ corresponding to $v$. Recall from the end of \S \ref{subsec:formal} that there is a group isomorphism between $F(\mathfrak{m}_v)$ (which, as a set, is $\mathfrak{m}_v\times \mathfrak{m}_v)$ and the group $J_1(K_v)$ of points reducing to $O$ modulo $\mathfrak{m}_v$. We would like to use this to obtain a $v$-adic sigma function on $J_1(K_v)$, by substituting $\mathcal{L}(T)$ in the Taylor series expansion of $\sigma(z)$.

Similarly to the elliptic curve case \cite{bernardi}, it will not always be possible to obtain a function on the whole of $J_1(K_v)$: 
we need to understand how $\sigma(\mathcal{L}(T))$ converges $v$-adically. We start with an auxiliary lemma.

\begin{lemma}\label{lemma:subs_exponential}
Let $f(T_1,T_2)\in \OO_v[[T_1,T_2]]$. Then $f(\mathcal{E}(z))$ is of the form
\begin{equation*}
f(\mathcal{E}(z)) = \sum_{i_1, i_2\in \N} \frac{a_{i_1,i_2}}{i_1! i_2!} z_1^{i_1}z_2^{i_2}, \qquad \text{where } a_{i_1,i_2}\in \OO_v.
\end{equation*}
In particular, it converges for all $z$ with $\min_i\{\ordnop_v(z_i)\}>\frac{\ordnop_v(p)}{p-1}$. Moreover,  for such $z$ we have
\begin{equation*}
\min_i\{\ordnop_v(\mathcal{E}_i(z))\} = \min_i\{\ordnop_v(z_i)\}.
\end{equation*}
Similarly, if $\min_i\{\ordnop_v(T_i)\}> \frac{\ordnop_v(p)}{p-1}$, then
\begin{equation*}
\min_i\{\ordnop_v(\mathcal{L}_i(T))\} = \min_i\{\ordnop_v(T_i)\}.
\end{equation*}
\end{lemma}

\begin{proof}
The first part follows from Proposition \ref{prop:exp_exp_log} and the fact that if $j_1,j_2\in \N$ are such that $j_1+j_2 = n$ then $\ordnop_v(j_1!j_2!) \leq \ordnop_v(n!)$ by Legendre's formula. For the same reasons, for the second part it suffices to show that if $\min_i\{\ordnop_v(z_i)\}>\frac{\ordnop_v(p)}{p-1}$ and $ n\geq 2$, then 
\begin{equation}\label{eq:vals}
n\min_i\{\ordnop_v(z_i)\}-\ordnop_v(n!) > \min_i\{\ordnop_v(z_i)\}.
\end{equation}
Let $s_p(n)$ be the sum of the base $p$ digits of $n$. Then 
\begin{equation*}
\ordnop_v(n!) = \frac{n-s_p(n)}{p-1}\ordnop_v(p) \leq  \frac{n-1}{p-1}\ordnop_v(p),
\end{equation*}
so \eqref{eq:vals} follows. The proof for $\mathcal{L}_i$ is identical.
\end{proof}
Let $[n]T $ be defined inductively by $[1]T = T$ and for $n\geq 2$, $[n]T= F([n-1]T, T)$ and let $\phi_n(T)$ be the expansion in $T$ of the $n$-th division polynomial (see \eqref{eq:def_div_poly} and surrounding paragraph).
\begin{thm}\label{thm:sigma_naive}
Let $\sigma(z) = z_1 + O(z_1,z_2)^3\in K[[z_1,z_2]]$ be the Taylor expansion around $0$ of the complex sigma function, and let
\begin{equation*}
\sigma_v^{(0)}(T)\colonequals \sigma(\mathcal{L}(T))\in K[[T]].
\end{equation*}
\begin{enumerate}[label=(\roman*)]
\item\label{thm_sigma:part1} $\sigma_v^{(0)}(T)$ satisfies the system of differential equations
\begin{equation*}
D_iD_j \log(\sigma_v^{(0)}(T)) =-X_{ij}(T), \qquad \text{for all } i,j\in \{1,2\} \qquad (\text{where } X_{21}\colonequals X_{12}).
\end{equation*}
\item\label{thm_sigma:part2} $\sigma_v^{(0)}(T)$ converges for all $T=(T_1,T_2)\in K_v^2$ satisfying $\min_{i}\{\ordnop_v(T_i)\} >\frac{\ordnop_v(p)}{p-1}$.
\item\label{thm_sigma:part3} When all terms are defined, we have
\begin{align*}
\frac{\sigma_v^{(0)}(F(T,S))\sigma_v^{(0)}(F(T,-S))}{\sigma_v^{(0)}(T)^2\sigma_v^{(0)}(S)^2}& = -X_{11}(T) + X_{11}(S) - X_{12}(T)X_{22}(S) + X_{22}(T) X_{12}(S);\\
\frac{\sigma_v^{(0)}([n]T)}{\sigma_v^{(0)}(T)^{n^2}} &= \phi_n(T).
\end{align*}
\end{enumerate}
\end{thm}

\begin{proof}
By Lemma \ref{lemma:log_and_inv_derivations}, we have
\begin{equation*}
(D_iD_j \log(\sigma_v^{(0)}(T)))(\mathcal{E}(z)) = \frac{\partial^2}{\partial z_i\partial z_j}\log(\sigma_v^{(0)}(\mathcal{E}(z))) = \frac{\partial^2 }{\partial z_i \partial z_j}\log(\sigma(z)) =- \wp_{ij}(z),
\end{equation*}
which gives \ref{thm_sigma:part1} by Lemma \ref{lemma:pijF}.

Since the divisor of $\sigma$ is the preimage of $\Theta$ under $\C^2\to J$, since $T_1$ represents $\Theta$ in a neighbourhood of the origin \cite[Lemma 32]{blakestadsthesis} and since $\sigma(z) = z_1 + O(z_1,z_2)^3$, we have 
\begin{equation*}\label{eq:sigma_div_E1}
\sigma(z) = \mathcal{E}_1(z)u(z), \qquad\text{for some}\quad u(z) =1 + O(z_1,z_2)\in K[[z_1,z_2]].
\end{equation*}
Combining with \ref{thm_sigma:part1} and the fact that $\sigma$ is an odd function, we get that
\begin{equation*}
\sigma_v^{(0)}(T) = T_1 u_v(T),
\end{equation*}
where $u_v(T) = 1 + O(T_1,T_2)^2\in K[[T_1,T_2]]$ satisfies
\begin{equation*}
D_i D_j \log(u_v(T)) =-X_{ij}(T) - D_iD_j \log(T_1). 
\end{equation*}
Since the left hand side of this equation is a power series, so is the right hand side. Moreover, the latter is easily seen to have coefficients in $\OO$. Evaluating at $\mathcal{E}(z)$ and using Lemma \ref{lemma:subs_exponential}, we see that 
\begin{equation*}
\frac{\partial^2}{\partial z_i \partial z_j}\log(u(z)) = \sum_{i_1 + i_2\geq 0} \frac{a_{i_1,i_2}}{i_1!i_2!} z_1^{i_1}z_2^{i_2}, \qquad \text{for some } a_{i_1,i_2}\in  \OO_v.
\end{equation*}

By properties of integration, the series $\log(u(z))$ has the same form and since the leading term of $u_v(T)$ is $1$, it has vanishing constant term. Thus, if $z_i =\mathcal{L}_i(T)$ with $\min_i\{\ordnop_v(T_i)\}>\frac{\ordnop_v(p)}{p-1}$, then $\ordnop_v(\log(u(z))) \geq \min_i\{\ordnop_v(z_i)\}>\frac{\ordnop_v(p)}{p-1}$, so $\log(u(z))$ is in the domain of the $p$-adic exponential (cf.\ Lemma \ref{lemma:subs_exponential}). Applying \cite[\S 2 Substitution Theorem]{mattuck}, we find that $u(z)$ converges whenever $\min_i\{\ordnop_v(z_i)\}>\frac{\ordnop_v(p)}{p-1}$. Hence $\sigma(z)$ converges there too, which is statement \ref{thm_sigma:part2}. 

The identities stated in part \ref{thm_sigma:part3} follow from the identities satisfied by the complex sigma function (\eqref{eq:def_div_poly} and \cite[Example 2.16]{Uchida}).
\end{proof}

\subsection{Infinitely many sigma functions and Blakestad's canonical one}\label{subsec:infty_sigma}
Retain the notation of \S \ref{subsec:naive_sigma}. For $1\leq i,j\leq 2$, let $c_{ij}\in K_v$ with $c_{12} = c_{21}$ and let $c = (c_{ij})_{i,j}$. Then the system 
\begin{equation*}
D_iD_j(\log(\sigma_v^{(c)}(T))) =-X_{ij}(T) + c_{ij}, \qquad \text{for all } 1\leq i,j\leq 2, 
\end{equation*}
has a unique odd solution $\sigma_v^{(c)}(T)\in K_v[[T]]$ of the form $T_1 (1+ O(T_1,T_2))$: if $c = 0$, the zero matrix, then $\sigma_v^{(0)}  = \sigma_v^{(c)}$ is the naive sigma function that we studied in \S \ref{subsec:naive_sigma}; in general, we have
\begin{equation}\label{eq:formulasigmac}
\sigma_v^{(c)}(T) = \sigma_v^{(0)}(T) \exp\biggl(\frac{1}{2}\sum_{1\leq i,j\leq 2} c_{ij}\mathcal{L}_i(T)\mathcal{L}_j(T)\biggr).
\end{equation}
See Appendix \ref{app:sigma_expansion} for the terms up to total degree at most $8$ of the formal power series $ \sigma_v^{(c)}(T)$.

\begin{prop} \label{prop:finite_index_subgp}\leavevmode
\begin{enumerate}[label=(\roman*)]
\item\label{prop:finite_index_subgp:1} The formal power series $\sigma_v^{(c)}(T)\in K_v[[T_1,T_2]]$ induces a function on a finite index subgroup $H_v$ of $J(K_v)$.
\item\label{prop:finite_index_subgp:2} For $P\in H_v$, $\sigma_v^{(c)}(T(P))$ vanishes if and only if $T_1(P) = 0$, i.e.\ it vanishes only on the theta divisor and there it vanishes to order $1$.
\item\label{prop:finite_index_subgp:3} When all terms are defined, we have 
\begin{align*}
\frac{\sigma_v^{(c)}(F(T,S))\sigma_v^{(c)}(F(T,-S))}{\sigma_v^{(c)}(T)^2\sigma_v^{(c)}(S)^2}& = -X_{11}(T) + X_{11}(S) - X_{12}(T)X_{22}(S) + X_{22}(T) X_{12}(S);\\
\frac{\sigma_v^{(c)}([n]T)}{\sigma_v^{(c)}(T)^{n^2}} &= \phi_n(T).
\end{align*}
\end{enumerate}
\end{prop}
\begin{proof}\hfill
\begin{enumerate}[label=(\roman*)]
\item The same argument as in \cite[Chapter 7, \S 5]{cassels_flynn} (see also \cite[III, \S 6]{mattuck}) shows that $J_1(K_v)$ has finite index in $J(K_v)$ (without any assumption on the reduction), and that for every positive integer $n$, $F(\mathfrak{m}_v^n)$ has finite index in $F(\mathfrak{m}_v)\cong J_1(K_v)$.
  Therefore, it suffices to show that there exists $n\in\Z_{>0}$ such that $\sigma_v^{(c)}(T)$ converges for all $T\in \mathfrak{m}_v^{n}\times \mathfrak{m}_v^{n}$. By Theorem \ref{thm:sigma_naive}, there exists $n_1\in\Z_{>0}$ such that $\sigma_v^{(0)}$ converges on $F(\mathfrak{m}_v^{n_1})$. By Lemma \ref{lemma:subs_exponential}, there exists $n_2\in \Z_{>0}$ such that $ \exp\biggl(\frac{1}{2}\sum_{1\leq i,j\leq 2} c_{ij}\mathcal{L}_i(T)\mathcal{L}_j(T)\biggr)$ converges on $F(\mathfrak{m}_v^{n_2})$. The conclusion follows for $n=\max\{n_1,n_2\}$ by \cite[\S 2 Multiplication Theorem]{mattuck}.
    \item We have $\sigma_v^{(c)}(T) = T_1\cdot \exp(f^{(c)}(T))$ for some formal power series $f^{(c)}(T)$. Moreover, we defined $H_v$ in such a way that $f^{(c)}(T)$ and $\exp(f^{(c)}(T))$ converge on $H_v$. Therefore, the only zeros are at $T_1 = 0$. 
\item This follows from Theorem \ref{thm:sigma_naive}\thinspace{}\ref{thm_sigma:part3}, Equation \eqref{eq:formulasigmac}, and the fact that $\mathcal{L}_i\colon F\to \G_a$ is a homomorphism.

\end{enumerate}
\end{proof}

In other words, for any $2\times 2$ symmetric matrix $c$ over $K_v$ we obtain a $v$-adically valued function on some finite index subgroup of $J(K_v)$ satisfying properties analogous to the complex sigma function. Given such a $c$, we also define the subspace $W^{(c)}$ of $H^1_{\dR}(C/K_v)$ to be the space generated by the classes of the following differentials (compare with \eqref{eq:compl_sub_complex}):
\begin{equation}\label{eq:eta_12_c}
\begin{aligned}
\eta_1^{(c)} = (-3x^3 - 2b_1x^2 -b_2x + c_{12} x + c_{11})\frac{dx}{2y},\\
 \eta_2^{(c)}= (-x^2 + c_{22} x + c_{12})\frac{dx}{2y}.
 \end{aligned}
\end{equation}

\begin{prop}\label{prop:bijection_sigma_is}
The map
\begin{equation*}
c = (c_{ij}) \mapsto W^{(c)}
\end{equation*}
gives a bijection between the set of $2\times 2$ symmetric matrices over $K_v$ and the set $\Is(C/K_v)$ of subspaces of $H^{1}_{\dR}(C/K_v)$ that are complementary to $H^{0}(C/K_v,\Omega^1)$ and isotropic with respect to the cup product. Moreover, $[\omega_1],[\omega_2],[-\eta_1^{(c)}],[-\eta_2^{(c)}]$ is a symplectic basis for $H^1_{\dR}(C/K_v)$.
\end{prop}

\begin{proof}
It is well-known that $\left\{[\omega_1], [\omega_2], \left[\frac{x^i dx}{2y}\right]\colon i =2, 3\right\}$ is a basis for $H_{\dR}^1(C/K_v)$ and that $H^{0}(C/K_v,\Omega^1)\subset H_{\dR}^1(C/K_v)$ is generated  by $[\omega_1]$ and $[\omega_2]$. Therefore, the space generated by $[\eta_1^{(c)}]$ and $[\eta_2^{(c)}]$ is complementary to the space of holomorphic forms and a simple computation of residues at infinity shows that it is isotropic with respect to the cup product, and that $[\omega_i]\cup[-\eta_j^{(c)}] = \delta_{ij}$. On the other hand, if $W\in \Is(C/K_v)$ and $[\xi_{1}],[\xi_{2}]$ is a basis for $W$ such that $[\omega_i]\cup[-\xi_j] = \delta_{ij}$, then we must have
\begin{equation*}
[\xi_{i}] = [\eta_i^{(0)}]  + \sum_{j=1}^2 b_{ij}[\omega_j]
\end{equation*}
for some symmetric matrix  $(b_{ij})$.
\end{proof}

Blakestad \cite{blakestadsthesis} does not consider the infinitely many $v$-adic sigma functions defined by \eqref{eq:formulasigmac}; he considers just one, which is, however, arguably more interesting than the others. Translated to our setting, his results may be phrased as follows. 
\begin{thm}[{\hspace{1sp}\cite[Propositions 34, 27, Corollary 37\protect\footnotemark]{blakestadsthesis}}]\footnotetext{Corollary 37 of \emph{loc.\ cit.}\ contains some typos. The correct formulae for $c_{11}$ and $c_{12}$ are: $c_{11} = 2b_1b_2 - b_1\alpha + b_1^2\beta + 3\delta - 3b_1\gamma + 3b_3$,  $c_{12} = b_2 + \alpha - b_1\beta$. }\label{thm:Blakestad_main}
Let $p\geq 5$ and suppose that $C$ has good reduction at $v\mid p$ and that $J$ has good ordinary reduction at $v$. Then
\begin{enumerate}[label=(\roman*)]
\item \label{thm:Blake_1} There exists a unique $v$-adic sigma function $\sigma_v^{(c)}(T)$ with $v$-adically integral coefficients. Let $b$ be the corresponding symmetric matrix. 
\item\label{thm:Blake_2} The differentials $\eta_i^{(b)}$ are uniquely determined by the property that the expansion $\int \eta_i^{(b)}$ in the local parameter $t = -\frac{x^2}{y}$ at $\infty$ has coefficients in $\OO_v$.
\item \label{thm:Blake_3} The matrix $b$ has coefficients in $\OO_v$ and can be computed modulo $p^n$ from explicit expansions in $t = -\frac{x^2}{y}$ of suitable functions in the Riemann--Roch spaces of $p^n\infty$ and $3p^n\infty$.
\end{enumerate}
\end{thm}

\begin{rmk}\label{rmk:inver_H1}
Some remarks on our phrasing of Blakestad's results are in order. First, Theorem \ref{thm:Blakestad_main}\thinspace{}\ref{thm:Blake_2} and \ref{thm:Blake_3} assume the invertibility over the residue field of a $2\times 2$ matrix $H_1$, defined on \cite[p.\,54]{blakestadsthesis}. Blakestad mentions at the beginning of \S 4.2 of \emph{loc.\ cit.}\ that this condition is related to the ordinarity of $J$. More precisely, we prove in \cite{BKM22} that if the equation for $C$ has semistable reduction, then $J$ is semistable ordinary if and only if $H_1$ is invertible. 

Secondly, Blakestad allows for $p=3$. However, in this case, the construction in \S 3.2.1 needs to be slightly modified, for example by defining $\phi_n$ and $\psi_n$ on p.\,52 for $n\geq 2$, instead of $n\geq 1$. Then the condition on invertibility of $H_1$ could be replaced by invertibility of $H_2$. For simplicity, we stated Theorem \ref{thm:Blakestad_main} for $p\geq 5$.
\end{rmk}

We next want to  show that the subspace $W^{(b)}\subset H^1_{\dR}(C/K_v)$ corresponding to Blakestad's $v$-adic sigma function is the unit root subspace of Frobenius (defined below). We achieve this by comparing the property of Theorem \ref{thm:Blakestad_main}\thinspace{}\ref{thm:Blake_2} with the characterisation of the unit root subspace provided by \cite[Corollary 5.9.6\thinspace{}(1)]{Katz_crystalline}. We have found the article \cite{Bogaart} a very useful reference for both theoretical and explicit results on various cohomological theories and we refer the reader to it for details. 

Retain the assumptions of Theorem \ref{thm:Blakestad_main} and assume that $K_v$ is an unramified extension of $\Q_p$. 
The equation defining $C$ also defines the affine part of a hyperelliptic curve over $\OO_v$, which we will denote by $\mathcal{C}/\OO_v$; denote by $H^1_{\dR}(\mathcal{C}/\OO_v)$ the first cohomology of $\mathcal{C}/\OO_v$. By  \cite{Berthelot74}, there is a canonical isomorphism $H^{1}_{\dR}(\mathcal{C}/\OO_v)\xrightarrow{\sim} H^1_{\cris}(\tilde{\mathcal{C}}/\OO_v)$, where $\tilde{\mathcal{C}}$ denotes the special fibre of $\mathcal{C}/\OO_v$ and $H^1_{\cris}(\tilde{\mathcal{C}}/\OO_v)$ is its first crystalline cohomology. This isomorphism equips $H^1_{\dR}(\mathcal{C}/\OO_v)$ with an $\OO_v$-linear endomorphism of Frobenius. From the isomorphism \cite[Proposition 2.2]{Bogaart}
\begin{equation*}
H^1_{\dR}(\mathcal{C}/\OO_v)\otimes_{\OO_v} K_v\xrightarrow{\sim} H^1_{\dR}(C/K_v),
\end{equation*}
 we also obtain a $K_v$-linear Frobenius endomorphism on the $K_v$-vector space $H^1_{\dR}(C/K_v)$. The \emph{unit root eigenspace of Frobenius} is the slope $0$ subspace of $H^1_{\dR}(C/K_v)$ for the Frobenius action. Since we are assuming that $J$ has ordinary reduction at $v$, this space has dimension equal to the dimension of $J$, that is to $2$ \cite[Theorem 3.1]{yui1978jacobian}.
 
Let $P$ be an arbitrary point in $\mathcal{C}(\OO_v)$. Let $\hat{\mathcal{C}}_P$ be the formal completion of $\mathcal{C}/\OO_v$ along $P$. Then $(\hat{\mathcal{C}}_{P}, P)$ is a pointed Lie variety of dimension $1$ over $\OO_v$ \cite[\S\S  5.1, 5.9]{Katz_crystalline}. Let $t$ be a coordinate for it. The first de Rham cohomology $H^1_{\dR}(\hat{\mathcal{C}}_{P}/\OO_v)$ is the $\OO_v$-module of (closed) differential one-forms in $\OO_v[[t]]dt$ modulo the exact one-forms \cite[\S 5.1]{Katz_crystalline}. Since any closed differential form becomes exact upon base-changing to $K_v$, applying $d$ gives an isomorphism of $\OO_v$-modules:
\begin{equation}\label{eq:iso_HdR}
M_{P}\colonequals \frac{\{f\in  K_v[[t]]: f(0) = 0 \text{ and } df\in \OO_v[[t]]dt\}}{\{f\in \OO_v[[t]]: f(0) = 0\}}\xrightarrow{\sim} H^1_{\dR}(\hat{\mathcal{C}}_{P}/\OO_v).
\end{equation}
(this is a special case of \cite[Lemma 5.1.2]{Katz_crystalline}).  

 By \cite[Corollary 5.9.6]{Katz_crystalline}, the unit root eigenspace of Frobenius is the subspace of $H^1_{\dR}(C/K_v)$ obtained by tensoring with $K_v$ the kernel of the formal-expansion-at-$P$ map
     \begin{equation}\label{eq:formal_expansion}
\beta_P\colon H^1_{\dR}(\mathcal{C}/\OO_v)\to H^1_{\dR}(\hat{\mathcal{C}}_{P}/\OO_v).
   \end{equation}   
   We use the superscript $^{-}$ to indicate $(-1)$-eigenspaces for the action of the hyperelliptic involution. Recall that we can regard $\eta_1^{(b)}$ and $\eta_2^{(b)}$ as elements of $H^1_{\dR}(C/K_v)$ by using the following isomorphism of $K_v$-vector spaces
\begin{equation}\label{eq:H1dR_sioH0}
H^1_{\dR}(C/K_v)\cong H^0(C/K_v,\Omega^1_{C/K_v}(4\infty))^{-};
\end{equation}
see, for instance, \cite[\S 5]{Bogaart}.
\begin{prop}\label{prop:Blakestad_space_unit_root}
Suppose that the assumptions of Theorem \ref{thm:Blakestad_main} are satisfied, and that $K_v$ is an unramified extension of $\Q_p$. Then the differentials $\eta_1^{(b)}$ and $\eta_2^{(b)}$ span the unit root eigenspace of Frobenius in $H^1_{\dR}(C/K_v)$.
\end{prop}

\begin{proof}
By abuse of notation, we write $\infty$ also for the section at infinity of $\mathcal{C}/\OO_v$. 
By \cite[Proposition 3.2]{Bogaart}, the $\OO_v$-module $H^1_{\dR}(\mathcal{C}/\OO_v)$ is a lattice in $H^1_{\dR}(C/K_v)$, which, in view of the assumption that $p\geq 5$ and the genus is $2$, we can identify by \cite[(3.9), Theorem 4.2\thinspace{}(ii)]{Bogaart}, with $H^0(\mathcal{C},\Omega^1_{\mathcal{C}/\OO_v}(4\infty))^{-}\xhookrightarrow{}H^0(C/K_v,\Omega^1_{C/K_v}(4\infty))^{-}$, under \eqref{eq:H1dR_sioH0}. By \cite[Proposition 5.2]{Bogaart}, the $\OO_v$-module $H^0(\mathcal{C},\Omega^1_{\mathcal{C}/\OO_v}(4\infty))^{-}$ is free and spanned by $x^{i}\frac{dx}{2y}$ for $0\leq i\leq 3$. 

The parameter $t=-\frac{x^2}{y}$ is a coordinate for $(\hat{\mathcal{C}}_{\infty},\infty)$. Let $M_{\infty}$ be as in \eqref{eq:iso_HdR} with respect to $t$.
The assumption that $p\geq 5$ also implies that
there is an isomorphism of $\OO_v$-modules
  \begin{equation}
   \begin{aligned}
   M_{\infty,(4\infty)}\colonequals  \frac{\{f = \sum_{n=-3}^{\infty} a_n t^n \in  t^{-3}K_v[[t]]:a_0 = 0 \text{ and } df\in \OO_v((t))dt\}}{\{f = \sum_{n=-3}^{\infty} a_n t^n \in  t^{-3}O_v[[t]]:a_0 = 0\}} \xrightarrow{\sim} M_{\infty}\\
   \sum_{n=-3}^{\infty} a_n t^n\mapsto \sum_{n=0}^{\infty} a_n t^n.
   \end{aligned}
   \end{equation} 
We obtain a map
   \begin{equation}\label{eq:comp_alpha}
  H^1_{\dR}(\mathcal{C}/\OO_v)\xrightarrow{\sim}H^0(\mathcal{C},\Omega^1_{\mathcal{C}/\OO_v}(4\infty))^{-}\xrightarrow{\alpha_{\infty}} M_{\infty,(4\infty)}\xrightarrow{\sim} M_{\infty}\xrightarrow{\sim} H^1_{\dR}(\hat{\mathcal{C}}_{\infty}/\OO_v),
   \end{equation}
  where $\alpha_{\infty}$ is the formal-expansion-in-$t$ map followed by formal integration. The differentials $\eta_1^{(b)}$ and $\eta_2^{(b)}$ belong to the kernel of $\alpha_{\infty}$. Let $\mathcal{C}^{\prime}$ be the affine curve over $\OO_v$ obtained from $\mathcal{C}$ by dehomogenising with respect to $x$. Then the composition  \eqref{eq:comp_alpha} agrees with composing the restriction map $H^1_{\dR}(\mathcal{C}/\OO_v)\to H^1_{\dR}(\mathcal{C}^{\prime}/\OO_v)$ with the expansion in $t$-map $H^1_{\dR}(\mathcal{C}^{\prime}/\OO_v)\to H^1_{\dR}(\hat{\mathcal{C}}_{\infty}/\OO_v)$, and hence with $\beta_{\infty}$. Comparing dimensions, we obtain the proposition. 
\end{proof}

\begin{cor}
Under the assumptions of Proposition \ref{prop:Blakestad_space_unit_root}, let  $P\in C(\OO_v)$ and $t$ be a coordinate for $(\hat{C}_{P},P)$. Then the expansions of the formal integrals $\int \eta_1^{(b)}(t)$  and $\int \eta_2^{(b)}(t)$ have coefficients in $\OO_v$.
\end{cor}

\begin{rmk}
We will refer to the $v$-adic sigma function of Theorem \ref{thm:Blakestad_main} as the \emph{canonical $v$-adic sigma function}. There are two reasons why we care about Proposition \ref{prop:Blakestad_space_unit_root}. The first one is that it makes the computation of the canonical $v$-adic sigma function a lot more efficient (at least when $K_v \cong \Q_p$), since we can compute the unit root subspace using Kedlaya's algorithm: see \S \ref{subsec:implementation_canonical_sigma} for details.  Secondly, in the following section, we will use $v$-adic sigma functions to define a $p$-adic height on $J(\overline{\Q})$. Now, there are other $p$-adic height constructions in the literature that depend on a choice of space $W_v\in \Is(C/K_v)$, at every $v\mid p$. For example, the $p$-adic height of Coleman--Gross \cite{ColemanGross}. In this case, the unit root eigenspace of Frobenius is the ``canonical'' choice of subspace and the corresponding height appears in a $p$-adic analogue of the Birch and Swinnerton-Dyer conjecture \cite{BaMuSt12}. In Section \ref{sec:Colmez}, we compare our $p$-adic height defined in terms of sigma functions with the $p$-adic height of Coleman--Gross. In light of Proposition \ref{prop:Blakestad_space_unit_root}, the comparison restricts to a comparison between the canonical heights. See also Remark \ref{rmk:comparison_canonical}.

On the other hand, there are situations where it is preferable to consider the naive $v$-adic sigma function. Indeed, unlike the canonical $v$-adic sigma function, the naive one does not require any assumption on the reduction of $C$ and $J$ at $v$. Furthermore, it is easier to compute. Finally, its tight link with the complex genus 2 sigma function makes it useful in proofs, as a lot of its properties follow formally from the properties of the complex sigma function.
\end{rmk}

\begin{rmk}\label{rmk:CM_case}
In analogy with the elliptic curve situation \cite{bernardi}, we could ask if the matrix $b$ has algebraic entries, independent of $p$ and $v$, in the special case where $J$ has complex multiplication. See \cite[Propositions 2.1, 2.14]{BannaiKobayashi} for results in this direction. 
Indeed, under the assumptions of \emph{loc.\ cit.}, let $q(z)$ be the unique quadratic form such that $\theta(z) = \sigma(z)e^{2\pi i q(z)}$ is a normalised theta function in the sense of \cite[p.\,87]{Lan82}. 
Then, by \cite[Proposition 2.1]{BannaiKobayashi}, the coefficients of the Taylor expansion of $\theta(z)$ around $0$ are algebraic. But since $\sigma(z)$ also has algebraic coefficients, it follows that $2\pi i q(z) = \frac{1}{2}\sum_{1\leq i,j\leq 2} c_{ij} z_iz_j$ with $c_{12}= c_{21}$ and algebraic $c_{ij}$. 
 Finally by \cite[Proposition 2.14]{BannaiKobayashi}, the coefficients of $\theta(\mathcal{L}(T))$ are integral. See \S \ref{subsec:eg_large_p} for an explicit example. 
\end{rmk}

\section{$p$-adic heights}\label{sec:padic_hts}
We keep the notation of Section \ref{sec:sigma_functions}: $K$ is a number field and the coefficients of the defining equation for $C$ lie in its ring of integers. Fix a prime number $p$ and a continuous idele class character
\begin{equation*}
\chi = \sum_v\chi_v\colon \A_K^{\times}/K^{\times} \to \Q_p.
\end{equation*}
By continuity, $\chi_v$ is identically zero if $v$ is an archimedean place, and $\chi_v$ is identically zero on $\OO_v^{\times}$ if $v$ is a non-archimedean place not dividing $p$. See \cite[\S 2.1]{QCnfs} for more properties and for a discussion on how to construct the finite-dimensional $\Q_p$-vector space of all such $\chi$, for a given $K$. If $L$ is a finite extension  of $K$, we denote by $\chi_L$ the continuous idele class character for $L$ obtained by composing $\chi$ with the idele norm $N_{L/K}\colon \A_L^{\times}\to \A_K^{\times}$. Note that, if $w$ is a place of $L$ above the place $v$ of $K$ and $x\in K_v^{\times}$, then $\chi_{L,w}(x) = [L_w:K_v]\chi_v(x)$. 

Let $J_{\Theta}\colonequals J \setminus \Supp(\Theta)$. The goal of this section is to define a quadratic form 
\begin{equation*}
h_p\colon J(\overline{\Q}) \to \Q_p,
\end{equation*}
dependent on $\chi$. We will do so in the following way. Suppose that $P$ is a point in $J_{\Theta}(L)$, for some finite extension $L$ of $K$. Given a non-archimedean place $w$ of $L$, denote by $n_w$ the degree of the field extension $L_w/\Q_{\ell}$, where $\ell$ is the rational prime below $w$. Then 
\begin{equation*}
h_p(P) = \frac{1}{[L:\Q]}\sum_{w}n_w\lambda_w(P),
\end{equation*}
where the sum runs over all the non-archimedean places $w$ of $L$ and $\lambda_w\colon J_{\Theta}(L_w) \to \Q_p$ is a ``quasi--quadratic'' function, dependent on $\chi_{L,w}$, which we call a $p$-adic N\'eron function at $w$ with respect to $2\Theta$. The reason for this terminology will be explained in \cite{BKM22}; see also Remark \ref{rmk:neron_away_p} below.

Before embarking on defining $p$-adic N\'eron function, we state a useful lemma. See also Appendix \ref{app:stevan} for more general, related, results. 
\begin{lemma}\label{lemma:exists_m}
Let $P\in J(L_w)$ be a non-torsion point and let $H$ be a finite index subgroup of $J(L_w)$. Then there exists a positive integer $m$ such that $mP\in H\setminus \Supp(\Theta)$.
\end{lemma}
\begin{proof}
Since $H$ has finite index in $J(L_w)$ and $P$ is non-torsion, there exists $m^{\prime}$ such that $Q\colonequals m^{\prime}P\in H\setminus\{0\}$. If $Q \in\Supp(\Theta)$, then $Q = [Q_1-\infty]$ for a non-Weierstrass point $Q_1\in C(L_{w})$. Thus, $2Q = [2Q_1-2\infty]\in H\setminus\Supp(\Theta)$.
\end{proof}
If $v$ is a place of $K$ lying above $p$, and $w\mid v$, the $p$-adic N\'eron function $\lambda_w$ will be defined in terms of a $v$-adic sigma function from \S \ref{subsec:infty_sigma}. Recall from Proposition \ref{prop:bijection_sigma_is} that the set of $v$-adic sigma functions is in bijection with the set $\Is(C/K_v)$. We thus fix, once and for all, a choice of subspace $W_v\in \Is(C/K_v)$, at every $v\mid p$. 
\subsection{$p$-adic N\'eron functions}
\label{subsec:local_Neron}
Let $L$ be a finite extension of $K$, let $v$ be a place of $K$ and let $w$ be a place of $L$ above $v$.

Assume first that $v\mid p$. 
Our choice of subspace $W_v$ induces a choice of $v$-adic sigma function $\sigma_{v}$. By Proposition \ref{prop:finite_index_subgp}, $\sigma_v$ converges on a subgroup $H_v$, dependent on $W_v$, of finite index in $J(K_v)$. Clearly it also converges on a finite index subgroup $H_{w}$ of $J(L_w)$.

\begin{mydef}\label{def:Neron_fct_above_p}
The \emph{(local) $p$-adic N\'eron function at $w\mid v\mid p$}, with respect to the divisor $2\Theta$, the character $\chi$ and the subspace $W_v$, is the function
\begin{equation*}
\lambda_w\colon J_{\Theta}(L_w) \to \Q_p,
\end{equation*}
defined as follows. If $P\in J_{\Theta}(L_w)$ is a non-torsion point and $m$ is a positive integer such that $mP\in H_w\setminus \Supp(\Theta)$, then 
\begin{equation*}
\lambda_w(P) = -\frac{2}{n_w m^2}\cdot \chi_{L,w}\left(\frac{\sigma_v(T(mP))}{\phi_m(P)}\right);
\end{equation*}
we may extend the definition of $\lambda_w$ to torsion-points in $J_{\Theta}(L_w)$ by continuity. 
\end{mydef}

By Lemma \ref{lemma:exists_m}, an integer $m$ satisfying the conditions of Definition \ref{def:Neron_fct_above_p} always exists, and by Proposition \ref{prop:finite_index_subgp}\thinspace{}\ref{prop:finite_index_subgp:3} and Equation \eqref{eq:div_quad}, the definition of $\lambda_w$ is independent of the choice of $m$. Moreover, by Proposition \ref{prop:finite_index_subgp}\thinspace{}\ref{prop:finite_index_subgp:2} and by Equation \eqref{eq:divisor_div_poly}, $\sigma_v\circ T\circ m$ and $\phi_m$ have the same zeros on $J_{\Theta}$. This is what allows to extend to torsion-points.  

Assume now that $v\nmid p$. As the name suggests, we are looking for a $\Q_p$-valued analogue of a classical real-valued N\'eron function of divisor $2\Theta$ on $J_{\Theta}(L_w)$ (in the sense of \cite[Chapter 11]{Lang_diophantine_geometry},\cite[\S B.9]{Hindry_silverman}, \cite[\S 7]{Uchida}, and in particular \cite{Uchidacanloc} for the genus $2$ case). In this case (i.e.\ when $w\nmid p$), it is possible to construct such a $p$-adic function starting from a suitable real-valued one, just by replacing a real-valued continuous homomorphism $L_w^{\times}\to\R$ with $\chi_{L,w}$. We refer to \cite{BKM22} for details. 

We take a slightly different approach here. Denote by $J_{1,w}(L_w)$ the model-dependent kernel of reduction for $J$ over $L_w$ (as in \eqref{eq:kernel_red}).

\begin{lemma}\label{lemma:equals_naive}
For $P\in J_{1,w}(L_w)\setminus \Supp(\Theta)$, let $\lambda_w(P) = -\frac{2}{n_w}\chi_{L,w}(T_1(P))$. Then 
\begin{equation}\label{eq:equals_naive}
\lambda_{w}(P) = -\frac{1}{n_w}\chi_{L,w}(\max_{i,j}\{|X_{ij}(P)|_w, 1\}).
\end{equation}
\end{lemma}

\begin{proof}
For such $P$,
\begin{equation*}
\chi_{L,w}(T_1(P)) = \chi_{L,w}(|T_1(P)|_w^{-1}) = \frac{1}{2}\chi_{L,w}(\max_{i,j}\{|X_{ij}(P)|_w, 1\}),
\end{equation*}
 (cf.\ the expansions \eqref{eq:Xij_exp}). 
\end{proof}

\begin{cor} \label{cor:quasi_quad_in_formal}
Let $\lambda_w\colon J_{1,w}(L_w)\setminus \Supp(\Theta)\to \Q_p$ as in Lemma \ref{lemma:equals_naive}. Then
\begin{enumerate}[label=(\roman*)]
\item for all integers $m$ such that $mP\not\in \Supp(\Theta)$,
\begin{equation*}
\lambda_w(mP) = m^2\lambda_w(P) - \frac{2}{n_w}\chi_{L,w}(\phi_{|m|}(P)).
\end{equation*}
\item for all $P,Q\in J_{1,w}(L_w)\setminus \Supp(\Theta)$ such that $P+Q,P-Q\not\in \Supp(\Theta)$, 
\begin{align*}
\lambda_w(P+Q) + \lambda_w(P-Q) = 2\lambda_w(P) + 2\lambda_w(Q)\\
 - \frac{2}{n_w}\chi_{L,w}(-X_{11}(P)+X_{11}(Q)-X_{12}(P)X_{22}(Q) + X_{22}(P)X_{12}(Q)).
\end{align*}
\end{enumerate}
\end{cor}
\begin{proof}
The equality of Lemma \ref{lemma:equals_naive} implies this, in view of \cite[Theorem 4.1]{Stoll} (see also \cite[Corollary 7.2]{Uchidacanloc}), \cite[Proposition 6.2]{muller_de_jong} and \cite[Theorem 7.5]{Uchida}.
\end{proof}

\begin{mydef}\label{def:Neron_fct_away_p}
The \emph{(local) $p$-adic N\'eron function at $w\mid v\nmid p$}, with respect to the divisor $2\Theta$ and the character $\chi$, is the function
\begin{equation*}
\lambda_w\colon J_{\Theta}(L_w) \to \Q_p,
\end{equation*}
defined as follows. If $P\in J_{\Theta}(L_w)$ is a non-torsion point and $m$ is a positive integer such that $mP\in J_{1,w}(L_w)\setminus \Supp(\Theta)$, then 
\begin{equation}\label{eq:lambdaw_away_p}
\lambda_w(P) = -\frac{2}{n_w m^2}\cdot \chi_{L,w}\left(\frac{T_1(mP)}{\phi_m(P)}\right);
\end{equation}
we may extend the definition of $\lambda_w$ to torsion-points in $J_{\Theta}(L_w)$ by continuity. 
\end{mydef}
Similarly to Definition \ref{def:Neron_fct_above_p}, also in this case Lemma \ref{lemma:exists_m} guarantees the existence of $m$ and the extension to torsion points is meaningful. By Corollary \ref{cor:quasi_quad_in_formal} and Equation \eqref{eq:div_quad}, the definition is independent of $m$.  

\begin{rmk}\label{rmk:neron_away_p}
\leavevmode
\begin{enumerate}[label=(\roman*)]
\item  For all $P\in J_{\Theta}(L_w) $, we have $\frac{n_w\lambda_w(P)}{\chi_{L,w}(\ell)}=: \lambda_w^{\prime}(P)\in \Q$, where $\ell$ is the rational prime below $w$. In fact, the real-valued N\'eron function of divisor $2\Theta$ of \cite{Uchidacanloc} at $P$ is $\lambda_w^{\prime}(P)\log|\ell|_w$, where $\log$ is the real logarithm.
\item\label{rmk:neron_away_p_good_reduction} If $C$ has good reduction at $w$, the equality of Lemma \ref{lemma:equals_naive} holds for all $P\in J_{\Theta}(L_{w})$, since in this case the right hand side of \eqref{eq:equals_naive} defines a function satisfying Corollary \ref{cor:quasi_quad_in_formal} on the whole of $J_{\Theta}(L_w)$ (cf.\ \cite[Lemma 1]{FlynnSmart} for the case $K=L=\Q$ and \cite[Proposition 5.2]{Stoll} for the general case; see also Section \ref{sec:application_bihyper} and, in particular, Theorem \ref{thm:properties_mu}). 
\item\label{rmk:neron_unramified} Equivalently, we could have used Definition \ref{def:Neron_fct_above_p} to define the local $p$-adic N\'eron function also in the case $v\nmid p$, since $\chi_{L,w}$ vanishes on $\OO_{w}^{\times}$ and $\sigma_v(T) = T_1 (1+ O(T_1,T_2))$.  We preferred stating the simplified formula, which does not involve the choice of a subspace $W_v$. Conversely, if $v\mid p$ and $\chi$ is unramified at $v$, Definition \ref{def:Neron_fct_above_p} simplifies to Definition \ref{def:Neron_fct_away_p}.
\end{enumerate}
\end{rmk}

We now state the properties satisfied by our local $p$-adic N\'eron function at $w$, regardless of whether $w$ divides or does not divide $p$. These are completely analogous to the properties of the real-valued N\'eron function: see \cite[Theorem 5.3, Theorem 5.6]{Uchidacanloc} for the genus $2$ case, and \cite[Theorem 7.5]{Uchida} for arbitrary genus (in the latter, the N\'eron function is with respect to $\Theta$, rather than $2\Theta$).
\begin{prop}\label{prop:properties_neron_fcts}
For any place $w$ of $L$, the local $p$-adic N\'eron function at $w$ has the following properties:
\begin{enumerate}[label=(\roman*)]
\item \label{it:extension} If $F$ is a finite extension of $L$ and $w^{\prime}$ a place of $F$ above $w$, then $\lambda_{w^{\prime}}$ restricts to $\lambda_w$ on $J_{\Theta}(L_w)$.
\item\label{it:quasi_quadraticity} For all $P\in J_{\Theta}(L_w)$ and for all integers $m$ such that $mP\in J_{\Theta}(L_w)$,
\begin{equation*}
\lambda_w(mP) = m^2\lambda_w(P) - \frac{2}{n_w}\chi_{L,w}(\phi_{|m|}(P)).
\end{equation*}
\item\label{it:quasi_parallelogram}  For all $P,Q\in J_{\Theta}(L_w)$ such that $P+Q,P-Q\in J_{\Theta}(L_w)$,
\begin{align*}
\lambda_w(P+Q) + \lambda_w(P-Q) = 2\lambda_w(P) + 2\lambda_w(Q)\\
 - \frac{2}{n_w}\chi_{L,w}(-X_{11}(P)+X_{11}(Q)-X_{12}(P)X_{22}(Q) + X_{22}(P)X_{12}(Q)).
\end{align*}
\end{enumerate}
\end{prop}
\begin{proof}
\ref{it:extension} is clear. 
By Proposition \ref{prop:finite_index_subgp}\thinspace{}\ref{prop:finite_index_subgp:3} and Corollary \ref{cor:quasi_quad_in_formal}, properties \ref{it:quasi_quadraticity} and \ref{it:quasi_parallelogram} hold in a subgroup of finite index $H_w$. Since we used \ref{it:quasi_quadraticity} to extend $\lambda_w$ outside of this subgroup, part \ref{it:quasi_quadraticity} holds everywhere in view of \eqref{eq:div_quad}. As regards \ref{it:quasi_parallelogram}, assume $P,Q,P+Q,P-Q$ are non-torsion points. Since $H_w$ has finite index in $J(L_w)$, we can find a positive integer $m$ such that $mR\in H_w\setminus\{0\}$ for each  $R\in \{P,Q,P+Q,P-Q\}$. By Lemma \ref{lemma:stevan}, we can also ensure, possibly after replacing $m$ with a multiple, that each $mR$ does not lie on $\Theta$. Then \ref{it:quasi_parallelogram} holds when $P$ and $Q$ are replaced by $mP$ and $mQ$. 
Finally, by \eqref{eq:div_par}, the equality \ref{it:quasi_parallelogram} holds for $P,Q$. The torsion case follows by continuity.
\end{proof}

\subsection{Global height}
\begin{mydef}\label{def:global_ht}
The \emph{global $p$-adic height} for $J$ with respect to the character $\chi$ and the subspaces $W_v\in \Is(C/K_v)$ (at $v\mid p$) is the function
\begin{equation*}
h_p\colon J(\overline{\Q}) \to \Q_p,
\end{equation*}
defined as follows. If $P$ belongs to $J_{\Theta}(L)$, for some finite extension $L$ of $K$, then
\begin{equation}\label{eq:hp_sum}
h_p(P) = \frac{1}{[L:\Q]}\sum_{w}n_w\lambda_w(P),
\end{equation}
where the sum runs over all the non-archimedean places $w$ of $L$ and $\lambda_w$ is the local $p$-adic N\'eron function of Definition \ref{def:Neron_fct_above_p} if $w\mid p$ and of Definition \ref{def:Neron_fct_away_p} if $w\nmid p$. Since by Proposition \ref{prop:properties_neron_fcts}\thinspace{}\ref{it:quasi_quadraticity}, $h_p(mP) = m^2h_p(P)$ for $P,mP\in J_{\Theta}(L)$, we extend $h_p$ to $J(L)$ using Lemma \ref{lemma:exists_m} and by requiring that
\begin{equation*}
h_p(mP) = m^2h_p(P)\qquad  \text{continues to hold and}\qquad h_p(0) = 0.
\end{equation*}
\end{mydef}
By Remark \ref{rmk:neron_away_p}\thinspace{}\ref{rmk:neron_away_p_good_reduction}, the sum in \eqref{eq:hp_sum} is finite and by Proposition \ref{prop:properties_neron_fcts}\thinspace{}\ref{it:extension}, it is independent of the choice of $L$. 
Our extension of the $p$-adic N\'eron functions to torsion points together with the vanishing of $\chi_L$ on $L^{\times}$ also guarantees that $h_p(P) = 0$ for a torsion point $P\in J_{\Theta}(L)$. Therefore the extension to $J(L)$ is well-defined. 

 By Proposition \ref{prop:properties_neron_fcts}\thinspace{}\ref{it:quasi_parallelogram}, the global $p$-adic height also satisfies the parallelogram law:
 \begin{equation*}
 h_p(P+Q) + h_p(P-Q) = 2h_p(P) + 2h_p(Q).
 \end{equation*}
 Therefore it induces a symmetric bilinear pairing $J(\overline{\Q})\times J(\overline{\Q})\to\Q_p$.

\begin{rmk}
By \cite{BKM22}, if $p$ and $C$ satisfy the assumptions of Theorem \ref{thm:Blakestad_main} and if we pick, for every $v$ at which $\chi$ is ramified, the canonical $v$-adic sigma function of Blakestad, the $p$-adic height $h_p$ coincides with the canonical global Mazur--Tate height \cite{mazur-tate}.
\end{rmk}

\begin{rmk} Assume $K=\Q$. Extending work of Perrin--Riou for elliptic curves \cite{Perrin-Riou-comptes, Perrin-Riou}, Trip has recently defined a quadratic form $h_p^{\naive}\colon J(\Q)\to \Q_p$ as a limit of a naive height: see \cite[Theorem 3.3.1, 
Definition 3.3.17, Theorem 3.3.19]{Trip} (in her thesis, $h_p^{\naive}$ is denoted $h_p$). Let $\chi$ be the cyclotomic character for $\Q$  (cf.\ \S \ref{subsec:implementation_Neron_fcts}), let $W_p\in \Is(C/\Q_p)$ be the subspace corresponding to the naive sigma function, and let $h_p$ be the global $p$-adic height from above with respect to these choices.  Then Trip shows that $h_p(u) = h_p^{\naive}(u)$ for all $u\in J(\Q)$ \cite[Theorem 3.3.25]{Trip}.
\end{rmk}

\section{Local $p$-adic N\'eron functions vs Colmez's Green functions}\label{sec:Colmez}
In the first part of this section (\S\S \ref{subsec:Colm_int}--\ref{subsec:Colm_heights}) we review some of the results of Colmez\cite{Colm96}. To be more precise, in \S \ref{subsec:Colm_int} we sketch Colmez's construction of a $p$-adic integration theory on smooth algebraic varieties over a $p$-adic field, which satisfies ``natural'' properties and does not require any assumption on the reduction. Similar techniques can be used to attach to divisors on abelian varieties certain $p$-adic Green functions, which are, essentially, double $p$-adic integrals. In \S \ref{subsec:Colm_curves}, single $p$-adic integrals of differentials of the third kind on curves are related to $p$-adic Green functions of divisor $\Theta$ on their Jacobians. \S \ref{subsec:Colm_int} and \S \ref{subsec:Colm_curves} are entirely expository. The results of \S \ref{subsec:Colm_curves} are applied in \S \ref{subsec:Colm_heights} to define an adelic height on Jacobians of curves over a number field $K$, whose component at $p$ coincides with the $p$-adic height of Coleman--Gross and Besser \cite{ColemanGross, BesserPairing}. Such adelic heights are defined in \cite{Colm96}, but we phrase Colmez's results in a slightly more general form: in particular, our adelic heights depend on the choice, at every place $p$ of $\Q$, of a continuous $\Q_p$-valued idele class character for $K$.

In \S \ref{subsec:comparison}, we then restrict to the situation of Section \ref{sec:padic_hts} and compare our local N\'eron functions to Colmez's Green functions of divisor $\Theta$: see Theorem \ref{thm:comparison} below. In view of \S \ref{subsec:Colm_heights}, we also obtain, in our odd degree genus $2$ setting, a formula relating the Coleman--Gross local $p$-adic height at $v$ with our $p$-adic N\'eron function at $v$ (Corollary \ref{cor:lambda_eq_CG}).

We use this to show that the global $p$-adic height $h_p$ of Definition \ref{def:global_ht} is equal to the Coleman--Gross global $p$-adic height (Corollary \ref{cor:global_CG_same_as_this}).
   
\begin{rmk}
We have chosen to devote considerable space to summarising Colmez's results for a few reasons. First, we crucially use a lot of them (and their proofs). Secondly, in order to state our comparison results to a larger degree of generality, we need Colmez's height construction to be extended to  arbitrary idele characters. Finally, we believe it would be difficult for the reader to contextualise the results of \S \ref{subsec:comparison} without having some familiarity with \cite{Colm96}.
   \end{rmk}

\subsection{Colmez's $p$-adic integration theory and Green functions}\label{subsec:Colm_int}
Let $K$ be a closed subfield of $\C_p$ (for instance, a finite extension of $\Q_p$). There is a unique locally analytic function $\Log\colon \OO_K^{\times}\to K$ satisfying:
\begin{equation}\label{eq:Log_colmez}
d\Log{t} = \frac{dt}{t}, \qquad \text{and}\qquad \Log(xy) = \Log(x) + \Log(y), \quad \text{for all } x,y\in \OO_K^{\times}.
\end{equation}
By viewing $\Log{p}$ as a variable, we can extend $\Log$ to a homomorphism $\Log\colon K^{\times}\to K_{\st} \colonequals K[\Log{p}]$. 
One of the main results of \cite{Colm96} is the development of a $K_{\st}$-valued theory of $p$-adic integration of closed differential $1$-forms on smooth geometrically connected algebraic varieties over $K$, without any assumption on the reduction.  Given such a variety $X$, recall that a rational closed differential 1-form is 
\begin{enumerate}[label = (\roman*)]
\item of the \emph{first kind} if it is holomorphic;
\item of the \emph{second kind} if its residue divisor is zero;
\item of the \emph{third kind} if it has at most simple poles and the residue at each pole is an integer. 
\end{enumerate}
See \cite[Lemme I.1.12, Remarque
I.1.18\thinspace{}(ii)]{Colm96} for a precise definition of residue divisor. Any closed differential $1$-form is a linear combination of differentials of the second and third kind  \cite[Corollaire I.1.16]{Colm96}.

Let $K(X)$ be the field of rational functions on $X$ and $K(X)_{\st}$ the ring generated by $K(X)$ and $\Log(f)$, for $f\in K(X)^{\times}$. 
A \emph{locally log-meromorphic} function is a function which locally (with respect to the $p$-adic topology) can be written as a polynomial in logarithms of analytic functions with meromorphic coefficients. 

\begin{thm}[{\hspace{1sp}\cite[Th\'eor\`eme 1]{Colm96}}]\label{thm:uniquetheory}
There exists a unique theory of integration on smooth algebraic varieties over $K$ satisfying the following properties, for all such varieties $X$, closed differential $1$-forms $\omega$, $\omega_1$ and $\omega_2$ on $X$, points $a,b,c\in X(K)$ and scalars $\mu_1,\mu_2\in K$:
\begin{enumerate}[label = (\roman*)]
\item \label{it:loc_an} $F(x) = \int_{a}^x \omega$ is a locally log-meromorphic function of $x\in X(K)$ with values in $K\oplus K \Log{p}$ satisfying $dF = \omega$; 
\item $\int_a^b (\mu_1\omega_1 + \mu_2\omega_2) = \mu_1 \int_a^b\omega_1 +\mu_2 \int_a^b\omega_2$;
\item $\int_a^c\omega = \int_a^b\omega + \int_b^c \omega$;
\item\label{it:exact} If $\omega = dg$ (with $g\in K(X)_{\st}$) is exact, then $\int_a^b \omega = g(b) - g(a)$;
\item\label{change_of_vars} If $f\colon X\to Y$ is a morphism of smooth varieties over $K$, $\eta$ is a closed differential $1$-form on $Y$, then $\int_a^b f^{*}\eta = \int_{f(a)}^{f(b)}\eta$.
\end{enumerate}
\end{thm}

The strategy is to first consider the case where $X$ is an abelian variety, and to then use an Albanese morphism combined with functoriality \ref{change_of_vars} to extend to general varieties. In view of \cite[Remarque
2\thinspace{}(iv)]{Colm96}, it is essentially enough to consider the case $K=\C_p$.

Assume now that $X$ is an abelian variety over $K$. We give a rough summary of Colmez's construction in this case. 

A holomorphic differential $\omega$ on $X$ is translation-invariant. Therefore, by \ref{change_of_vars}, the primitive of $\omega$ vanishing at $0$ must be a group homomorphism $X(K)\to K$.  We can construct such a primitive as follows. There is an open subgroup $V$ of $X(K)$ such that the unique formal primitive $\mathcal{L}_{\omega}$ of $\omega$ that satisfies $\mathcal{L}_{\omega}(0) = 0$ converges on $V$ and is a group homomorphism $V\to K$. 
 Since $X(K)/V$ is torsion \cite[{Th\'eor\`eme II.1.9}]{Colm96}, we can extend $\mathcal{L}_{\omega}$ to a locally analytic group homomorphism $\mathcal{L}_{\omega}\colon X(K)\to K$ by insisting that $\mathcal{L}_{\omega}(na) = n\mathcal{L}_{\omega}(a)$ for all integers $n$ and points $a$. 
Conversely, given a locally analytic group homomorphism $\mathcal{G}\colon X(K)\to K$, its differential is a translation-invariant, and hence holomorphic, differential. That is, the map $\mathcal{G}\mapsto d\mathcal{G}$ induces an isomorphism between the $K$-vector space of locally analytic group homomorphisms $X(K)\to K$ (which we call \emph{one-logarithms} of $X$) and $H^0(X,\Omega^1)$ \cite[Lemme II.1.13]{Colm96}.

\begin{example}
When $X$ is the Jacobian of the curve $C$ of \S \ref{subsec:formal}, and $\omega = \Omega_i$ for $i\in\{1,2\}$, the primitive $\mathcal{L}_{\Omega_i}$ is an extension to $J(K)$ of the $i$-th component $\mathcal{L}_i$ of the strict logarithm $\mathcal{L}$. See Proposition \ref{prop:exp_exp_log} (and Lemma \ref{lemma:subs_exponential}) for convergence properties of $\mathcal{L}_i$. 
The $K$-vector space of one-logarithms of $X$ is $2$-dimensional and spanned by $\mathcal{L}_{\Omega_1}$ and $\mathcal{L}_{\Omega_2}$.
\end{example}

The case of differentials of the second and third kind is more involved. Given a positive integer $n$ and a subset $I$ of $\{1,\dots, n\}$, let 
\begin{equation*}
m_I\colon X^{n+1}\to X, \qquad (x,h_1,\dots,h_n)\mapsto x + \sum_{i\in I} h_i
\end{equation*}
and given $\alpha$, a differential form, a function, or a divisor on $X$, set
\begin{equation*}
\Delta^{[n]}\alpha = \sum_{I\subseteq \{1,\dots, n\}} (-1)^{n-|I|} m_I^{*} \alpha.
\end{equation*}
By the theorem of the square \cite[Propositions I.3.14, I.3.15]{Colm96}, if $\omega$ is a differential form of the second or third kind on $X$, then $\Delta^{[2]}\omega$ is exact on $X^{3}$. 
Let $F_{\Delta^{[2]}\omega}$ be the element of $K(X^3)_{\st}$ satisfying $dF_{\Delta^{[2]}\omega} = \Delta^{[2]}\omega$ and $F_{\Delta^{[2]}\omega}(x,0,0) = 0$. By \ref{change_of_vars} and the choice of $F_{\Delta^{[2]}\omega}$, we must have
\begin{equation*}
\Delta^{[2]}\biggl(\int \omega\biggr) = F_{\Delta^{[2]}\omega}.
\end{equation*}
Colmez proves a sufficient criterion for a locally log-meromorphic function $g$ on $X(\C_p)^{n+1}$ to be of the form $\Delta^{[n]}f$, for a locally log-meromorphic function $f$ on $X(\C_p)$; when this applies, $f$ is unique up to addition by a polynomial of degree at most $n-1$ in the one-logarithms of $X$: see \cite[Th\'eor\`eme II.1.16]{Colm96} for a precise statement. Applying this technical result to $g=F_{\Delta^{[2]}\omega}$, one solves for a primitive $f = \int \omega$ (see \cite[Proposition II.1.17]{Colm96}). See \S \ref{subsec:diff_first_2nd_third} for explicit formulae when $X$ is the Jacobian of the curve $C$. 

Theorem \ref{thm:uniquetheory} concerns single integration only. However, applying a similar strategy to the one we have just outlined for integrals of differentials of the second and third kind, Colmez also constructs certain $p$-adic Green functions on the abelian variety $X$, which are double integrals in disguise (see Proposition \ref{prop:dpartial_second_kind} below).

The theorem of the cube states the following:
\begin{prop}[{\hspace{1sp}\cite[Proposition I.3.18]{Colm96}}]\label{prop:Delta3_principal}
If $X$ is an abelian variety over an algebraically closed field of characteristic $0$ and $D$ is a divisor of $X$, then $\Delta^{[3]}D$ is principal.
\end{prop}
Let $D$ be a divisor of $X$ and let $f_D^{(4)}$ be the rational function on $X^4$ whose divisor is $\Delta^{[3]}D$ and whose restriction to $(X^3\times\{0\})\cup (X^2\times \{0\}\times X) \cup (X\times\{0\}\times X^2)$ is equal to the constant function $1$.  
The aforementioned criterion of \cite[Th\'eor\`eme II.1.16]{Colm96} applied to $\Log f_D^{(4)}$ allows one to go from $X^4$ to $X$:

\begin{prop}[{\hspace{1sp}\cite[Proposition 4, Proposition II.1.19]{Colm96}}]\label{prop:padicGreen}
There exists a locally log-meromorphic function $G_D$ on $X$, with values in $K\oplus \Q\Log{p}$, such that $\Delta^{[3]}G_D = \Log{f_D^{(4)}}$. Moreover, $G_D$ is unique up to addition by a polynomial of degree at most $2$ in the one-logarithms of $X$, it has a logarithmic singularity along $D$ and it is locally analytic away from the support of $D$.
\end{prop}

\begin{mydef}
A function $G_D$ satisfying Proposition \ref{prop:padicGreen} is called a \emph{Green function of divisor $D$}.
\end{mydef}

The following proposition tells us that the Green function $G_D$ is a sort of double $p$-adic integral:

\begin{prop}[{\hspace{1sp}\cite[Proposition II.1.20]{Colm96}}]\label{prop:dpartial_second_kind}
Let $\partial$ be an invariant derivation on $X$. If $G_D$ is a Green function of divisor $D$, then $d(\partial G_D)$ is a differential form of the second kind.
\end{prop}

\subsection{Colmez's theory on curves and Jacobians}\label{subsec:Colm_curves}
As in \S \ref{subsec:Colm_int}, let $K$ be a closed subfield of $\C_p$. Let $X$ be a smooth, projective, geometrically irreducible curve of positive genus $g$ over $K$ such that $X(K)$ is non-empty, let $J$ be its Jacobian, and given a fixed $P_0\in X(K)$, consider the embedding $\iota\colon X\xhookrightarrow{} J$ with base point $P_0$. This induces an isomorphism $\iota^{*}\colon H_{\dr}^1(J)\to H_{\dr}^1(X)$. 

\begin{rmk}\label{rmk:from_X_to_C}
When $X$ is the genus $2$ curve given explicitly by an equation of the form \eqref{eq:Grant}, a lot of the notation that we are going to introduce now has already been introduced or fixed (e.g.\ $\iota$ is the embedding with respect to $P_0 = \infty$). Since there is no clear advantage in specialising to that setting here, we keep the exposition of Colmez's results general. The rule to keep in mind when we return to the case $X = C$ in \S \ref{subsec:comparison} is that all the notation of this section should be specialised to that previously fixed for $C$.
\end{rmk}

Fix a basis $(\Omega_1,\dots, \Omega_g)$ for $H^0(J,\Omega^1)$ and denote by $(\partial_1,\dots,\partial_g)$ its dual basis of invariant derivations. Consider the function $\varpi\colon X^g\to J$, defined by $\varpi(P_1,\dots,P_g) = \iota(P_1)+ \cdots + \iota(P_g)$. Then $\varpi$ is surjective and induces a birational equivalence of the $g$-th symmetric power $X^{(g)}$ of $X$ with $J$. We call a point $u\in J$ \emph{generic} if there is a unique element of $X^{(g)}$ mapping to $u$ under the map induced by $\varpi$ and the latter has distinct $P_i$. 
The \emph{theta divisor}, denoted $\Theta$, is the divisor of $J$ which is the image of
\begin{equation*}
X^{g-1}\to J, \qquad  (P_1,\dots,P_{g-1} )\mapsto - \iota(P_1)\cdots - \iota(P_{g-1}).
\end{equation*}
For generic $u = \iota(P_{1,u})+ \cdots +\iota(P_{g,u})$, we consider the following divisor on $X$ (cf.\ \cite[Lemma 6.7]{milneAV}):
\begin{align*}
X_u &=\sum_{P\in \iota^{-1}(u+ \Theta)} P
 = P_{1,u}+\dots+ P_{g,u}.
\end{align*}
The theta divisor is \emph{symmetric}, i.e.\ there exists a unique $w\in J$ such that $x$ belongs to $\Theta$ if and only if $w- x$ belongs to $\Theta$ (by \cite[Corollaire I.4.7]{Colm96} and Lefschetz principle). 
If $G$ is a Green function of divisor $\Theta$, so is $G_w(x) \colonequals G(w-x)$ and we say that $G$ is \emph{symmetric} if $G(x) = G_w(x)$. 

Recall that, by Proposition \ref{prop:dpartial_second_kind}, $d(\partial_i G)$ is a differential of the second kind on $J$.

\begin{prop}[{\hspace{1sp}\cite[Proposition II.2.4]{Colm96}}]\label{prop:bijection}
The function that maps a Green function $G$ to the subspace of $H^1_{\dr}(X)$ generated by $\iota^*d(\partial_1 G), \dots, \iota^* d(\partial_g G)$ induces a bijection between the set of symmetric Green functions of divisor $\Theta$ (up to addition by a constant) and the set $\Is(X)$ of subspaces of $H^1_{\dr}(X)$ which are complementary to $H^0(X,\Omega^1)$ and isotropic for the cup product. 
\end{prop}

Using this proposition and properties of Green functions, Colmez expresses integrals of differentials of the third kind on $X$ (as in Theorem \ref{thm:uniquetheory}) in terms of pullbacks of symmetric Green functions for the theta divisor of $J$. We now explain how.

Given a differential $\omega$ of the third kind on $X$, let $\Res_P(\omega)$ be the residue of $\omega$ at $P$ and let
\begin{equation*}
\Div(\omega) = \sum_{P\in X}\Res_P(\omega) P \in \Div^0(X).
\end{equation*}
In particular, 
\begin{equation*}
\Div\left(\frac{df}{f}\right) = \div(f),
\end{equation*}
where the latter is the divisor of a rational function, in the usual sense. 
 
Fix an element $W\in \Is(X)$ and let $G_{\Theta,W}$ be a corresponding symmetric Green function (by Proposition \ref{prop:bijection}, this is unique up to addition by a constant); set $F_{i,W} = \partial_i G_{\Theta,W}$.  
Given generic $u,v\in J$, consider the differential
\begin{equation}
\omega_u - \omega_v = \iota^*\left(\sum_{i=1}^g (F_{i,W}(x- u) - F_{i,W}(x-v)) \Omega_i\right).
\end{equation}
\begin{lemma}
$\omega_u - \omega_v$ is of the third kind and
\begin{equation*}
\Div(\omega_u - \omega_v) = X_u - X_v.
\end{equation*}
\end{lemma}
\begin{proof}
See \cite[p.\,92]{Colm96}.
\end{proof}

We consider the $\Z$-module generated by the differentials of the form $\omega_u -\omega_v$ for $u,v\in J$ generic points; we call a differential in this module a \emph{differential of type $(1,1)$}. This definition depends on the choice of $W\in \Is(X)$.

\begin{prop}[{\hspace{1sp}\cite[Proposition II.2.9]{Colm96}}]\label{prop:dec_third_kind}
Let $\omega$ be a differential of the third kind on $X$. Then there exist a unique holomorphic differential $\omega^{1,0}$ and a unique differential $\omega^{1,1}$ of type $(1,1)$ such that $\omega = \omega^{1,0} + \omega^{1,1}$. Moreover,
\begin{enumerate}[label = (\roman*)]
\item If $\omega = \frac{df}{f}$, then $\omega^{1,0} = 0$.
\item\label{it:dec_4} If $\omega^{1,1} = \sum_i n_i \omega_{u_i}$, then the image of $\Div(\omega)$ in $J$ is $\sum_i n_iu_i$ and 
\begin{equation*}
\int^x \omega = \int^x \omega^{1,0} + \sum_i n_{i} G_{\Theta,W}(\iota(x)-u_i).
\end{equation*}
\end{enumerate}
\end{prop}
\begin{example}
If $\Div(\omega) = P_1 - P_2$, pick $Q_1, \dots, Q_{g-1}$ such that $u = \iota(P_1) + \iota(Q_1) +  \cdots + \iota(Q_{g-1})$ and $v = \iota(P_2) + \iota(Q_1) + \cdots + \iota(Q_{g-1})$ are both generic. Then
\begin{equation*}
\Div(\omega_u -\omega_v) = X_u - X_v = P_1 - P_2;
\end{equation*}
hence 
\begin{equation*}
\omega - (\omega_u - \omega_v) = \omega^{1,0}\in H^0(X, \Omega^1)
\end{equation*}
and
\begin{equation*}
\int^x \omega = \int^x \omega^{1,0} +G_{\Theta,W}(\iota(x)-u)- G_{\Theta,W}(\iota(x)-v).
\end{equation*}
This example should clarify that, although the differential $\omega^{1,1}$ is unique, its decomposition as a sum $\sum_i n_i\omega_{u_i}$ is not. The choices of the generic points $u_i\in J$ affect the logarithmic singularities of the individual terms $G_{\Theta,W}(\iota(x)-u_i)$; however, overall, $\sum_i n_{i} G_{\Theta,W}(\iota(x)-u_i)$ is locally analytic outside $\Div(\omega)$.
\end{example}
For every $i\in\{1,\dots,g\}$, denote by $\mathcal{L}_i$ the one-logarithm of $J$ whose differential is the basis element $\Omega_i$ for $H^0(J,\Omega^1)$. 

\begin{mydef}\label{def:LogJtilde}
For a differential $\omega$ of the third kind on $X$, we define $\Log_{\tilde{J}}(\omega)\in H^1_{\dr}(X)$ as the class of the differential
\begin{equation*}
 \omega^{1,0} + \sum_{i=1}^g \mathcal{L}_i(\Div(\omega))\eta_i,
\end{equation*}
where $\omega^{1,0}$ is as in Proposition \ref{prop:dec_third_kind} and, for every $i\in \{1,\dots,g\}$, $\eta_i = \iota^*d(\partial_i G_{\Theta,W})$. 
\end{mydef}
As a corollary to Proposition \ref{prop:dec_third_kind} and since $\eta_i\in W$ by construction, we obtain 
\begin{lemma}\label{lemma:prop_LogJtilde}\leavevmode
\begin{enumerate}[label = (\roman*)]
\item $\Log_{\tilde{J}}(\omega) = [\omega]$ if $\omega$ is holomorphic.
\item If $\omega$ is of type $(1,1)$ with respect to $W$, then $\Log_{\tilde{J}}(\omega)$ belongs to $W$.
\item \label{it:prop_LogJtilde3}$\Log_{\tilde{J}}\left(\frac{df}{f}\right) = 0$.
\end{enumerate}
\end{lemma}

\begin{rmk}\label{rmk:logJtilde}
The function $\Log_{\tilde{J}}$ is defined more conceptually in \cite[Th\'eor\`eme II.2.11]{Colm96} and it is then shown to satisfy Definition \ref{def:LogJtilde} and Lemma \ref{lemma:prop_LogJtilde}. Denoting by $T(K)$ the group of differentials of the third kind on $X$, and by $T_{\ell}(K)$ the subgroup of logarithmic differentials (i.e.\ those of the form $\frac{df}{f}$ for some $f\in K(X)^{\times}$), we see that $\Log_{\tilde{J}}$ induces a group homomorphism
\begin{equation*}
\Log_{\tilde{J}}\colon T(K)/T_{\ell}(K)\to H^1_{\dR}(X).
\end{equation*}
In fact, this agrees with the canonical homomorphism $\Psi$ of \cite[Proposition 2.5]{ColemanGross} (see \cite[Remarque II.2.12]{Colm96}).
\end{rmk}

\subsection{Colmez's adelic heights}\label{subsec:Colm_heights}
\subsubsection{Notation and preliminary choices}
\label{subsubsec:notation_colm_heights}
Let $K$ be a number field. For every place $v$ of $K$, we define a $\Q$-vector space $\mathscr{L}(K_v^{\times})$ and a logarithm $\Log_v\colon K_v^{\times}\to \mathscr{L}(K_v^{\times})$, as follows.

If $v$ is a non-archimedean place, let $p$ be the rational prime below $v$ and let $\Log_v$ be the logarithm $\Log$ for the field $K_v$, as defined at the beginning of \S \ref{subsec:Colm_int}. Then $\Log_v$ takes values in $\mathscr{L}(K^{\times}_v) \colonequals K_v\oplus \Q \Log_v{p}$. An element $x\in \mathscr{L}(K_v^{\times})$ can be written uniquely as $x = x^{(0)} + x^{(1)}\Log_v{p}$ with $x^{(0)}\in K_v$ and $x^{(1)}\in \Q$. We will use the superscripts $(0)$ and $(1)$ with this meaning. 

If $v$ is an archimedean place, we let $\mathscr{L}(K_v^{\times}) = \R$ and $\Log_v\colon K_v^{\times}\to \mathscr{L}(K_v^{\times})$ the map $\Log_v{x} = \log{|x|_{\infty}}$, where $|\cdot |_{\infty}$ is the standard absolute value on $\R$ or $\C$. 

We then define the $\Q$-vector space
\begin{equation*}
\mathscr{L}(\A_{K}^{\times}) = {\prod}^{\prime} \mathscr{L}(K_v^{\times})= \{(x_v)_v\in \prod \mathscr{L}(K_v^{\times}) : x_v\in K_v \ \text{for almost all } v\}.
\end{equation*}
Since $\Log_v(\OO_v^{\times})\subset K_v$, applying $\Log_v$ component-wise gives a well-defined map
$\Log\colon \A_K^{\times}\to \mathscr{L}(\A_{K}^{\times})$.
Denoting by $\mathscr{L}(K^{\times})$ the $\Q$-vector subspace of $\mathscr{L}(\A_{K}^{\times})$ generated by $\Log(K^{\times})$, we set
\begin{equation*}
\mathscr{L}(\A_{K}^{\times}/K^{\times})\colonequals \mathscr{L}(\A_{K}^{\times})/\mathscr{L}(K^{\times}).
\end{equation*}

Next, we define a $\Q$-linear map $T\colon \mathscr{L}(\A_{K}^{\times}/K^{\times})\to \prod_{p\leq \infty} \Q_p$ (where $\Q_{\infty}\colonequals \R$). The map $T$ will depend on the choice, at every $p$, of a continuous idele class character $\chi^{(p)}\colon \A_{K}^{\times}/K^{\times} \to \Q_p$ (cf.\ Section \ref{sec:padic_hts} for the case $p<\infty$). We denote by $\mathscr{I}(K)$ the set of all possible choices of $\chi = (\chi^{(p)})_{p\leq \infty}$.

\begin{rmk}
In Colmez's set-up there is no dependence on an element of $\mathscr{I}(K)$. On the other hand, implicit in the construction is the choice of (a scalar multiple of) the cyclotomic character at $p<\infty$ (see \cite[Example 2.7]{QCnfs}) and a choice of a continuous real-valued idele class character. The latter choice amounts to the choice of a multiplicative constant.
\end{rmk}
\begin{lemma}\label{lemma:Tvp}
Let $\chi^{(p)}\colon \A_{K}^{\times}/K^{\times} \to \Q_p$ be a continuous idele class character. For every place $v$ of $K$, there exists a unique $\Q$-linear map $T_v^{(p)}\colon \mathscr{L}(K_v^{\times})\to \Q_p$ making the following diagram commutative
\begin{equation*}
   \xymatrix{
      {K_{v}^\times}  \ar[rr]^{\chi_{v}^{(p)}} \ar[dr]_{\Log_v} & &   \Q_p.\\
      & \mathscr{L}(K_v^{\times})\ar[ur]_{T_v^{(p)}}
  }
  \end{equation*}
  Moreover, for all but finitely many $v$, $T_v^{(p)}$ vanishes identically on $K_v$. 
\end{lemma}
\begin{proof}
$\Q$-linearity gives uniqueness. 
For existence,
\begin{itemize}
\item if $v$ and $p$ are infinite, then $\chi_v^{(\infty)} = \alpha\Log_v$ for some $\alpha\in \R$, so we let $T_v^{(\infty)}(x) = \alpha x$ for all $x\in \mathscr{L}(K_v^{\times}) = \R$;
\item if $p$ is finite and $v\mid p$, there exists a unique $\Q_p$-linear map $t_v^{(p)}\colon K_v\to \Q_p$ such that $\chi_v^{(p)} = t_v^{(p)}\circ \Log_v$ on $\OO_v^{\times}$ (see, for example, \cite[\S 2.1]{QCnfs}). We then let
\begin{equation*} 
T_v^{(p)}(x) = t_v^{(p)}(x^{(0)})  + x^{(1)}\chi_v^{(p)}(p),\qquad  \text{for all } x\in \mathscr{L}(K_v^{\times});
\end{equation*}
\item if $p$ is finite and $v$ is infinite, $\chi_v^{(p)}$ is identically zero, so $T_v^{(p)}(x) = 0$;
\item in all the remaining cases, $\chi_v^{(p)}$ is identically zero on $\OO_v^{\times}$, so
\begin{equation*}
T_v^{(p)}(x) = x^{(1)} \chi_v^{(p)}(\ell),
\end{equation*}
where $\ell$ is the rational prime below $v$.\qedhere
\end{itemize}
\end{proof}
\begin{mydef}\label{def:Tmap}
Fix $\chi = (\chi^{(p)})_p\in \mathscr{I}(K)$ and, for every pair $(v,p)$ of a place of $K$ and one of $\Q$, denote by $T_v^{(p)}$ the map of Lemma \ref{lemma:Tvp} for $\chi_v^{(p)}$. By the lemma, the following map (whose definition depends on $\chi$) is well-defined and $\Q$-linear:
\begin{equation}\label{eq:maps_prod_Qp}
T \colon \mathscr{L}(\A_{K}^{\times}/K^{\times})\to \prod_{p\leq \infty} \Q_p, \qquad T((x_v))_p = \sum_v T_v^{(p)}(x_v).
\end{equation}
\end{mydef}
Suppose $L$ is a finite extension of $K$. Then for every place $v$ of $K$ and place $w\mid v$ of $L$, there is a unique $\Q$-linear map $\Tr_{L_{w}/K_v}\colon \mathscr{L}(L_w^{\times})\to \mathscr{L}(K_v^{\times})$ satisfying
\begin{equation*}
\Tr_{L_w/K_v}(\Log_w(x)) = \Log_v(N_{L_w/K_v}(x)) \qquad \text{for all } x\in L_w^{\times}.
\end{equation*}
These local trace maps induce a $\Q$-linear map
\begin{equation*}
\Tr_{L/K}\colon \mathscr{L}(\A_L^{\times})\to \mathscr{L}(\A_K^{\times}), \qquad \Tr_{L/K}(\dots,x_w,\dots) = \biggl(\dots, \sum_{w\mid v} \Tr_{L_w/K_v} (x_w),\dots\biggr)
\end{equation*}
that factors through
\begin{equation*}
\Tr_{L/K}\colon \mathscr{L}(\A_L^{\times}/L^{\times})\to \mathscr{L}(\A_K^{\times}/K^{\times}).
\end{equation*}

Therefore, given $\chi\in \mathscr{I}(K)$ as above, the composition of $T$ with $\Tr_{L/K}$ gives a $\Q$-linear $\mathscr{L}(\A_L^{\times}/L^{\times})\to \prod_{p\leq \infty}\Q_p$. 
\begin{rmk}\label{rmk:trace_ext}
In fact,  $T\circ \Tr_{L/K}$ is equal to the $\Q$-linear map $\mathscr{L}(\A_{L}^{\times}/L^{\times})\to \prod_{p\leq \infty} \Q_p$ that Definition \ref{def:Tmap} associates to the character $\chi_{L}\in \mathscr{I}(L)$ given by $\chi_L^{(p)} = \chi^{(p)}\circ N_{L/K}$ where $N_{L/K}$ is the norm on the idele class group.
\end{rmk}

\subsubsection{Adelic heights with respect to $\chi$}
Let now $X$ be a smooth, projective, geometrically irreducible curve over $K$, of genus greater than $0$, with Jacobian $J$. We assume that we have fixed $\chi\in \mathscr{I}(K)$, which by Definition \ref{def:Tmap} determines a choice of $\Q$-linear map $T \colon \mathscr{L}(\A_{K}^{\times}/K^{\times})\to \prod_{p\leq \infty} \Q_p$. We warn the reader that this choice will not be reflected in our notation.

The goal is to construct a symmetric bilinear pairing $J(\overline{\Q})\times J(\overline{\Q})\to  \prod_{p\leq \infty} \Q_p$, using $T$. Projecting to $\Q_p$, for a chosen prime $p$, (or to $\R$) we will also obtain a $\Q_p$-valued (or real-valued) height function in the usual sense.

Besides the choice of $\chi\in \mathscr{I}(K)$, the adelic height depends on a choice, for each non-archimedean place $v$, of subspace $W_v$ of $H_{\dr}^1(X/K_v)$, complementary to the space of holomorphic forms and isotropic with respect to the cup product; in other words, on a subspace $W_v$ in the set $\Is(X/K_v)$. We denote by $\Is_K(X)$ the set of such choices; an element $W\in \Is_K(X)$ has a component $W_v$ at each finite place $v$ of $K$. Fix $W\in \Is_K(X)$; we can and will view $W$ as an element in $\Is_L(X)$ for any finite extension $L$ of $K$. 

Given a divisor $D$ of degree $0$ on $X/K$ and a non-archimedean place $v$, there is a unique differential $\omega_{D,W_v}$ of the third kind of type $(1,1)$ with respect to $W_v$ whose residue divisor is $D$ (Proposition \ref{prop:dec_third_kind}). Equivalently, by Lemma \ref{lemma:prop_LogJtilde}, $\omega_{D,W_v}$ is the unique differential of the third kind whose residue divisor is $D$ and whose image under $\Log_{\tilde{J}}$ belongs to $W_v$. We denote by $G_{D,W,v}$ a primitive of $\omega_{D,W_v}$, which, by Proposition \ref{prop:dec_third_kind}\thinspace{}\ref{it:dec_4}, is, up to a constant, of the form
\begin{equation*}
\sum_{i} n_{i} G_{\Theta,W_v}(\iota(x)-u_i),
\end{equation*}
for some generic $u_i\in J$ and $n_i\in \Z$.

If $v$ is archimedean, there exists a unique differential $\omega_{D,W_v}$ of the third kind whose residue divisor is $D$ and all of whose periods are purely imaginary (the notation is a bit misleading, because $W_{\infty}$ has technically not been defined). We write $G_{D,W,v}$ for the real part of a primitive of $\omega_{D,W_v}$. 

The \emph{adelic Green function} associated to $D$ with respect to $W$ is the function 
\begin{equation*}
G_{D,W}\colon \prod_v X(K_v) \to \prod_v \mathscr{L}(K_v^{\times}), \qquad G_{D,W}((\dots, x_v,\dots)) = (\dots, G_{D,W,v}(x_v),\dots).
\end{equation*}
The Green function $G_{D,W}$ is only well-defined on $(x_v)\in \prod_v X(K_v) $ such that $x_v$ does not belong to the support of $D$, for any $v$. Moreover, each $G_{D,W,v}$ is unique only up to addition of a constant function. This lack of uniqueness disappears if we view $G_{D,W}$ as a function on degree $0$ divisors, after extending by linearity.

\begin{prop}[{\hspace{1sp}\cite[Lemme II.2.16]{Colm96}}] \label{prop:prop_global_Green}
If $D_1$ and $D_2$ are disjoint degree $0$ divisors over $K$, then
\begin{enumerate}[label = (\roman*)]
\item\label{it:adG_1} $G_{D_1,W}(D_2) = G_{D_2,W}(D_1)$.
\item\label{it:adG_2} $G_{D_1,W}(D_2)\in \mathscr{L}(\A_K^{\times})$.
\item\label{it:adG_3} If $D_1 = \div(f)$, then $G_{D_1,W,v}(D_2) = \Log_v(f(D_2))$. In particular, if $D_1$ or $D_2$ is principal, then $G_{D_1,W}(D_2)\in \mathscr{L}(K^{\times})$. 
\item\label{it:adG_4} If $L/K$ is a finite extension, then $\Tr_{L/K}(G_{D_1\otimes L, W}(D_2\otimes L)) = [L:K]G_{D_1,W}(D_2)$. 
\end{enumerate}
\end{prop}
\begin{mydef}\label{def:pairingW}
The \emph{adelic height with respect to $T$} is the symmetric bilinear pairing
\begin{equation*}
\langle \cdot, \cdot\rangle_{W}\colon J(\overline{\Q})\times J(\overline{\Q})\to \prod_{p\leq \infty}\Q_p
\end{equation*}
defined as follows. Given $x,y\in J(\overline{\Q})$, pick $D_1,D_2\in \Div^0(X)$ such that the class of $D_1$ and $D_2$ equals $x$ and $y$, respectively. Assume that $D_1$ and $D_2$ are defined over $L$, for some $L/K$ of finite degree, and that the supports of $D_1$ and $D_2$ are disjoint. Then 
\begin{equation*}
\langle x,y \rangle_{W}\colonequals \frac{1}{[L:\Q]} \cdot T\circ \Tr_{L/K}(G_{D_1,W}(D_2)).
\end{equation*}
\end{mydef}

\begin{rmk}
Proposition \ref{prop:prop_global_Green} guarantees that $\langle x, y\rangle_W$ is symmetric and well-defined, in the sense that it is independent of the choices of $D_1$ and $D_2$ and of $L$. By symmetry, by the definition of $G_{D_1,W}$ and by $\Q$-linearity of $T\circ \Tr_{L/K}$, we get bilinearity of $\langle \cdot, \cdot \rangle_W$.
\end{rmk}
We also obtain, for each place $p$, a pairing
\begin{equation*}
\langle\cdot, \cdot \rangle_{W,p} \colon J(\overline{\Q})\times J(\overline{\Q})\to \Q_p
\end{equation*}
by taking the $p$-th component of $\langle\cdot, \cdot \rangle_W$. 
The $\Q_p$-valued pairing $\langle\cdot, \cdot \rangle_{W,p}$ decomposes as a sum of local height pairings. In particular, retaining the notation of Definition \ref{def:pairingW}, we have
\begin{equation*}
\langle x,y \rangle_{W,p} = \frac{1}{[L:\Q]}\sum_w \langle D_1,D_2\rangle_{p,w,W},
\end{equation*}
where $\langle \cdot ,\cdot \rangle_{p,w,W}$ is the $\Q_p$-valued pairing on relatively prime degree zero divisors with support in $X(L_w)$ defined as follows:
\begin{equation}\label{def:loc_ht_Colm}
\langle D_1,D_2\rangle_{p,w,W} = T_{v}^{(p)}\circ \Tr_{L_{w}/K_v}(G_{D_1,W,w}(D_2));
\end{equation}
here $v$ is the place of $K$ below $w$.

The real (resp.\ $p$-adic) height pairing of N\'eron--Tate (resp.\ Coleman--Gross) is also defined in terms of local height pairings on relatively prime degree zero divisors. 

\begin{prop}[{\hspace{1sp}\cite[Proposition 2.3]{Gross86}, \cite[Proposition 1.2]{ColemanGross}}]\label{prop:NT_pairing}
Let $p$ be a place of $\Q$ and let $\chi^{(p)}\colon \A_K^{\times}/K^{\times}\to \Q_p$ be a continuous idele class character. Let $L$ be a finite extension of $K$ and let $\chi_L^{(p)} = \chi^{(p)}\circ N_{L/K}\colon \A_L^{\times}/L^{\times}\to \Q_p$. Let $w$ be a place of $L$, such that, if $p \neq \infty$, then $w\nmid p$. There exists a unique $\Q_p$-valued pairing 
$\langle\cdot, \cdot \rangle_{p,w}^{N}$ on relatively prime degree zero divisors with support in $X(L_w)$ such that:
\begin{enumerate}[label=(\roman*)]
\item $\langle D_1,D_2\rangle_{p,w}^N = \langle D_2,D_1\rangle_{p,w}^N$;
\item $\langle D_1, D_2 + D_3\rangle_{p,w}^N = \langle D_1,D_2\rangle_{p,w}^N + \langle D_1,D_3\rangle_{p,w}^N$;
\item $\langle\div(f), D \rangle_{p,w}^N = \chi_{L,w}^{(p)}(f(D))$;
\item $\langle(x)-(x_0),D \rangle_{p,w}^N$ is a continuous function of $x$.
\end{enumerate}
\end{prop}

\begin{mydef}
The \emph{N\'eron--Tate real height pairing} with respect to the idele character $\chi^{(\infty)}$ is  the pairing
\begin{equation*}
\langle\cdot, \cdot \rangle^{NT}\colon J(\overline{\Q})\times J(\overline{\Q})\to \R
\end{equation*}
defined as follows. If $x,y, D_1,D_2, L$ are as in Definition \ref{def:pairingW}, then
\begin{equation*}
\langle x,y\rangle^{NT} = \frac{1}{[L:\Q]}\sum_w \langle D_1,D_2\rangle_{\infty,w}^{N}.
\end{equation*}
\end{mydef}
\begin{rmk}
Up to a multiplicative constant, there is only one choice of continuous idele class character $\chi^{(\infty)}$.
\end{rmk}
In the statement of Proposition \ref{prop:NT_pairing} we excluded pairs $(p,w)$ where $p$ is finite and $w\mid p$, because uniqueness fails in this case. Nevertheless, Coleman--Gross \cite{ColemanGross} defined a local pairing dependent on some choice of auxiliary data.  Let $v$ be the place of $K$ below $w$, let $W_v$ be a subspace of $H_{\dr}^1(X/K_v)$ isotropic with respect to the cup product and complementary to $H^0(X,\Omega^1)$. Unravelling definitions, we see that the local height of Coleman--Gross at $w$ satisfies  
\begin{equation}\label{eq:equality_colm_ht_CG_above_p}
\langle D_1,D_2\rangle_{p,w,W_v}^{N} = \langle D_1,D_2\rangle_{p,w,W},
\end{equation}
provided that the $v$-adic component of $W$ is $W_v$ and the $p$-adic component of $\chi$ is $\chi^{(p)}$. 
Technically speaking, Coleman--Gross assumed good reduction at each $v\mid p$ as they defined the left hand side of \eqref{eq:equality_colm_ht_CG_above_p} using Coleman integration \cite{Col82, ColemandeShalit}, but Besser \cite[Section 7]{Besser_pArakelov} \cite{BesserPairing} extended the definition to arbitrary reduction type, using the $p$-adic integration theory of Vologodsky \cite{Vologodsky}. Single Vologodsky integration agrees with Colmez's (summarised in \S \ref{subsec:Colm_int}). The only difference between (single) Coleman and Colmez integration in good reduction is that the former depends on the branch of the logarithm, while the latter treats the value of the logarithm at $p$ as a variable. This difference is lost in the local height parings of \eqref{eq:equality_colm_ht_CG_above_p} since we are applying compatible trace maps to $\Q_p$. The equality \eqref{eq:equality_colm_ht_CG_above_p} therefore indeed  holds.

In view of Besser's extension, we may state Colmez's results without restricting to good reduction. 

\begin{mydef}
For every $v\mid p$, let $W_v$ be a subspace of $H_{\dr}^1(X/ K_v)$, complementary to $H^0(X, \Omega^1)$ and isotropic with respect to the cup product. Then the \emph{(extended) Coleman--Gross $p$-adic height pairing} with respect to $(W_v)_{v\mid p}$ and $\chi^{(p)}$ is the pairing
\begin{equation*}
\langle\cdot, \cdot \rangle^{CG}_{(W_v)} \colon J(\overline{\Q})\times J(\overline{\Q})\to \Q_p
\end{equation*}
defined as follows. If $x,y,D_1,D_2,L$ are as in Definition \ref{def:pairingW}, then
\begin{equation*}
\langle x,y \rangle^{CG}_{(W_v)} =  \frac{1}{[L:\Q]}\biggl(\sum_{w\mid v\mid p}\langle D_1,D_2\rangle_{p,w,W_v}^N + \sum_{w<\infty,w\nmid p}\langle D_1,D_2 \rangle_{p,w}^{N} \biggr).
\end{equation*}
\end{mydef}

\begin{thm}[{\hspace{1sp}\cite[Théorème II.2.18]{Colm96} $+$ $\epsilon$}]\label{thm:NT_CG_eq_Colmez}
The pairing $\langle \cdot, \cdot \rangle_{W,p}$ is equal to the N\'eron--Tate real height pairing  with respect to $\chi^{(\infty)}$ if $p=\infty$ and to the extended $p$-adic height pairing of Coleman--Gross with respect to $(W_v)_{v\mid p}$ and $\chi^{(p)}$ if $p$ is finite. Moreover, if $p$ is infinite, or $p$ is finite and $w\nmid p$, then
\begin{equation*}
\langle \cdot, \cdot \rangle_{p,w,W} = \langle \cdot, \cdot\rangle_{p,w}^N,
\end{equation*}
and so $\langle \cdot, \cdot \rangle_{p,w,W}$ is independent of $W$; otherwise,
\begin{equation*}
\langle \cdot, \cdot \rangle_{p,w,W} = \langle \cdot, \cdot\rangle_{p,w,W_v}^N.
\end{equation*}
\end{thm}

\begin{proof}
Colmez proves the result for a specific choice of $\chi\in \mathscr{I}(K)$; given the definitions and results above, the proof carries over to an arbitrary $\chi$. In particular, it suffices to show equality of the local height pairings for all the pairs $(p,w)$ satisfying the assumptions of Proposition \ref{prop:NT_pairing}. For example, if $p, w$ are non-archimedean and $w\mid \ell\neq p$, by \eqref{def:loc_ht_Colm} and the proof of Lemma \ref{lemma:Tvp}, we have
\begin{equation}\label{eq:G_instead_of_int}
\langle D_1, D_2\rangle_{p,w,W} =G_{D_1,W,w}(D_2)^{(1)}\chi_{L,w}^{(p)}(\ell).
\end{equation}
By uniqueness, it suffices to show that \eqref{eq:G_instead_of_int} satisfies all the properties of Proposition \ref{prop:NT_pairing}, and this is independent of the choice of $\chi^{(p)}$.
\end{proof}

\begin{rmk}
As Colmez observes, by the comparison result of Theorem \ref{thm:NT_CG_eq_Colmez} and Equation \eqref{eq:G_instead_of_int}, we obtain a formula for the N\'eron--Tate local height which does not use intersection theory. This is analogous to the difference between Definition \ref{def:Neron_fct_away_p} and the definition of $p$-adic N\'eron functions away from $p$ of \cite{BKM22}. 

Recently, Besser, M\"uller and Srinivasan \cite{BMP_adelic_metric} gave a new construction of $p$-adic heights on varieties over number fields using $p$-adic Arakelov theory \cite{Besser_pArakelov}, and, in particular, $p$-adic log functions. Their height satisfies properties that generalise what we have just observed concerning Colmez's heights: namely, the local contributions at all places are defined analytically. In the special case of Jacobians of curves, they compare their construction to Colmez's. For this, they use a result of Besser \cite[Appendix B]{Besser_pArakelov} that relates the theory of log functions on line bundles on abelian varieties to Colmez's Green functions. 
\end{rmk}

\begin{rmk}
If $p$ is finite, $w\mid v\mid p$, and $\chi^{(p)}$ is unramified at $v$ (i.e.\ $\chi_v^{(p)}(\OO_v^{\times}) = 0$), then $\langle D_1, D_2\rangle_{p,w,W}$ only depends on $G_{D_1,W,w}(D_2)^{(1)}$ and not on $G_{D_1,W,w}(D_2)^{(0)}$, since the local trace map $t_v^{(p)}$ of the proof of Lemma \ref{lemma:Tvp} is trivial. In fact, also in this case the local height is independent of the choice of $W_v$, and is given by formula \eqref{eq:G_instead_of_int}.
Compare this with Remark \ref{rmk:neron_away_p}\thinspace{}\ref{rmk:neron_unramified}.
\end{rmk}

\subsection{$p$-Adic N\'eron functions, Green functions and Coleman--Gross heights}\label{subsec:comparison}
In this subsection, we go back to the situation of Section \ref{sec:padic_hts}, where $p$ is a finite prime of $\Q$ and $C$ is a smooth genus $2$ curve over a number field $K$, given by an equation of the form \eqref{eq:Grant}, with coefficients in the ring of integers of $K$. Let $v$ be a place of $K$, and consider the base-change of $C$ to $K_v$.

We apply Colmez's results from \S\S \ref{subsec:Colm_curves}--\ref{subsec:Colm_heights} according to Remark \ref{rmk:from_X_to_C}: in particular,
\begin{itemize}
\item $\iota$ is the embedding with respect to the base-point $P_0 = \infty$.
\item $(\Omega_1,\Omega_2)$ is the basis for $H^0(J,\Omega^1)$ of Definition \ref{def:basis_inv_diff_inv_der}, described explicitly in terms of the $\PP^8$ embedding of $J$ in Lemma \ref{lemma:inv_difder}\thinspace{}\eqref{lemma_part:inv_difder_1}. The dual basis $(\partial_1,\partial_2)$ of invariant derivations is also described in Lemma \ref{lemma:inv_difder}\thinspace{}\eqref{lemma_part:inv_difder_2}.
\item $x\in\Theta$ if and only if $-x\in \Theta$, so the element $w\in J$ defined just before Proposition \ref{prop:bijection} is just $0$. In particular, a $v$-adic Green function of division $\Theta$ is symmetric if and only if it is even. 
\end{itemize}
According to Proposition \ref{prop:bijection}, the set of symmetric $v$-adic Green functions of divisor $\Theta$ (up to addition by a constant function) is in bijection with the subset $\Is(C/K_v)$ of the set of subspaces of $H^1_{\dR}(C/ K_v)$. By Proposition \ref{prop:bijection_sigma_is}, there is also a bijection between $\Is(C/K_v)$ and the set of $v$-adic sigma functions. 

Suppose $\chi\colon \A_K^{\times}/K^{\times}\to \Q_p$ is a continuous idele class character. Let $T_v\colon \mathscr{L}(K_v^{\times}) \to \Q_p$ be the trace map at $v$ induced by $\chi$ as in Lemma \ref{lemma:Tvp}. Fix a subspace $W_v\in \Is(C/K_v)$. 
In this subsection, we prove:
\begin{thm}\label{thm:comparison} 
There exists a symmetric $v$-adic Green function $G\colonequals G_{\Theta,W_v}$ of divisor $\Theta$ simultaneously satisfying the following properties: 
\begin{enumerate}[label=(\roman*)]
 \item $\langle\iota^{*}[d(\partial_1 G)],\iota^{*}[d(\partial_2 G)] \rangle = W_v$; 
 \item Let $L$ be a finite extension of $K$, let $w\mid v$ be a place of $L$, let $\lambda_w\colonequals \lambda_{w,W_v}$ be the local $p$-adic N\'eron function at $w$ with respect to $W_v\in \Is(C/K_v)$ and $\chi$, and let $T_w\colonequals T_{v}\circ \Tr_{L_{w}/K_v}$ .  Then
 \begin{equation}\label{eq:eq_lambda_Gtheta}
 n_w\lambda_{w} = -2T_w(G).
 \end{equation}
\end{enumerate}
\end{thm}
\begin{rmk}
The theorem also holds true if $v$ is a prime at which $\chi$ is unramified; in this case, both the left and right hand side of \eqref{eq:eq_lambda_Gtheta} are independent of $W_v$ (cf.\ Definition \ref{def:Neron_fct_away_p}, Remark \ref{rmk:neron_away_p}\thinspace{}\ref{rmk:neron_unramified} and Theorem \ref{thm:NT_CG_eq_Colmez}).
\end{rmk}
The comparison of the $p$-adic component of the adelic-Green-function height and the (extended) Coleman--Gross $p$-adic height then gives, as a corollary, a comparison result between the local N\'eron function $\lambda_{w,W_v}$ and the local height pairing $\langle\cdot, \cdot \rangle_{p,w,W_v}^N$ with respect to $\chi$. 

\begin{cor}\label{cor:lambda_eq_CG} Let $L$ be a finite extension of $K$, and $w$ a place of $L$ above the place $v$ of $K$. Let $D_1 = \sum_{i=1}^{r} k_i P_i$ and $D_2 = \sum_{j=1}^{s} m_j Q_j$ be disjoint degree $0$ divisors supported on $C(L_w)$. Choose points $u_1,\dots,u_r\in J(L_w)$ such that $u_1,\dots,u_r$ are generic, $\sum_{i} k_i C_{u_i} = D_1$ and $\iota(Q_j)-u_i\not \in \Supp(\Theta)$ for all $i,j$. Then
\begin{align*}
\langle D_1, D_2 \rangle_{p,w,W_v}^N
= -\frac{n_w}{2}\sum_{i=1}^r \sum_{j=1}^{s}k_i m_j\lambda_{w,W_v}(\iota(Q_j) -u_i).
\end{align*}
\end{cor}
\begin{proof}
This follows from Theorem \ref{thm:comparison} and Theorem \ref{thm:NT_CG_eq_Colmez}.
\end{proof}
\begin{example}\label{eg:finding_u1_u2}
With the notation of Corollary \ref{cor:lambda_eq_CG}, 
let $R\in C(L_w)$ such that $R\neq P_i,P_i^{-}, Q_j$ for all $i\in\{1,\dots,r\}$ and all $j\in\{1,\dots,s\}$. Then we may choose  $u_i  = \iota(P_i) + \iota(R)$. 
\end{example}
\begin{rmk}
The comparison also holds at the primes at which $\chi$ is unramified. For these, there is no dependence on $W_v$.
\end{rmk}

\begin{cor}\label{cor:global_CG_same_as_this}
The Coleman--Gross global height pairing on $J(\overline{\Q})\times J(\overline{\Q})$ is equal to the symmetric bilinear pairing induced by $h_p$, provided that we make the same choices of $W_v\in \Is(C/K_v)$ at all places $v$ of ramification for $\chi$. 
\end{cor}
\begin{rmk}\label{rmk:comparison_canonical}
Suppose that $p\geq 5$, that $C$ has good ordinary reduction at every place $v\mid p$ of ramification for $\chi$, and that, for each such $v$, $K_v$ is an unramified extension of $\Q_p$. Then, at every such $v$, we may take $W_v$ to be the unit root eigenspace of Frobenius. By Proposition \ref{prop:Blakestad_space_unit_root}, this corresponds to choosing the canonical $v$-adic sigma function of Blakestad. Under these assumptions, Corollary \ref{cor:global_CG_same_as_this} thus relates the Coleman--Gross $p$-adic height with respect to the unit root eigenspaces of Frobenius to the $p$-adic height with respect to the canonical $v$-adic sigma functions. 
\end{rmk}
\begin{proof}[Proof of Corollary \ref{cor:global_CG_same_as_this}]
We omit the choices of subspace $W_v$ from the notation. Let $L$ be an arbitrary finite extension of $K$. Since $C(L)$ is finite by Faltings's theorem \cite{faltings}, there are finitely many points in $J(L)$ which are either of the form $[2P-2\infty]$ or $[P+W-2\infty]$ for a Weierstrass point $W$.  Therefore, any point in $J(L)\setminus J_{\tors}(L)$ admits a multiple $v$ of the form
\begin{equation*}
v = [P_1 - P_2],
\end{equation*}
for some non-Weierstrass points $P_1,P_2$ such that $P_1\neq P_2^{-}$. Thus, by quadraticity, it suffices to show that $h_p(v) = \langle v, v\rangle^{CG}$ for all $v$ in this form.

 In this case, the divisors $D_1 = P_1 - P_2$ and $D_2 = P_2^{-} - P_1^{-}$ are disjoint and satisfy $[D_1] = [D_2] = v\not\in \Supp(\Theta)$. By possibly replacing $L$ with a field extension, we may assume that the supports of $D_1$ and $D_2$ are defined over $L$. Let $w$ be a non-archimedean place of $L$, and let $v_1 =[P_2^{-} -P_1] $, $v_2 = [P_2^{-}-P_2]$, $v_3 = [P_1^{-}-P_1]$. By Corollary \ref{cor:lambda_eq_CG} and Example \ref{eg:finding_u1_u2}, we have
\begin{equation*}
\langle D_1, D_2\rangle_{p,w}^N = - \frac{n_w}{2} \left(2\lambda_w(v_1) - \lambda_w(v_2) - \lambda_w(v_3) \right).
\end{equation*}
Thus, by Proposition \ref{prop:properties_neron_fcts}\thinspace{}\ref{it:quasi_parallelogram}, 
\begin{equation*}
\langle D_1, D_2\rangle_{p,w}^N =  n_w\lambda_w(v) - \chi_{L,w}(-X_{11}(v_1) + X_{11}(v) - X_{12}(v_1)X_{22}(v) + X_{22}(v_1)X_{12}(v)).
\end{equation*}
Summing over all places $w$ gives the result.
\end{proof}

It remains to prove Theorem \ref{thm:comparison}.
 This will be a corollary of some intermediate results. In order to keep the notation more legible, for now, let $\sigma_v$ be just some $v$-adic sigma function for $C/K_v$.
\begin{lemma}\label{lemma:log_sigma_loc_Green}
Let $V$ be a finite index subgroup of $J(K_v)$ on which $\sigma_v$ converges. Then there exists a finite index subgroup $V^{\prime}$ of $J(K_v)$ such that $V^{\prime}\subseteq V$ and the restriction of $\Log(\sigma_v)$ to $V^{\prime}$ satisfies Proposition \ref{prop:padicGreen} for $D = \Theta$; that is, 
\begin{equation*}
\Delta^{[3]}(\Log(\sigma_v)) = \Log g_{\Theta}^{(4)},
\end{equation*}
where $g_{\Theta}^{(4)}$ is the restriction to $V^{\prime, 4}$ of $f_{\Theta}^{(4)}$, $\Log(\sigma_v)$ is locally analytic outside $\Theta$ and has a logarithmic singularity along $\Theta$. 
\end{lemma}
\begin{proof}
If $\sigma_v$ and $\sigma_v^{\prime}$ are two different sigma functions, then $\Log(\sigma_v) - \Log(\sigma_v^{\prime})$ is a polynomial of degree two in the formal one-logarithms of $J$, so, since the statement of the lemma allows us to move to a smaller $V^{\prime}$, it suffices to consider the naive $v$-adic sigma function of Theorem \ref{thm:sigma_naive}. Upon switching variables to the additive group, the latter is the Taylor expansion of the complex sigma function $\sigma$, so it is enough to show that 
\begin{enumerate}[label = (\roman*)]
\item \label{it:periodic_sigma} the function $\sigma$ is such that
\begin{equation*}
F_{\Theta}^{(4)} (z_0,z_1,z_2,z_3) = \frac{\sigma(z_0 +z_1 + z_2+z_3) \sigma(z_0 + z_1) \sigma(z_0 + z_2)\sigma(z_0 + z_3)}{\sigma(z_0 + z_1 + z_2)\sigma(z_0 + z_1 + z_3)\sigma(z_0 + z_2+z_3)\sigma(z_0)} 
\end{equation*}
is periodic with period $\Lambda^4$;
\item \label{it:divisor_sigma} the divisor of $F_{\Theta}^{(4)}$ is $\Delta^{[3]}\Theta$. 
\end{enumerate}
Part \ref{it:periodic_sigma} follows from the translation properties of $\sigma$ \cite[Proposition 2.2]{Uchida}, while part \ref{it:divisor_sigma} by \cite[Proposition 2.7]{Uchida}.
Note that on $J^3\times\{0\}\cup (J^2\times\{0\}\times J)\cup (J\times \{0\}\times J^2)$ we have $F_{\Theta}^{(4)} = 1$. 
\end{proof}

Lemma \ref{lemma:log_sigma_loc_Green} tells us that, locally, $\Log(\sigma_v)$ is a $v$-adic Green function of divisor $\Theta$. Next, we want to show that we can extend $\Log(\sigma_v)$ to a $v$-adic Green function on $J$, and that the right way to do so is by using a quasi-quadraticity formula involving division polynomials, just like we did when defining local N\'eron functions. 

In order to achieve this, we observe that the existence of $G_{\Theta}$ (and uniqueness, up to addition by polynomials of degree at most $2$ in the one-logarithms of $J$), as stated in Proposition \ref{prop:padicGreen}, is proved by Colmez using \cite[Th\'eor\`eme II.1.16]{Colm96}. The proof of this technical result comprises two steps: first one solves the problem locally, and then extends the solution using \cite[Proposition II.1.23]{Colm96}.  

\begin{prop}\label{prop:sigma_extends}
The function $\Log(\sigma_v)$ extends to a unique Green function $G_{\Theta}$ for $\Theta$ on $J$. Moreover, $G_{\Theta}$ is an even function and 
\begin{equation}\label{eq:G_quasi_quad}
G_{\Theta}(nx) - n^2 G_{\Theta}(x) = \Log (\phi_n(x)),
\end{equation}
where $\phi_n$ is the $n$-th division polynomial of \eqref{eq:def_div_poly}.
\end{prop}

\begin{proof}
By \cite[Proposition II.1.23]{Colm96}, $\Log(\sigma)$ extends to a unique Green function $G_{\Theta}$ on $J$. 
Now, $\Log(\sigma)$ is an even function, and since $[-1]^{*}\Theta = \Theta$ we have that $\Log f_{\Theta}^{(4)}(-z_0,-z_1,-z_2,-z_3)$  and $\Log f_{\Theta}^{(4)}(z_0, z_1,z_2,z_3)$ differ by a constant. Evaluating at $z_3 = 0$, we find that the constant is $0$. Thus the formula of \cite[Proposition II.1.23]{Colm96} with $f = \Log(\sigma)$ gives an even function, satisfying
\begin{align}\label{eq:def_green}
G_{\Theta}(z_0 + z_1 + z_2 + z_3) - G_{\Theta}(z_0+ z_1+ z_2) - G_{\Theta}(z_0+ z_2+ z_3) - G_{\Theta}(z_0+ z_1+ z_3) \\
+ G_{\Theta}(z_0 + z_1) + G_{\Theta}(z_0+ z_2) + G_{\Theta}(z_0+ z_3) - G_{\Theta}(z_0) = \Delta^{[3]} G_{\Theta}(z_0,z_1,z_2,z_3) = \Log(f_{\Theta}^{(4)})\nonumber
\end{align}
We now use this to prove \eqref{eq:G_quasi_quad}. The proof will be similar to the one for elliptic curves given in the proof of \cite[Proposition II.2.20]{Colm96}.

Both the right and left hand side of \eqref{eq:def_green} have a logarithmic singularity at $\{0\}\times J^3$. We have
\begin{equation}\label{eq:Gz1z2z3}
\begin{aligned}
G_{\Theta}( z_1 + z_2 + z_3) - G_{\Theta}(z_1+ z_2) - G_{\Theta}(z_2 + z_3) - G_{\Theta}(z_1 + z_3) \\
+ G_{\Theta}(z_1) + G_{\Theta}( z_2) + G_{\Theta}( z_3)  = \lim_{z_0\to 0}(  \Log(f_\Theta^{(4)})(z_0,z_1,z_2,z_3)+ G_{\Theta}(z_0))= \Log(h(z_1,z_2,z_3)),
\end{aligned}
\end{equation}
for some rational function $h$. Once again, letting $z_1 = z_2$, taking limits $z_3\to  -z_1$ and using that $G_{\Theta}$ is even, we get
\begin{equation}\label{eq:G2z1}
4G_{\Theta}(z_1)-G_{\Theta}(2z_1) = \Log(f_2^{-1}(z_1))
\end{equation}
for some rational function $f_2$. Going back to \eqref{eq:Gz1z2z3} and letting $z_1 = (n-1)x, z_2 = z_3 = x$, we get
\begin{equation*}
G_{\Theta}((n+1)x) - 2G_{\Theta}(nx) - G_{\Theta}(2x) + G_{\Theta}((n-1)x) + 2G_{\Theta}(x) = \Log(h((n-1)x,x,x)).
\end{equation*}
By induction, we conclude that for all $n\geq 1$ there exists a rational function $f_n$ such that
\begin{equation*}
G_{\Theta}(nx) - n^2G_{\Theta}(x) = \Log(f_n(x)).
\end{equation*}
Comparing log-singularities, we see that $f_n$ equals $\phi_n$, up to multiplication by a constant, which, by comparing around $0$, we see we can take to be equal to $1$.
\end{proof}

It remains to show that if $\sigma_v$ corresponds to $W_v$ via Proposition \ref{prop:bijection_sigma_is}, then the Green function $G_{\Theta}$ of Proposition \ref{prop:sigma_extends} also corresponds to $W_v$ under the bijection of Proposition \ref{prop:bijection}.
\begin{lemma}\label{lemma:explicit_subspace_H1dRJ}
Let $\sigma_v$ be the sigma function satisfying the differential equations
\begin{equation*}
D_iD_j(\Log(\sigma_v))= - X_{ij} + c_{ij}, \qquad c_{12} = c_{21}.
\end{equation*}
Let $G_{\Theta}$ be the Green function associated to $\sigma_v$ by Proposition \ref{prop:sigma_extends}. 
Then $\eta_{i,J} =  d(\partial_i G_{\Theta})$ is given explicitly by 
\begin{equation*}
\eta_{i,J} = (-X_{1i} + c_{1i}) \Omega_1 + (-X_{i2} + c_{i2})\Omega_2.
\end{equation*}
\end{lemma}
\begin{proof}
By Lemma \ref{lemma:dLi_eq_Omegai}, $D_i$ is the restriction of $\partial_i$ to the formal group law.  So we have
\begin{align*}
\partial_i\partial_j(\Log(\sigma_v(nx))) = n^2(-X_{ij}(nx)  + c_{ij})\\
\partial_i \partial_j(\Log(\phi_n(x))) = -n^2X_{ij}(nx) + n^2X_{ij}(x)
\end{align*}
(cf.\ \cite[Proposition 4.10]{Uchida}), and thus $\partial_i\partial_j(G_{\Theta}) = -X_{ij} + c_{ij}$, as desired.
\end{proof}

\begin{lemma}\label{lemma:explicit_subspace_H1dRX}
With the notation of Lemma \ref{lemma:explicit_subspace_H1dRJ}, let $[\eta_1]  =\iota^*[\eta_{1,J}], [\eta_2] = \iota^{*}[\eta_{2,J}]$. Then we may choose
\begin{equation*}
\eta_1 = (-3x^3 - 2b_1x^2 - b_2x + c_{12} x + c_{11})\frac{dx}{2y},\qquad \eta_2= (-x^2 + c_{22} x + c_{12})\frac{dx}{2y},
\end{equation*}
i.e.\ $\eta_1 = \eta_1^{(c)}$ and $\eta_2 = \eta_2^{(c)}$ where $\eta_1^{(c)}$ and $\eta_2^{(c)}$ are the differentials of \eqref{eq:eta_12_c}.
\end{lemma}

\begin{proof}
Let $\pi_i\colon C^2\to C$ be the projection map, and let $\varpi\colon C^2\to J$ be the map sending $(P_1,P_2)$ to $\iota(P_1)+ \iota(P_2)$. According to \cite[Lemme II.2.1]{Colm96}, given a differential $\omega$ of the second kind on $C$, there exists a unique differential $\omega_J$ of the second kind on $J$ such that $\varpi^{*}\omega_J = \pi_1^{*}\omega + \pi_2^{*}\omega$. Moreover the class of $\iota^{*}\omega_J$ equals that of $\omega$ in $H^1_{\dR}(C)$. 
We deduce that any differential $\xi$ of the second kind on $J$ is equivalent to one whose image under $\varpi^{*}$ is of the form
\begin{equation}\label{eq:pullback}
\pi_1^*\omega + \pi_2^*\omega,
\end{equation}
and, moreover, that $[\iota^* \xi] = [\omega]$. Therefore, in order to compute the classes of $[\eta_i]$, it will suffice to find a representative for $\varpi^*[\eta_{i,J}]$ of the form \eqref{eq:pullback}. By \cite[p.\,66]{blakestadsthesis}, we have
\begin{equation*}
d(X_{222}) = (3X_{12}X_{22} - X_{11} + 2b_1X_{12}) \Omega_1 + (3X_{22}^2 + 2X_{12} + 2b_1X_{22} + b_2) \Omega_2;
\end{equation*}
moreover, by definition, $\varpi^{*}\Omega_1= \frac{dx_1}{2y_1} + \frac{dx_2}{2y_2}$, $\varpi^{*}\Omega_2 = x_1\frac{dx_1}{2y_1} + x_2\frac{dx_2}{2y_2}$ and $\varpi^{*}X_{12} = -x_1x_2$, $\varpi^{*}X_{22} =x_1+x_2$ . A computation then shows that
\begin{equation}\label{eq:eta_iJ_eta_i}
\begin{aligned}
\varpi^{*}(\eta_{1,J}  - d(X_{222})) &= \pi_1^{*}\eta_1 + \pi_2^{*}\eta_1 \\
 \varpi^{*}(\eta_{2,J}) &=\pi_1^{*}\eta_2 + \pi_2^{*}\eta_2. 
 \end{aligned} \qedhere
\end{equation}
\end{proof}

\begin{proof}[Proof of Theorem \ref{thm:comparison}]
This follows from Lemma \ref{lemma:Tvp}, Remark \ref{rmk:trace_ext}, Proposition \ref{prop:sigma_extends} and Lemmas \ref{lemma:explicit_subspace_H1dRJ}, \ref{lemma:explicit_subspace_H1dRX}.
\end{proof}

\section{Implementation and examples}\label{sec:implementation}
\subsection{Formal group and sigma functions}\label{subsec:implementation_formal_gp_sigma}
Let $C$, $R$ and $K$ be as in \S \ref{subsec:formal}. Recall that the functions $\frac{1}{X_{111}}$, $\frac{X_{ij}}{X_{111}}$ and $\frac{X_{ijk}}{X_{111}}$ can be expanded as power series over $R$ in the formal group parameters $T_1$ and $T_2$. Grant proved this by explaining how to use the equations defining $J$ in $\PP^8$ to obtain recursive formulae for the coefficients in the expansion (cf.\ \eqref{eq:Xij_exp}, \eqref{eq:Xijk_exp} and the surrounding discussion).

We implemented this recursion (in conjunction with Remark \ref{rmk:odd_and_even_expansions}). As a result, we can compute the expansions of $\frac{1}{X_{111}}$, $\frac{X_{ij}}{X_{111}}$ and $\frac{X_{ijk}}{X_{111}}$ as power series in $T_1$ and $T_2$, up to any desired precision.  Combining this with the addition formulae on $J$ provided by \cite[Theorem 3.2]{Grant1990}, we can then compute the formal group law $F(T,S) = (F_1,F_2)(T,S)$ of Theorem \ref{thm:formal_gp_law} explicitly, up to any desired precision.

With the formal group law at hand, the computation of the strict logarithm $\mathcal{L} = (\mathcal{L}_1,\mathcal{L}_2)$ and of the naive sigma function are an easy exercise in solving systems of differential equations over a power series ring.  Note that \eqref{eq:Di} expresses the derivation $D_i$, for each $i\in\{1,2\}$, in terms of the derivatives with respect to $T_1$ and $T_2$ and the formal group law. 

Given the naive sigma function and the strict logarithm $\mathcal{L}$, we can compute the expansion of any other sigma function using the formula \eqref{eq:formulasigmac}. We included in Appendix \ref{app:sigma_expansion} the first terms in the expansion of $\sigma_v^{(c)}(T)$ for an arbitrary curve $C$ and for arbitrary constants $c_{11},c_{12},c_{22}$ (so, in particular, the coefficients belong to $\Q[b_1,\dots, b_5][c_{11},c_{12},c_{22}]$).

\subsection{The canonical sigma function}\label{subsec:implementation_canonical_sigma}
Assume now that $K$ is the completion of a number field at a non-archimedean place $v$, and that the assumptions of Theorem \ref{thm:Blakestad_main} are satisfied. We would like to compute the canonical $v$-adic sigma function of Blakestad. By Proposition \ref{prop:Blakestad_space_unit_root}, if $K$ is an unramified extension of $\Q_p$, this is the $v$-adic sigma function corresponding to the unit root eigenspace of Frobenius. While it is possible to compute this space using Theorem \ref{thm:Blakestad_main}\thinspace{}\ref{thm:Blake_3}, it is computationally very inefficient. 

If $K$ is isomorphic to $\Q_p$,  for a prime $p$ greater than $3$, we can alternatively compute the unit root eigenspace of Frobenius using Kedlaya's algorithm \cite{kedlaya} and \cite[Proposition 6.1]{Bes-Bal10}. We explain how to do this in practice and how to deduce the symmetric matrix $b$ corresponding to the unit root eigenspace.

So assume now that $C$ is defined over $\Q_p$, where $p$ is a prime greater than or equal to $5$, of good reduction for $C$ and ordinary reduction for $J$. Consider the characteristic polynomial $\chi(t)$ of Frobenius of the Jacobian of the base-change of $C$ to $\F_p$, which is of the form 
\begin{equation*}
\chi(t) = t^4 + a_1t^{3}+a_2t^2+pa_1t+p^2\in \Z[t].
\end{equation*}
The Jacobian over $\F_p$ is ordinary if and only if $a_2$ is coprime with $p$ (see for instance \cite[\S 3]{yui1978jacobian}). In this case, the factorisation  
\begin{equation*}
\chi(t) \equiv t^4 +a_1t^3+a_2t^2 \bmod{p} = t^2(t^2+a_1t+a_2)
\end{equation*}
 lifts to a factorisation over $\Z_p$ by Hensel's lemma. Let $\chi_{W^{(b)}}(t) = t^2+\hat{a}_1t + \hat{a}_2\in \Z_p[t]$ be the lift of $t^2+a_1t+a_2 \bmod{p}$ in such a factorisation of $\chi(t)$.
  
 Let $\mathcal{C}/\Z_p$ be the hyperelliptic curve over $\Z_p$ whose complement of the section at infinity is the closed subscheme of $\mathbb{A}_{\Z_p}^2$ described by our usual affine equation for $C$. Recall that we have
 \begin{equation*}
H^1_{\dR}(\mathcal{C}/\Z_p)\hookrightarrow{} H^1_{\dR}(\mathcal{C}/\Z_p)\otimes \Q_p \cong H^1_{\dR}(C/\Q_p),
 \end{equation*}
 and that this induces a $\Q_p$-linear Frobenius endomorphism on $H^1_{\dR}(C/\Q_p)$. The characteristic polynomial of this endomorphism is equal to $\chi(t)$ (see, for instance,  \cite[Theorem 5.3.2]{Edix_pt_counting} combined with Remark \ref{rmk:compatibility_frob} below). 
The unit root eigenspace of Frobenius is the $2$-dimensional Frobenius-invariant $\Q_p$-vector subspace $W^{(b)}$ of $H^1_{\dR}(C/\Q_p)$ on which Frobenius acts with characteristic polynomial $\chi_{W^{(b)}}(t)$.
 \begin{rmk}\label{rmk:compatibility_frob}
 Technically speaking, Kedlaya's algorithm computes the Frobenius action on the Monsky--Washnitzer cohomology $H_{\MW}^1(\tilde{C}^{\prime}/\F_p, \Q_p)$ of the curve $\tilde{C}^{\prime}$ obtained from the reduction of $C$ modulo $p$ by removing the Weierstrass points. Since the odd part of $H_{\MW}^1(\tilde{C}^{\prime}/\F_p,\Q_p)$ is isomorphic to $H^1_{\dR}(C/\Q_p)$ and the isomorphism is compatible with the Frobenius action on $H^1_{\dR}(C/\Q_p)$ introduced in \S \ref{subsec:infty_sigma}, we will ignore this subtlety. For details, see for instance \cite[Theorem 2.6]{Bogaart}.
 \end{rmk}
 In order to compute the differentials $\eta_1^{(b)}$ and $\eta_2^{(b)}$ of Theorem \ref{thm:Blakestad_main}, we work directly with $H^1_{\dR}(\mathcal{C}/\Z_p)$ and its $\Z_p$-linear Frobenius action. As we saw in the proof of Proposition \ref{prop:Blakestad_space_unit_root}, under our running assumptions, this $\Z_p$-module is isomorphic to $H^0(\mathcal{C},\Omega^1_{\mathcal{C}/\Z_p}(4\infty))^{-}$ and the latter is freely generated by 
 \begin{equation}\label{eq:mathcalB}
\mathcal{B} = \left\{\frac{x^i dx}{2y}\colon i =0,\dots, 3\right\}.
\end{equation}
From now on, we identify $H^1_{\dR}(\mathcal{C}/\Z_p)$ with the $\Z_p$-span of $\mathcal{B}$. The action of Frobenius, denoted by $\phi^{*}$, on $H^1_{\dR}(\mathcal{C}/\Z_p)$ with respect to $\mathcal{B}$ can be computed using Kedlaya's algorithm \cite{kedlaya}.

The $\Q_p$-vector space $W^{(b)}$ is obtained by tensoring with $\Q_p$ the $\Z_p$-submodule $W^{(b)}_0\subset H^1_{\dR}(\mathcal{C}/\Z_p)$ defined as the kernel of the expansion at $P$-map $\beta_P\colon H^1_{\dR}(\mathcal{C}/\Z_p)\to H^1_{\dR}(\hat{\mathcal{C}}_{P}/\Z_p)$, for any $P\in \mathcal{C}(\Z_p)$. The differentials $\eta_1^{(b)}$, $\eta_2^{(b)}$ belong to $W^{(b)}_0$ and, in fact, freely generate it as a $\Z_p$-module. Moreover, $H^1_{\dR}(\mathcal{C}/\Z_p)$ is isomorphic to $W^{(b)}_0 \oplus \langle\frac{dx}{2y}, \frac{xdx}{2y}\rangle_{\Zp}$.

Now, 
\begin{equation}\label{eq:holo_pH1}
\phi^{*}\left(\frac{dx}{2y}\right), \phi^{*}\left(\frac{xdx}{2y}\right)\in pH^1_{\dR}(\mathcal{C}/\Z_p);
\end{equation}
see, for instance, \cite[Proof of Lemma 3.4]{harrison}.
In addition, since the determinant of the matrix of Frobenius on any basis for $W^{(b)}_0$ is equal to $\tilde{a}_2\in \Z_p^{\times}$, we have $\phi^{*}(W^{(b)}_0) = W^{(b)}_0$ by \cite[XIII, Propositions 3.1 and 4.16]{Lang:algebra}. From these considerations, one deduces the following algorithm to compute a basis for $W^{(b)}_0$ modulo a prescribed precision. Note that essentially the same result is stated in higher generality in \cite{Bes-Bal10}; however, no distinction seems to be made there between the $\Q_p$-vector space $W^{(b)}$ and the $\Z_p$-module $W^{(b)}_0$.
\begin{prop}[{\hspace{1sp}\cite[Proposition 6.1]{Bes-Bal10}}]\label{cor:basis_using_Kedlaya}
Let $M\in M_4(\Z_p)$ be the matrix of Frobenius with respect to $\mathcal{B}$, and let $n$ be a positive integer. There is a basis for $W^{(b)}_0$ whose coefficients with respect to $\mathcal{B}$ reduce modulo $p^n$ to the third and fourth column of $M^n$. 
\end{prop}
Using Corollary \ref{cor:basis_using_Kedlaya}, we can compute a basis for $W^{(b)}_0$ modulo $p^n$, from which we can deduce $\eta_1^{(b)}$, $\eta_2^{(b)}$ modulo $p^n$.

\begin{rmk}
If $p=3$, the set $\mathcal{B}$ is not a basis for $H^1_{\dR}(\mathcal{C}/\Z_p)$ (see \cite[\S 5]{Bogaart}). 
However, with some extra care, it should be possible to compute the unit root eigenspace in a similar way and to relate this to the canonical sigma function of Blakestad (see Remark \ref{rmk:inver_H1}).
\end{rmk}

\subsection{Division polynomials} \label{subsec:division_polynomials}
In the previous two subsections, we discussed how to compute the expansion of various sigma functions. Recall that we used division polynomials to extend a sigma function outside of its domain of convergence.

In fact, if we want to compute the value of a local N\'eron function using Definition \ref{def:Neron_fct_above_p} or \ref{def:Neron_fct_away_p}, we do not need to compute the division polynomial, but just its value at the point of interest. \texttt{Magma} \cite{magma} code to compute values of division polynomials is provided with\cite{muller_de_jong}. We translated this into \texttt{SageMath}, with some minor modifications (see Remark \ref{rmk:code_div_poly_below}).

For the purpose of this problem, we may more generally assume that $C$ is defined over any characteristic $0$ field. Let $P$ be a point on the Jacobian, away from the theta divisor. The strategy is as follows:
\begin{enumerate}
\item For $1\leq n\leq 5$, Uchida computed $\phi_n$ as a polynomial in the variables $\wp_{ij}$ and $\wp_{ijk}$ with coefficients in $\Z[b_1,\dots,b_5]$. Evaluating at $\wp_{ij}(P),\wp_{ijk}(P)$ gives $\phi_n(P)$ for $1\leq n\leq 5$. 
\item \label{it:div_poly_2} Compute $\phi_n(P)$ for $6\leq n\leq 8$ using the recurrence relation \cite[Theorem 9 (corrected)]{Kanayama_corrections}. 
 Since the computation of $\phi_6(P)$ and $\phi_8(P)$ requires division by $\phi_2(P)$, we must assume that $2P\not\in \Supp(\Theta)$.
\item \label{it:div_poly_3} Compute $\phi_n(P)$ recursively for $n\geq 8$ using \cite[Example 6.6]{Uchida}. This requires division by
\begin{equation*}
\phi_5(P) - \phi_4(P)\phi_2(P)^3+\phi_3(P)^3,
\end{equation*}
so the formulae can only be applied when this is non-zero. For $n$ even we further need that $\phi_2$ does not vanish at $P$. 
\end{enumerate}
For certain values of $n$, the code by de Jong--M\"uller is not directly applicable if either $\phi_2$ or $\phi_5-\phi_4\phi_2^3 + \phi_3^3$ vanishes at our point $P$ of interest. If we are only interested in the value at $P$ of the global height $h_p$, we may remedy the situation by computing $h_p(mP)$ for a non-zero integer $m$, and deducing the value $h_p(P)$ by quadraticity. Otherwise, we may compute $\phi_n(P)$ by taking limits over points approaching $P$ of the recurrence relations of Steps \eqref{it:div_poly_2} and \eqref{it:div_poly_3}.

\begin{rmk}\label{rmk:code_div_poly_below}
The recursion formulae of Step \eqref{it:div_poly_3} express $\phi_{2m}$ and $\phi_{2m+1}$ in terms of $\phi_k$ for $k\in\{2,3,4,5\}\cup \{m-4,\dots, m+4\}$. In particular, in order to compute $\phi_n$, it is not necessary to compute $\phi_k$ for all $k\leq n-1$; rather, it suffices to compute $\phi_k$ for $k$ in some intervals. We implemented this improvement. The same trick is used in Harvey's implementation \cite{harvey} on \texttt{SageMath} of $p$-adic heights on elliptic curves.
\end{rmk}

\subsection{$p$-adic N\'eron functions}\label{subsec:implementation_Neron_fcts}
In order to compute the $p$-adic N\'eron functions of Definitions \ref{def:Neron_fct_above_p} and \ref{def:Neron_fct_away_p}, it remains to discuss how to compute a suitable integer $m$ and how to work with an idele class character for a number field $K$ explicitly. 

The latter problem is discussed in detail in \cite[\S 2.1]{QCnfs}. For the purpose of this work we restrict the implementation to the case $K=\Q$, for which the $\Q_p$-vector space of continuous idele class characters is one-dimensional, and generated by the cyclotomic character $\chi^{\cyc}$ (cf.\ \cite[Example 2.7]{QCnfs}). Explicitly, we have
\begin{equation*}
\chi^{\cyc}_{p}(x) = \log_p(x) \qquad \text{for all } x\in \Q_p^{\times}, 
\end{equation*}
where $\log_p$ is the branch of the $p$-adic logarithm that vanishes at $p$, and, for $q\neq p$,
\begin{equation*}
\chi^{\cyc}_q(x) = \log_p|x|_q, \qquad \text{for all } x\in \Q_q^{\times}.
\end{equation*}
The $p$-adic sigma functions that we are mostly interested in are the naive one and the canonical one (when this exists). If $p$ is odd, they both converge on the whole of $J_1(\Q_p)$. Indeed, the canonical sigma function has integral coefficients, while the convergence of the naive sigma function is given by Theorem \ref{thm:sigma_naive}\thinspace{}\ref{thm_sigma:part2}.

Therefore, given $P\in J(\Q_p)$, we need to compute (a multiple of) its order in $J(\Q_p)/J_1(\Q_p)$. If $p$ is a prime of good reduction, this is the order of $\tilde{P}$ in the finite group $\tilde{J}(\F_p)$. 

If $p$ is an odd prime of bad reduction, denote by $J_0(\Q_p)$ the subgroup of $J(\Q_p)$ consisting of the points that reduce to the component of the origin in the smooth part of  $\tilde{J}$.
We can apply \cite[Remark 5.12, Figure 1]{Bruin-Stoll:MWSieve} to determine the index of $J_1(\Q_p)$ inside $J_0(\Q_p)$, and the discussion on p.\,294--295 of \emph{loc.\ cit.}\ to determine a multiple of $P$ that belongs to $J_0(\Q_p)$. Since it is easy to check whether a point belongs to $J_1(\Q_p)$, we can alternatively compute successive multiples of $ [J_0(\Q_p):J_1(\Q_p)]P$, until we land in $J_1(\Q_p)$. A trial and error strategy can also be applied when $p=2$.

\subsection{Example: A prime greater than $10^{6}$}\label{subsec:eg_large_p}
We now consider the genus $2$ curve over $\Q$ with LMFDB label \href{http://www.lmfdb.org/Genus2Curve/Q/160000/c/800000/1}{160000.c.800000.1} \cite{lmfdb}
\begin{equation*}
C\colon y^2 = x^5 - 1.
\end{equation*}
Its Jacobian $J$ is geometrically simple and has complex multiplication by $\Q(\zeta_5)$, where $\zeta_5$ is a primitive fifth root of unity. We shall see that every $p\equiv 1\bmod{10}$ is of good ordinary reduction and that the constants $c_{ij}$ corresponding to the unit root eigenspace for $p$ are all equal to zero. In particular, they are algebraic and independent of $p$: compare with Remark \ref{rmk:CM_case}. This enables us to compute the canonical $p$-adic N\'eron function at $p$ for large primes without using Proposition \ref{cor:basis_using_Kedlaya} explicitly. 

\begin{lemma}[{\hspace{1sp}\cite{yui1978jacobian}}]
Let $p$ be a prime congruent to $1$ modulo $10$. Then $C$ has good reduction at $p$, and $J$ has good ordinary reduction at $p$.
\end{lemma}
\begin{proof}
The only primes at which the given equation for $C$ has bad reduction are $2$ and $5$. For ordinarity, it suffices to show that the Cartier--Manin matrix of $C$ modulo $p$ has non-zero determinant: see \cite[Example 3.3]{yui1978jacobian}.
\end{proof}

\begin{lemma}
Let $p$ be a prime congruent to $1$ modulo $10$. Then the classes of the differentials 
\begin{equation*}
\eta_1^{(0)} = -3x^3\frac{dx}{2y}, \qquad \eta_2^{(0)} = -x^2\frac{dx}{2y}
\end{equation*}
span the unit root eigenspace of Frobenius. In particular, the canonical sigma function is equal to the naive sigma function.
\end{lemma}
\begin{proof}
First note that $p$ splits completely in $\Q(\zeta_5)$. The $\Q(\zeta_5)$-automorphism $\psi\colon (x,y)\mapsto (\zeta_5 x, y)$ of the curve $C$ reduces to an automorphism $\tilde{\phi}$ over $\F_p$, and this commutes with the $p$-th power Frobenius endomorphism $\pi$. Therefore, the lifts of $\pi\circ \tilde{\psi}$ and $\tilde{\psi}\circ \pi$ induce homotopic maps on differentials \cite[Theorem (2.4.4) (iii)]{vanderPut} and hence the same map in cohomology. In other words, for every differential $\omega$ of the second kind,
\begin{equation}\label{eq:Frobpsi_psiFrob}
\psi^{*}(\phi^*([\omega])) = \phi^*(\psi^{*}([\omega])).
\end{equation}
For each $i\in\{0,1,2,3\}$, the differential $\omega_i = x^{i}\frac{dx}{2y}$ is an eigenvector for $\psi^{*}$ with eigenvalue $\zeta_5^{i+1}$. Therefore, by \eqref{eq:Frobpsi_psiFrob}, the matrix of Frobenius with respect to $\mathcal{B}$ is diagonal. The result then follows by \eqref{eq:holo_pH1}. 
\end{proof}
 The lemmas apply to the prime $p = 10^6 + 81$ and we shall use them to compute the canonical $p$-adic N\'eron function at $p$ for the point
 \begin{equation*}
 P = [(-1-i,-2-i ) - (-1+i,2-i )]\in J_{\Theta}(\Q).
 \end{equation*}
 The first terms of the canonical $p$-adic sigma function can be deduced from the expansion of Appendix \ref{app:sigma_expansion} by specialising to $b_1=b_2 = b_3=b_4 =c_{11} = c_{12} = c_{22}= 0$, $b_5 = -1$. For this reason, we omit here the expansion (although the precision of the computations below requires higher $T$-adic precision than the one given in the appendix). 
 
Let $m$ be a positive integer such that $mP\in J_1(\Q_p)\setminus \Supp(\Theta)$. Since $m$ is large, the computation of $mP$ is expensive. However, if we compute the canonical $p$-adic sigma function up to $O(T_1,T_2)^n$, we can only determine $\sigma_p(T(mP))$ modulo $p^n$, and hence there is no advantage in computing $T_i(mP)$ as a rational number, compared to a $p$-adic number to precision $O(p^n)$. In other words, we may perform multiplication-by-$m$ on the Jacobian over the $p$-adics, to a suitable $p$-adic precision.

The order of $J(\F_p)$ is $m^{\prime} = 1001600512000$ and the computation of $m^{\prime} P$ shows that we may take $m=m^{\prime}$.

Similarly, it suffices to compute the value of the $m$-th division polynomial at $P$ to finite $p$-adic precision; we do this as outlined in \S \ref{subsec:division_polynomials}. Combining the above intermediate computations, we conclude that the $p$-adic N\'eron function for the prime $p$, at the point $P$, with respect to the idele class character $\chi^{\cyc}$ and the unit root eigenspace of Frobenius is:
 \begin{align*}
\lambda_p(P) &=  790065 \cdot p + 875980 \cdot p^{2} + 899921 \cdot p^{3} + 943161 \cdot p^{4} + 701712 \cdot p^{5} \\
&+ 507099 \cdot p^{6} + 399164 \cdot p^{7} + 725683 \cdot p^{8} + 423209 \cdot p^{9} + 174881 \cdot p^{10} \\
&+ 96387 \cdot p^{11} + 973189 \cdot p^{12} + 88349 \cdot p^{13} + 970515 \cdot p^{14} + 117600 \cdot p^{15} \\
&+ 519019 \cdot p^{16} + 639751 \cdot p^{17} + 971144 \cdot p^{18} + 996211 \cdot p^{19} + O(p^{20}). 
\end{align*}
The computation was run on a single core of a $32$-core $2.3$GHz AMD Opteron 6276 processor with 256GB RAM. 
All the computations were run in \texttt{SageMath} \cite{sage} and took approximately 14.3 seconds in total. Of these, 12.6 seconds were needed for the computation of the expansion of the sigma function up to $O(T_1,T_2)^{21}$. We computed this as a power series with coefficients in $\Q$, rather than approximating the coefficients $p$-adically.

In order to compute the global $p$-adic height $h_p(P)$, we also need to compute the values at $P$ of the $p$-adic N\'eron functions at the primes $q\neq p$. First, since
\begin{equation*}
X_{11}(P) =5 ,\qquad X_{12}(P) = -2,\qquad X_{22}(P) =-2 ,
\end{equation*}
and $C$ has good reduction at all the primes different from $2$ and $5$, by Remark \ref{rmk:neron_away_p}\thinspace{}\ref{rmk:neron_away_p_good_reduction} the only non-trivial contributions can occur at $q\in\{2,5\}$.

The point $2P$ reduces to the identity modulo $2$, and does not lie on the theta divisor. 
Therefore, by Definition \ref{def:Neron_fct_away_p},
\begin{equation*}
\lambda_2(P) = -\frac{1}{2}\log_p\left\vert\frac{T_1(2P)}{\phi_2(P)}\right\vert_2 =  -\frac{1}{2}\log_p\left\vert\frac{2 + O(2^7)}{-30} \right\vert_2= 0.
\end{equation*}
As far as $q = 5$ is concerned, by \cite[Remark 5.12, Figure 1]{Bruin-Stoll:MWSieve}, the group $J_0(\Q_5)/J_1(\Q_5)$ has order $25$. An explicit computation shows that $25P$ does not belong to $J_1(\Q_5)$, but $50P$ does. Hence,
\begin{equation*}
\lambda_5(P) = - \frac{2}{50^2}\log_p\left\vert\frac{T_1(50P)}{\phi_{50}(P)}\right\vert_5 = -\frac{1}{2\cdot 625}\log_p\left\vert\frac{4\cdot 5^2 + O(5^3)}{4\cdot 5^{627} + O(5^{628})}\right\vert_5 = -\frac{1}{2}\log_p(5).
\end{equation*}

In conclusion, the global canonical $p$-adic height of $P$ is
\begin{align*}
h_p(P) &= \lambda_p(P) + \lambda_5(P)\\
&= 227482 \cdot p + 997009 \cdot p^{2} + 96340 \cdot p^{3} + 795588 \cdot p^{4} + 602398 \cdot p^{5} \\
&+ 562446 \cdot p^{6} + 378071 \cdot p^{7} + 977705 \cdot p^{8} + 744905 \cdot p^{9} + 778414 \cdot p^{10} \\
&+ 506461 \cdot p^{11} + 834642 \cdot p^{12} + 129041 \cdot p^{13} + 687989 \cdot p^{14} + 134678 \cdot p^{15} \\
&+ 452034 \cdot p^{16} + 429426 \cdot p^{17} + 552523 \cdot p^{18} + 572577 \cdot p^{19} + O(p^{20}).
\end{align*}

\subsection{Example: Coleman--Gross local heights via N\'eron functions}\label{subsec:example_hts_Neron_fcts}
Let
\begin{equation*}
C\colon y^2 = x^3(x-1)^2 + 1,
\end{equation*}
let $P_1 = (1,-1)$, $P_2 = (0,1)\in X(\Q)$, and let $D_1 = P_1-P_2$, $D_2 = P_2^{-} - P_1^{-} \in \Div^0(C)$. 
The prime $p=11$ is of good reduction for $C$ (and hence for its Jacobian $J$), and of ordinary reduction for $J$. Let $W\colonequals W^{(b)}\subset H^1_{\dR}(C/\Q_p) $ be the unit root eigenspace of Frobenius.

In this subsection, we verify the equalities of Corollaries \ref{cor:lambda_eq_CG} and \ref{cor:global_CG_same_as_this} numerically for the pair of divisors $D_1$, $D_2$, as follows. The Coleman--Gross local $p$-adic height pairings $\langle D_1,D_2 \rangle_{p,p,W}^N$ and $\langle D_1,D_2 \rangle_{p,q}^N$ ($q\neq p$) were computed directly from their definition in \cite[\S 7.2.1]{BBM0} (the notation in \emph{loc.\ cit.}\ differs from ours). We check that we obtain the same answer if we take a linear combination of values of $p$-adic N\'eron functions satisfying the assumptions of  Corollary \ref{cor:lambda_eq_CG}. 

Secondly, the divisors $D_1$ and $D_2$ are linearly equivalent, so, by Corollary \ref{cor:global_CG_same_as_this}, the sum of the Coleman--Gross local $p$-adic height pairings on $D_1,D_2$ is equal to the canonical global $p$-adic height $h_p([D_1])$. We verify this numerically by computing $h_p([D_1])$ (almost) directly from its definition. 

To ease notation, we will from now on omit the unit root subspace $W$ from the subscripts. 
Consider the following points in $J(\Q)$:
\begin{align*}
u_1 = \iota(P_1) = [P_1 - \infty], \qquad u_2 = \iota(P_2) = [P_2-\infty],\\
v_1 = [P_2^{-} - P_1],\qquad v_2 = [P_2^{-} - P_2], \qquad v_3 = [P_1^{-}-P_1].
\end{align*}
By Example \ref{eg:finding_u1_u2}, the points $u_1,u_2$ satisfy the assumptions of Corollary \ref{cor:lambda_eq_CG}, for any rational prime $w=q$. 
With this choice, the equality of the corollary translates to
\begin{equation}\label{eq:eq_CG_Neron_in_practice}
\langle D_1,D_2\rangle_{p,q}^N = -\frac{1}{2}(2\lambda_q(v_1) - \lambda_q(v_2) - \lambda_q(v_3)).
\end{equation}
Consider first the case $q = p$. As a preliminary step in the computation of the values of $\lambda_p$ appearing in \eqref{eq:eq_CG_Neron_in_practice}, we need to determine the symmetric matrix $b$ corresponding to the choice of unit root eigenspace of Frobenius. In this case, we use Kedlaya's algorithm as explained in \S \ref{subsec:implementation_canonical_sigma}. We get
\begin{align*}
b_{11} &= 6 + 6 \cdot 11 + 3 \cdot 11^{2} + 6 \cdot 11^{4} + 2 \cdot 11^{5} + 10 \cdot 11^{6} + 11^{7} + 6 \cdot 11^{8} + 9 \cdot 11^{9} + O(11^{10}),\\
 b_{12}  &= 3 + 10 \cdot 11 + 10 \cdot 11^{2} + 11^{4} + 11^{5} + 5 \cdot 11^{6} + 11^{7} + 3 \cdot 11^{8} + 4 \cdot 11^{9} + O(11^{10}) , \\
  b_{22} &=  4 + 3 \cdot 11 + 6 \cdot 11^{2} + 6 \cdot 11^{3} + 9 \cdot 11^{4} + 10 \cdot 11^{5} + 4 \cdot 11^{6} + 5 \cdot 11^{7} + 2 \cdot 11^{8} + 2 \cdot 11^{9} + O(11^{10}).
\end{align*}
The first terms of the canonical $p$-adic sigma function can be deduced from the expansion of Appendix \ref{app:sigma_expansion} by specialising to $b_1=-2,b_2 =1, b_3=b_4 =0, b_5 = 1, c_{ij} = b_{ij}$ for all $1\leq i\leq j\leq 2$.

Using Definition \ref{def:Neron_fct_above_p} with $m=116$, $29$ and $58$, respectively, and the explicit techniques of \S \ref{subsec:implementation_formal_gp_sigma}--\ref{subsec:implementation_Neron_fcts}, we compute:
\begin{align*}
\lambda_p(v_1) &= 9 \cdot 11 + 8 \cdot 11^{2} + 8 \cdot 11^{3} + 4 \cdot 11^{4} + 5 \cdot 11^{5} + 10 \cdot 11^{6} + 8 \cdot 11^{7} + 9 \cdot 11^{8} + O(11^{9}),\\
\lambda_p(v_2) &= 2 \cdot 11^{3} + 6 \cdot 11^{4} + 10 \cdot 11^{5} + 10 \cdot 11^{6} + 9 \cdot 11^{7} + 10 \cdot 11^{8} + O(11^{9}),\\
\lambda_p(v_3) &=2 \cdot 11 + 4 \cdot 11^{2} + 2 \cdot 11^{3} + 6 \cdot 11^{4} + 2 \cdot 11^{5} + 5 \cdot 11^{6} + 3 \cdot 11^{7} + 11^{8} + O(11^{9}).
\end{align*} 
We deduce that
\begin{equation*}
\langle D_1,D_2\rangle_{p,p}^N = 3 \cdot 11 + 4 \cdot 11^{2} + 4 \cdot 11^{3} + 11^{4} + 11^{5} + 3 \cdot 11^{6} + 3 \cdot 11^{7} + 7 \cdot 11^{8} + O(11^{9}),
\end{equation*}
which agrees with the computation of \cite{BBM0} up to $O(11^7)$, i.e.\ up to the precision of the computation in \emph{loc.\ cit.}

Next, we consider the prime $q = 2$. By Definition \ref{def:Neron_fct_away_p}, we have
\begin{align*}
\lambda_2(v_1) &= -\frac{2}{12^2}\log_p\left\vert\frac{T_1(12v_1)}{\phi_{12}(v_1)} \right\vert_{2} = - \frac{7}{6}\log_p(2)\\
\lambda_2(v_2) &=  -\frac{2}{3^2}\log_p\left\vert\frac{T_1(3v_2)}{\phi_3(v_2)} \right\vert_{2} = - \frac{2}{3}\log_p(2)\\
\lambda_2(v_3) &=  -\frac{2}{2^2}\log_p\left\vert\frac{T_1(2v_3)}{\phi_2(v_3)} \right\vert_{2} = 0,
\end{align*}
from which we conclude that
\begin{equation*}
\langle D_1,D_2\rangle_{p,q}^N = \frac{5}{6}\log_p(2). 
\end{equation*}
Finally, for all $q\not\in\{2,11\}$, we have $\lambda_q(v_1) = \lambda_q(v_2) = \lambda_q(v_3) = 0$, and hence $\langle D_1,D_2\rangle_{p,q}^N = 0$, which completes the numerical verification of Corollary \ref{cor:lambda_eq_CG}.

As far as the verification of Corollary \ref{cor:global_CG_same_as_this} is concerned, since $4[D_1] = -v_2$,
\begin{equation*}
h_p([D_1]) = \frac{h_p(v_2)}{16} = \frac{1}{16}(\lambda_p(v_2) + \lambda_2(v_2)) = 8 \cdot 11 + 2 \cdot 11^{4} + 6 \cdot 11^{5} + 6 \cdot 11^{6} + 5 \cdot 11^{7} + 9 \cdot 11^{8} + O(11^{9}),
\end{equation*}
which equals $\sum_{q} \langle D_1, D_2 \rangle_{p,q}^N$.

\subsection{Integrals of differentials of the first, second and third kind}\label{subsec:diff_first_2nd_third}
Let now $p$ be any rational prime ($p=2$ is allowed) and let $C$ be defined over $\Q_p$ and given, as usual, by a monic degree $5$ model with coefficients in $\Z_p$. We make no assumptions on the reduction. By \S \ref{subsec:implementation_formal_gp_sigma}, we can compute the expansion of the strict formal group logarithm $\mathcal{L} = (\mathcal{L}_1,\mathcal{L}_2)$. By Proposition \ref{prop:exp_exp_log}, this converges on $J_1(\Q_p)$. Let $u\in J(\Q_p)$. It was explained in \S \ref{subsec:implementation_Neron_fcts} how we can determine a positive integer $m$ such that $mu\in J_1(\Q_p)$. Then, for each $i\in \{1,2\}$, the Colmez integral of $\Omega_i$ satisfies
\begin{equation*}
\int_0^{u} \Omega_i =\frac{1}{m}\mathcal{L}_i(mu). 
\end{equation*}
Let $\omega_i = x^{i-1}\frac{dx}{2y}$ and $P_1,P_2\in C(\Q_p)$. By Definition \ref{def:basis_inv_diff_inv_der} and Theorem \ref{thm:uniquetheory}\thinspace{}\ref{change_of_vars}, we have
\begin{equation*}
\int_{P_2}^{P_1} \omega_i = \int_{0}^{[P_1 - P_2]} \Omega_i.
\end{equation*}
In this way, we can compute integrals of holomorphic differentials on $C$, without any assumption on the reduction. 
\begin{rmk}
For genus $2$ curves described by $y^2 = f(x)$ where $f(x)$ is a sextic (or quintic) polynomial, the more general formal group description of Flynn \cite{Flynn2, Flynn5} can be used to compute integrals of holomorphic differentials. See for example \cite{flynn:flexible} for an application to Chabauty's method.
\end{rmk}

A non-holomorphic closed differential of the second kind on $J$ is not translation-invariant, but it is semi-invariant \cite[Theorem 2.8]{barsotti}. This means that it is linearly equivalent to its pullback under translation by any $u\in J$, and hence, in principle, we can use the formal group machinery to compute the Colmez integral of such a differential. For example, for $i\in\{1,2\}$, consider the differential (Lemma \ref{lemma:explicit_subspace_H1dRJ})
\begin{equation*}
\eta_{i,J}^{(0)} = -X_{1i} \Omega_1 - X_{i2}\Omega_2.
\end{equation*}
Using Theorem \ref{thm:uniquetheory}\thinspace{}\ref{change_of_vars} and \cite[Proposition 4.10]{Uchida}, we see that the integral of $\eta_{i,J}^{(0)}$ must satisfy
\begin{equation*}
\int_{mv}^{mu}\eta_{i,J}^{(0)} = m\int_{v}^{u}\eta_{i,J}^{(0)} + \frac{1}{m}\frac{D_i(\phi_m)}{\phi_m}(u) - \frac{1}{m}\frac{D_i(\phi_m)}{\phi_m}(v) \qquad (u,v\in J(\Q_p)).
\end{equation*}
By choosing $m$ appropriately, the computation of the integral on the left hand side can be reduced to solving a system of differential equations in the formal group (i.e.\ over a power series ring); note that this is in fact an intermediate step in the computation of the $p$-adic sigma function. Pulling back to $C$ using \eqref{eq:eta_iJ_eta_i} yields a formula for the integral of $\eta_i^{(0)}$. The implementation of \S \ref{subsec:division_polynomials} does not immediately equip us with a way of computing $D_i(\phi_m)$ (since we do not compute the division polynomials, but only their values); however, one could derive from the recurrence relations of Kanayama \cite[Theorem 9 (corrected)]{Kanayama_corrections} and Uchida \cite[Example 6.6]{Uchida} (cf.\ \S \ref{subsec:division_polynomials}) an expression for $D_i(\phi_m)(u)$ in terms of $\phi_k(u)$ ($k\leq m$), $D_i(\phi_k)(u)$ ($k < m$) and higher order derivatives of $\phi_k$ for $k\leq 5$. 

Finally, Corollary \ref{cor:lambda_eq_CG} provides a formula for the computation of differentials of the third kind on $C$ by means of $p$-adic N\'eron functions, formula which we tested numerically in \S \ref{subsec:example_hts_Neron_fcts}.

In summary, our implementation, although primarily aimed at the computation of $p$-adic N\'eron functions and heights, can be adapted to compute integrals of differentials of the first, second and third kind on $C$, and thereby offers an alternative (for the specific setting of a genus $2$ odd degree hyperelliptic curve) to Coleman integration algorithms in good reduction due to Balakrishnan--Bradshaw--Kedlaya and Balakrishnan--Besser \cite{BBK09, Bes-Bal10} and Colmez--Vologodsky integration algorithms in  bad reduction due to Katz--Kaya and Kaya \cite{Katz_Kaya, kaya_vologodskyII}.

\section{An application: quadratic Chabauty for bihyperelliptic curves} \label{sec:application_bihyper}
The goal of this section is twofold. Balakrishnan--Besser--M\"uller \cite{BBM0, BalakrishnanBesserMullerIntegralPoints} described a Chabauty-like method for computing integral points on odd degree hyperelliptic curves over $\Q$ with genus equal to the Mordell--Weil rank of the Jacobian. This uses an extension of Coleman--Gross local height pairings to divisors with non-disjoint support, and the quadraticity of the global height. The technique is known as \emph{quadratic Chabauty} (for integral points). The intermediate goal of the section is to rephrase this method in the genus 2 case in terms of our setup. In fact, we do not just translate, but rather re-prove, the results of \cite{BBM0} in terms of N\'eron functions (see Remark \ref{rmk:comparison_not_applicable} below). For this we crucially use results of Stoll \cite{Stoll} and M\"uller--Stoll \cite{mueller-stoll}.

Our main goal is to describe a quadratic Chabauty-like method for determining the rational points on certain genus 4 bihyperelliptic curves $X$. In particular, we assume that $X$ has two genus 2 quotients, each satisfying the assumptions of quadratic Chabauty for integral points. We then describe a $p$-adic locally analytic function on $X(\Q_p)$, which,  when restricted to $X(\Q)$, takes values in a finite explicit subset of $\Q_p$. This is a genus $4$ analogue of the quadratic Chabauty method for genus $2$ bielliptic curves of Balakrishnan--Dogra \cite{BDQCI}; see also \cite{Bianchi20, BP22}.

Let $C$ and $K$ be as in Sections \ref{sec:sigma_functions} and \ref{sec:padic_hts}: that is, $K$ is a number field and $C$ is the curve over $K$ defined by an equation of the form $y^2 = f(x)$, where $f(x)$ is a monic polynomial of degree $5$ with coefficients in the ring of integers of $K$ and no repeated roots. Fix a prime $p$ and a non-trivial continuous $\Q_p$-valued idele class character $\chi$ for $K$. Given a non-archimedean place $v$ of $K$, let $\lambda_v$ be the local $p$-adic N\'eron function at $v$ (with respect to a choice of subspace of $H^1_{\dR}(C/K_v)$ if $\chi$ is ramified at $v$).

If $\chi$ is unramified at $v$, we saw (Lemma \ref{lemma:equals_naive}) that there exists a finite index subgroup $H_v$ of $J(K_v)$ such that, for all $u\in H_v\setminus \Supp(\Theta)$, the local N\'eron function $\lambda_v$ only depends on the $v$-adic valuation of $X_{ij}(u)$, for $i,j\in\{1,2\}$. More generally, we have the following. For $u\in J_{\Theta}(K_v)$, let
\begin{equation*}
\lambda_v^{\naive}(u)= -\frac{1}{n_v}\chi_v(\max_{i,j}\{|X_{ij}(u)|_v,1\})
\end{equation*}
and let $\mu_v\colon J(K_v)\to \Q$ be the function of \cite[Definition 3.1]{mueller-stoll}. 
Then the $p$-adic N\'eron function at $v$ is obtained from $\lambda_v^{\naive}$ by adding a correction term:
\begin{equation*}
\lambda_v(u) = \lambda_v^{\naive}(u) +\frac{1}{n_v}\mu_v(u)\chi_v(\pi_v),
\end{equation*}
where $\pi_v$ is a uniformiser in $K_v$.
Denote by $k_v$ the residue field of $K_v$, and let $\Phi$ be the component group of the N\'eron model of $J$ over $\Spec(\OO_v)$, where $\OO_v$ is the ring of integers of $K_v$.
\begin{thm}[{\hspace{1sp}\cite[Theorems 3.10, 11.3, 7.4, Proposition 12.3, Lemma 12.5]{mueller-stoll}}] \label{thm:properties_mu}\leavevmode
\begin{enumerate}[label=(\roman*)]
\item The set $H_v = \{u\in J(K_v):\mu_v(u) = 0\}$ is a subgroup of finite index of $J(K_v)$ containing $J_1(K_v)$.
\item The function $\mu_v$ factors through $J(K_v)/H_v$.
\item For every $u\in J(K_v)$, we have $0\leq \mu_v(u) \leq \frac{\ordnop_v(\Delta)}{4}$, where $\Delta$ is the discriminant of $C$. 
\item Let $N$ be the exponent of $\Phi(\overline{k_v})$. Then, for every $u\in J(K_v)$, we have $\mu_v(u)\in \frac{1}{N}\Z$. Moreover, $N\leq \max\{2,\lfloor{\frac{\ordnop_v(\Delta)^2}{3}\rfloor}\}$.
\item If our fixed equation for $C$ determines a model over $\OO_v$ with rational singularities, then $\mu_v$ factors through $\Phi(k_v)$.
\end{enumerate}
\end{thm}

Consider now the function
\begin{equation*}
\nu_v\colon \{P\in C(K_v): y(P)\neq 0,\infty \}\to \Q_p, \qquad \nu_v(P) = \lambda_v(2\iota(P)) + \frac{2}{n_v}\chi_v(2y(P)).
\end{equation*}

Denote by $C(\OO_v)$ the set of affine points with coordinates in $\OO_v$, with respect to our fixed equation for $C$.

\begin{lemma}\label{lemma:nu_v}If $\chi$ is unramified at $v$, the function $\nu_v$ satisfies the following properties:
\begin{enumerate}[label=(\roman*)]
\item\label{lemma:nu_v_integral}  If $P\in C(\OO_v)\setminus \{P:y(P) = 0\}$, then
\begin{equation*}
\nu_v(P) = -\frac{2}{n_v}\chi_v(\max\{|f^{\prime}(x(P))|_v, |2y(P)|_v\}) +\frac{1}{n_v}\mu_v(2\iota(P))\chi_v(\pi_v).
\end{equation*}
In particular, the set $\Gamma_v\colonequals\nu_v(C(\OO_v)\setminus\{P:y(P) = 0\})$ is finite, and there exists a computable finite set $\Gamma_v^{\prime}\subset\Q_p$ such that $\Gamma_v\subset\Gamma_v^{\prime}$. If $v$ is a prime of good reduction for our equation for $C$, we may take $\Gamma_v^{\prime} = \{0\}$.
\item \label{lemma:nu_v_nonintegral} If $P\not\in C(\OO_v)$, then $\nu_v(P) = \frac{8}{n_v}\chi_v(x(P))$. 
\end{enumerate}
\end{lemma}

\begin{proof}
For $P$ in the domain of $\nu_v$, let 
\begin{equation*}
 \nu_v^{\naive}(P)= \lambda_v^{\naive}(2\iota(P)) + \frac{2}{n_v}\chi_v(2y(P)).
\end{equation*}
It follows from \cite[\S 2]{FlynnSmart} that 
\begin{align*}
X_{22}(2\iota(P)) = 2x(P), \quad X_{12}(2\iota(P)) = -x(P)^2, \\
X_{11}(2\iota(P)) = \frac{f^{\prime}(x(P))^2 -(2y(P))^2(6x(P)^3+4b_1x(P)^2 + 2b_2x(P) + b_3)}{(2y(P))^2}.
\end{align*}
Therefore, if $P\in C(\OO_v)$, we have 
\begin{equation*}
\nu_v^{\naive}(P) = -\frac{2}{n_v}\chi_v(\max\{|f^{\prime}(x(P))|_v, |2y(P)|_v\}).
\end{equation*}
This can take finitely many values since $\max\{|f^{\prime}(x(P))|_v, |2y(P)|_v\}\geq |2\Disc(f)|_v$, where $\Disc(f)$ is the discriminant of $f$. 
 In particular, if $2\Disc(f)$ is a $v$-adic unit, then $\nu_v^{\naive}(P)$ vanishes. 

If $P\not\in C(\OO_v)$, the coordinates of $P$ satisfy $2\ordnop_v(y(P)) = 5\ordnop_v(x(P))$. Therefore, since the numerator of $X_{11}(2\iota(P))$ is monic of degree $8$ as a polynomial in $x(P)$, we have
\begin{equation*}
\nu_v^{\naive}(P) = \frac{8}{n_v}\chi_v(x(P)).
\end{equation*}
It remains to understand how $\nu_v$ differs from $\nu_v^{\naive}$. For this, we apply Theorem \ref{thm:properties_mu}.
\end{proof}

Assume now for simplicity that $K=\Q$. Recall from \S \ref{subsec:Colm_int} that every locally analytic group homomorphism $J(\Q_p)\to\Q_p$ arises as an extension to $J(\Q_p)$ of a formal group homomorphism to $\G_a$. By abuse of notation, denote by
\begin{equation*}
\mathcal{L}=(\mathcal{L}_1,\mathcal{L}_2)\colon J(\Q_p)\otimes \Q_p\to\Q_p^{2}
\end{equation*}
the $\Q_p$-linear map induced by the strict formal group logarithm.
\begin{ass}\label{ass:QC_int}
The restriction of $\mathcal{L}$ to $J(\Q)\otimes \Q_p$ is injective. 
\end{ass}

\begin{rmk}
If the rank of $J(\Q)$ is at most $2$, Assumption \ref{ass:QC_int} will often be satisfied: if $J$ is simple, see \cite[Conjecture 1]{waldschmidt}.
\end{rmk}
Under Assumption \ref{ass:QC_int}, any $\Q_p$-linear map $J(\Q)\otimes \Q_p\to \Q_p$ is a linear combination of $\mathcal{L}_1$ and $\mathcal{L}_2$, and any quadratic form $J(\Q)\otimes \Q_p\to\Q_p$ is a linear combination of $\mathcal{Q} = \{\mathcal{L}_1^2, \mathcal{L}_1\mathcal{L}_2,\mathcal{L}_2^2\}$. The latter applies in particular to the $p$-adic height $h_p$: there exist $\alpha_1,\alpha_2,\alpha_3\in \Q_p$ such that
\begin{equation}\label{eq:h_p_basis_qf}
 h_p(u) = \alpha_1\mathcal{L}_1^2 (u)+ \alpha_2\mathcal{L}_1(u)\mathcal{L}_2(u) + \alpha_3\mathcal{L}_2^2(u) \qquad \text{for all }u\in J(\Q). 
 \end{equation} 
 Restricting this equality to points of the form $2\iota(P)$, where $P$ is a point on $C$, and writing $h_p$ as a sum of local N\'eron functions, one can use Equation \eqref{eq:h_p_basis_qf} and Lemma \ref{lemma:nu_v}\thinspace{}\ref{lemma:nu_v_integral} to write down a locally analytic function $\rho\colon C(\Z_p)\to \Q_p$ and a finite set $\Gamma\subset\Q_p$ such that $\rho(C(\Z))\subset \Gamma$.  In other words, Lemma \ref{lemma:nu_v} allows us to phrase the quadratic Chabauty method of \cite{BBM0} in terms of $p$-adic N\'eron functions, in place of local Coleman--Gross height pairings.

\begin{rmk}\label{rmk:comparison_not_applicable}
Corollary \ref{cor:lambda_eq_CG} is a comparison result between the Coleman--Gross local height pairings and the $p$-adic N\'eron functions. However, it only applies to divisors with disjoint support, while the method of \cite{BBM0} crucially requires an extension of the Coleman--Gross local heights to arbitrary divisors. Lemma \ref{lemma:nu_v} allows us to give a direct proof that suitable local N\'eron functions can alternatively be used. 
On the other hand, in \cite{BKM22}, we prove a comparison result between these extended Coleman--Gross local height pairings and the N\'eron functions, from which the rephrasing of quadratic Chabauty for $C(\Z)$ in terms of the $\lambda_q$ is straightforward.

We also observe that the explicit examples of \cite{BBM0, BalakrishnanBesserMullerIntegralPoints} rely, for the computation of $\mathcal{L}(\iota(P))$,  on algorithms for Coleman integration on $C$ based on Kedlaya's algorithm \cite{kedlaya}. In our genus $2$ setting, we may replace this step with \S \ref{subsec:implementation_formal_gp_sigma} (see also \S \ref{subsec:diff_first_2nd_third}).
\end{rmk}

We now apply similar ideas to give an explicit quadratic-Chabauty-type method for determining the \emph{rational} points on certain genus $4$ curves that admit degree $2$ maps to curves in the same form as $C$. This is an analogue of the explicit quadratic Chabauty method for the rational points on genus $2$ bielliptic curves due to Balakrishnan--Dogra \cite{BDQCI}; our proof is based on the proofs thereof given in \cite[Proposition 6.5]{Bianchi20} and \cite[Theorem 2.3]{BP22}. 

In particular, we give a proof that does not use Kim's theory \cite{KimP1, Kimunipotent}, but only the properties of the local N\'eron functions that we have described so far. As such, it is much more elementary in nature to the general framework of quadratic Chabauty for rational points of \cite{BDQCI, SplitCartan}. 
 
Unlike the above discussion on integral points on $C$, to the author's knowledge, Proposition \ref{prop:QC_bihyper} below is not a rephrasing of a  result involving Coleman--Gross local heights already appearing in the literature. On the other hand, using the comparison result of \cite{BKM22} mentioned in  Remark \ref{rmk:comparison_not_applicable}, it could be rephrased in terms of Coleman--Gross heights.

Suppose that $X$ is a genus $4$ hyperelliptic curve over $\Q$ given by an equation of the form
\begin{equation}\label{eq:X}
X\colon y^2 = f_5x^{10} + f_4 x^8 + f_3 x^6 + f_2x^4 + f_1x^2 + f_0 , \qquad f_i\in \Z.
\end{equation}
There are maps $\varphi_1\colon X\to C_1$, $\varphi_2\colon X\to C_2$, to the following curves $C_1$ and $C_2$:
\begin{align*}
C_1\colon y^2 = x^5 + f_4x^4 + f_3f_5x^3 + f_2f_5^2x^2 + f_1f_5^3x + f_0 f_5^4,\qquad &\varphi_1(x,y) = (f_5 x^2,f_5^2 y),\\
C_2 \colon y^2 = x^5 + f_1 x^4 + f_0f_2x^3 + f_0^2f_3x^2 + f_0^3f_4x + f_0^4 f_5,\qquad &\varphi_2(x,y) = (f_0x^{-2},f_0^2yx^{-5}).
\end{align*}
Denote by $J_1$ and $J_2$ the Jacobian of $C_1$ and $C_2$, respectively, and by $\mathcal{Q}_i$ the set of quadratic forms $\{\mathcal{L}_1^{2},\mathcal{L}_1\mathcal{L}_2,\mathcal{L}_2^2\}$ for $J_i$. Let $Z =  X(\Q_p)\setminus \{z: x(z)\in \{0,\infty\}\text{ or } y(z) = 0\}$.

\begin{prop}\label{prop:QC_bihyper}
Suppose that $p$ is a prime of good reduction for the Equation \eqref{eq:X}, and that $J_i$ satisfies Assumption \ref{ass:QC_int} for every $i\in\{1,2\}$,  so the global $p$-adic height on $J_i(\Q)$ is equal to the restriction to $J_i(\Q)$ of $\sum_{Q\in \mathcal{Q}_i} \alpha_{Q}^{(i)}Q$, for some $\alpha_Q^{(i)}\in\Q_p$. The function $\rho\colon Z\to \Q_p$, defined by
\begin{align*}
\rho(z) &=\lambda_p(2\iota(\varphi_1(z))) - \lambda_p(2\iota(\varphi_2(z))) - 6\chi_p(x(z))\\ 
&- \sum_{Q\in \mathcal{Q}_1} \alpha_{Q}^{(1)}Q(2\iota(\varphi_1(z))) + \sum_{Q\in \mathcal{Q}_2} \alpha_{Q}^{(2)}Q(2\iota(\varphi_2(z))),
\end{align*}
can be continued to a locally analytic function $\tilde{\rho}\colon X(\Q_p)\to \Q_p$. Moreover, there exists a finite and computable set $\Upsilon\subset\Q_p$ such that
\begin{equation*}
\tilde{\rho}(X(\Q))\subset \Upsilon.
\end{equation*}
\end{prop}
\begin{proof}
At each point in $X(\Q_p)\setminus Z$, there are exactly two terms of $\rho(z)$ that have a logarithmic singularity:
\begin{itemize}
\item if $x(z) = \infty$, the terms are: $\lambda_p(2\iota(\varphi_1(z)))$ and $-6\chi_p(x(z))$;
\item if $x(z) = 0$, the terms are: $-\lambda_p(2\iota(\varphi_2(z)))$ and $-6\chi_p(x(z))$;
\item if $y(z) = 0$, the terms are: $\lambda_p(2\iota(\varphi_1(z)))$ and $-\lambda_p(2\iota(\varphi_2(z)))$.
\end{itemize}
In each of the three cases, we claim that the logarithmic singularities coming from the two terms cancel out. We prove this for $x(z) = \infty$ and leave the other cases as an exercise to the reader. So assume $x(z) = \infty$ and let $t$ be a uniformiser at $z$, reducing to a uniformiser modulo $p$: without loss of generality, we choose $t = x^{-1}$. Since $\varphi_1$ is unramified at $z$, for any local coordinate $t_1$ at $\varphi_1(z)$, we have $t_1(\varphi_1(z(t))) = tu(t)$ for some unit power series $u(t)\in \Z_p[[t]]$. By definition of $\lambda_p$, the logarithmic term of $\lambda_p(2\iota(\varphi_1(z(t))))$ is the same as that of $-2\chi_p(T_1(2\iota(\varphi_1(z(t)))))$. By Lemma \ref{lemma:Ti_terms_t}, this is $-6\chi_p(t) = 6\chi_p(x(t))$. 

We prove the existence of a finite computable set $\Upsilon$ for $X(\Q)\setminus Z$ (the limiting values will follow by continuity).
If $z\in X(\Q)\setminus Z$, we have
\begin{align*}
\rho(z)& =\sum_{q\neq p}(- \lambda_q(2\iota(\varphi_1(z))) + \lambda_q(2\iota(\varphi_2(z))) +6\chi_q(x(z)))\\
&= \sum_{q\neq p} (- \nu_q(\varphi_1(z)) + \nu_q(\varphi_2(z)) +4\chi_q(f_5f_0^{-1}x(z)^4))=: \sum_{q\neq p} w_q(z).
\end{align*}
Let $Z_q$ be defined analogously to $Z$, but for the prime $q$. We claim that there exists a finite computable set $\Upsilon_q$ such that $w_q(Z_q)\subset \Upsilon_q$, and, moreover, that we can take $\Upsilon_q = \{0\}$ at every prime $q$ of good reduction; the statement of the proposition then follows from this.  In order to prove the existence and properties of $\Upsilon_q$, we analyse the $q$-adic valuation of $x(z)$ and, correspondingly, that of $x(\varphi_1(z))$ and $x(\varphi_2(z))$ and apply Lemma \ref{lemma:nu_v}. We distinguish between two cases: 
\begin{itemize}
\item $\varphi_1(z)$ and $\varphi_2(z)$ are both integral if and only if $-\ordnop_q(f_5)\leq 2\ordnop_q(x(z))\leq \ordnop_q(f_0)$.  In this case, each of the three terms in $w_q(z)$ takes values in a finite computable set (by Lemma \ref{lemma:nu_v}\thinspace \ref{lemma:nu_v_integral} and the fact that $\chi_q(x(z))$ only depends on the valuation of $x(z)$).
\item In the remaining cases, exactly one of $\varphi_1(z)$ and $\varphi_2(z)$ is integral. Say $\varphi_i(z)$ is integral and $\varphi_j(z)$ is not. Then we apply Lemma \ref{lemma:nu_v}\thinspace{}\ref{lemma:nu_v_integral} to $\nu_q(\varphi_i(z))$; moreover, by Lemma \ref{lemma:nu_v}\thinspace{}\ref{lemma:nu_v_nonintegral}, 
\begin{equation*}
(-1)^j \nu_q(\varphi_j(z)) + 4\chi_q(f_5f_0^{-1}x(z)^4) = (-1)^{j}4\chi_q(f_5f_0). \qedhere
\end{equation*} 
\end{itemize}
\end{proof}

\begin{rmk}
With Lemma \ref{lemma:Ti_terms_t} and Lemma \ref{lemma:nu_v} on hand, the proof of Proposition \ref{prop:QC_bihyper} is very similar to the proof of the genus $2$ case in \cite[Theorem 2.3]{BP22}, to which we refer the reader for more details.
\end{rmk}

\begin{example}
Consider
\begin{equation*}
X\colon y^2 = x^{10} - x^6 +1
\end{equation*}
with corresponding genus $2$ curves
\begin{equation*}
C_1\colon y^2 = x^5 - x^3 + 1, \qquad C_2\colon y^2 = x^5 - x^2 + 1.
\end{equation*}
Let $p=5$, let $\chi = \chi^{\cyc}$ and consider the $p$-adic local N\'eron functions on $C_1$ and $C_2$ with respect to their naive sigma functions. The rank of $J_i(\Q)$ is equal to $2$ for each $i$, and the assumptions of Proposition \ref{prop:QC_bihyper} apply. 

The discriminants of $C_1$ and $C_2$ are both equal to $2^8\cdot 7\cdot 431$. Therefore, with reference to the proof of the proposition, in order to determine a suitable set $\Upsilon$ we have to compute $\Upsilon_q$, for $q\in\{2,7,431\}$, which reduces to computing $\Gamma_q$ or a finite superset $\Gamma_q^{\prime}$ (as in Lemma \ref{lemma:nu_v}\thinspace{}\ref{lemma:nu_v_integral}) for each of $C_1$ and $C_2$.  
For $q\in\{7,431\}$, by \cite[Proposition 5.2 and Remark below]{Stoll}, $\mu_q$ is identically zero on $J_i(\Q_q)$, for each $i$. Moreover, there is no point in $C_i(\Z_q)$ for which the function $\nu_q^{\naive}$ defined in the proof of Lemma \ref{lemma:nu_v} is non-trivial.
 Therefore, the set $\Gamma_q$ is equal to $0$ for each of $C_1$ and $C_2$, and $\Upsilon = \Upsilon_2$. 

We omit a proof of the following fact: we may take\footnote{Rather than using Lemma \ref{lemma:nu_v}\thinspace{}\ref{lemma:nu_v_integral}, this set was computed using the comparison results of \cite{BKM22} mentioned in Remark \ref{rmk:comparison_not_applicable} and the results in \cite{BBM0, BalakrishnanBesserMullerIntegralPoints}. We thank Steffen M\"uller for computing that the set $T$ of \cite[Theorem 3.1]{BBM0} is equal to $\{0,-\frac{2}{3}\log_p(2)\}$ for $C_1$, and to $\{0,-\frac{1}{2}\log_p(2)\}$ for $C_2$. The set $\Upsilon$ was deduced from this.}
\begin{equation}\label{eq:Gamma}
\Upsilon_2 = \left\{0, \frac{8}{3}\log_p(2), -2\log_p(2)\right\}. 
\end{equation}

We can compute the expansion of $\tilde{\rho}(z)$ in every residue disc of the reduction map $X(\Q_p)\to \tilde{X}(\F_p)$ and study the set
\begin{equation*}
 \{z\in X(\Q_p):\tilde{\rho}(z) - \upsilon,\text{ for some } \upsilon\in \Upsilon\}\supseteq X(\Q).
\end{equation*}

For example, the points in $X(\Q_p)$ reducing to $\overline{(0,4)}\in X(\F_p)$ are parametrised by $t = \frac{x}{p}\in \Z_p$ and 
\begin{equation*}
\tilde{\rho}(z(t)) = 2\cdot 5 + 2\cdot 5^2  +(2\cdot 5^2 + 2\cdot 5^3 )t^2 + (3\cdot 5^3 )t^4  + O(5^{4}, t^6).
\end{equation*}

\begin{itemize}
\item If $\upsilon =0$ or $\upsilon = -2\log_p(2) = 5 +O(5^2)$, the power series $\rho(z(t)) - \upsilon$ has no zeros in $\Z_p$;
\item If $\upsilon =\frac{8}{3}\log_p(2) = 2\cdot 5 + 2\cdot 5^2 + O(5^4)$, the power series $\tilde{\rho}(z(t)) - \upsilon$  has a double root at $t=0$: this corresponds to the rational point $(0,-1)\in X(\Q)$. The root has multiplicity two because $(0,-1)$ is fixed by the automorphism of $X$ mapping $(x,y)$ to $(-x,y)$, under which $\tilde{\rho}$ is invariant.
\end{itemize}
This example is intended as an informal illustration of Proposition \ref{prop:QC_bihyper}, but upon writing down an argument for \eqref{eq:Gamma} and more details on how we computed the expansion of $\tilde{\rho}$, it could be turned into a proof that $(0,-1)$ is the unique rational point reducing to $\overline{(0,4)}$ modulo $5$. By repeating this procedure in every residue disc of $X(\Q_p)$ (possibly in combination with the Mordell--Weil sieve), it should be possible to determine the full set of rational points $X(\Q)$.
\end{example}

\begin{rmk}\label{rmk:QC_bihyper_gen}
Proposition \ref{prop:QC_bihyper} is phrased in terms of $p$-adic local N\'eron functions, which we defined in this article for genus $2$ curves and, more generally, we can compute on genus $\leq 2$ curves. As previously mentioned, the proposition admits a version in terms of Coleman--Gross local height pairings of a divisor with itself. More generally, in terms of such Coleman--Gross local heights, the result can be extended to rational points on even genus $d-1\geq 2$ hyperelliptic curves of the form:
\begin{equation*}
X\colon y^2 = \sum_{n=0}^{d} f_n x^{2n}, \qquad f_n\in \Z.
\end{equation*}
\end{rmk}

\appendix
\section*{Appendix}
\renewcommand{\thesubsection}{\Alph{subsection}}
\counterwithin{lemma}{subsection}
\counterwithin{rmk}{subsection}
\subsection{Multiples of points away from the theta divisor}\label{app:stevan}
Let $K$ be a field of characteristic $0$ and let $C\colon y^2 = f(x)$ with $f(x)\in K[x]$ monic, of degree $5$, with no repeated roots. Let $J$ be the Jacobian of $C$. Recall that:
\begin{enumerate}[label=(\Roman*)]
\item \label{pr:I} If $u\in J$ is a non-zero point, there exists a unique unordered pair $(P_1,P_2)\in C^{(2)}$ such that $u = [P_1+P_2-2\infty]$. Moreover, $u\in J[2]\setminus\{0\}$ if and only if $P_1$ and $P_2$ are both Weierstrass points.
\item \label{pr:II} $[P_1+P_2 - 2\infty] = 0$ if and only if $P_1 = P_2^{-}$.
 \end{enumerate}
 From these the following lemma easily follows.
\begin{lemma}\label{lemma:aux_app_stevan}
Let $u,v,w\in J$. 
If $u,v,w\in J$ lie on the theta divisor, then $u+v+w\not\in\Supp(\Theta)$, unless $u=-v$ or $u=-w$ or $v=-w$.  
\end{lemma}

\begin{proof}
Let $u = [P-\infty], v=[Q-\infty], w = [R-\infty]$. Then $u+v+w = [S-\infty]$ if and only if
\begin{equation*}
[P+Q-2\infty] = [R^{-}+S-2\infty],
\end{equation*}
which gives the claim by \ref{pr:I} and \ref{pr:II}.
\end{proof}

More involved is the proof of following result.
\begin{lemma}[{\hspace{1sp}\cite[Lemma 8.1\thinspace{}(iii)]{muller_de_jong}}]\label{lemma:aux_app_stevan_gaps}
Let $u\not\in J[2]$ and let $k$ be an integer. Then at least one of $ku, (k+1)u, (k+2)u, (k+3)u$ is not in the support of the theta divisor. 
\end{lemma}

The following lemma and its proof are due to S.\ Gajovi\'c.
\begin{lemma}[Gajovi\'c]\label{lemma:stevan}
Let $u_1,u_2,u_3,u_4\in J$ such that none of $u_1,u_2,u_3,u_4$ is a torsion point. Then there exists an integer $1\leq m\leq 41$ such that none of $mu_1$, $mu_2$, $mu_3$ and $mu_4$ belongs to the support of the theta divisor.
\end{lemma}
\begin{proof} 
Assume by contradiction that such an integer does not exist. Then we can partition $\{1,\dots, 41\}$ into four subsets $S_1,S_2,S_3,S_4$ such that if $i\in \{1,\dots, 4\}$ and $k\in S_{i}$ then $ku_i\in \Supp(\Theta)$.  By Lemmas \ref{lemma:aux_app_stevan} and \ref{lemma:aux_app_stevan_gaps}, the set $S_i$ satisfies:
\begin{enumerate}[label=(\roman*)]
\item\label{it:pr_1_partition} if $k_1,k_2,k_3\in S_i$ (not necessarily distinct), then $k_1+k_2\not\in S_i$ and $k_1+k_2+k_3\not\in S_i$;
\item\label{it:pr_2_partition} if $k, k+1, k+2\in S_i$ then $k+3\not\in S_i$.
\end{enumerate}
The repository \cite{github_padic_g2} contains \texttt{Magma} \cite{magma} code to check that there exists no partition of $\{1,\dots, 41\}$ into four subsets $S_{1},S_2, S_3, S_4$, each satisfying properties \ref{it:pr_1_partition} and \ref{it:pr_2_partition}. 
\end{proof}

\begin{rmk}
In the proof of Lemma \ref{lemma:stevan}, considering both constraints \ref{it:pr_1_partition} and \ref{it:pr_2_partition}, rather than just \ref{it:pr_1_partition}, does not lead to any improvement in the upper bound for $m$. However, it speeds up the \texttt{Magma} computation.
\end{rmk}

\subsection{Formal group parameters expansions}
Let $K_v$ be the completion of a number field $K$ at a non-archimedean place $v$, let $\OO_v$ be its ring of integers, and let $C\colon y^2 = f(x)$, where $f(x) = x^5+b_1x^4+b_2x^3+b_3x^2+b_4x+b_5\in \OO_v[x]$ has no repeated roots.
\begin{lemma}\label{lemma:Ti_terms_t}
Suppose that the equation for $C$ has good reduction and let $P = (x,y)\in C(K_v)$ be a point reducing to a Weierstrass point modulo $v$. Then $2\iota(P)\in J_1(K_v)$ and 
\begin{equation*}
T_1(2\iota(P)) =  -2y\frac{g_1(x)}{g_2(x)},\qquad T_2(2\iota(P))= -2y\frac{g_3(x)}{g_2(x)},
\end{equation*}
where 
\begin{align*}
g_1(x) &=  f^{\prime}(x)^2 -4f(x)(6x^3 + 4b_1x^2 +2b_2 x + b_3)\\
g_2(x) &= -g_1(x)f^{\prime}(x) + f(x)\tilde{g}_2(x)\\
\tilde{g}_2(x) &= 4 x^{7} -8 b_{2} x^{5} -32 b_{3} x^{4} + 4\left(b_{2}^{2} - 4 b_{1} b_{3} - 11 b_{4}\right) x^{3} \\
&\ \ -8\left(3 b_{1} b_{4} + 7 b_{5}\right) x^{2}  -4\left(b_{2} b_{4} + 8 b_{1} b_{5}\right) x - 8 b_{2} b_{5}\\
g_3(x) &= xg_1(x) - 2f(x)(x^4 + b_2x^2+b_4).
\end{align*}
In particular, if $t$ is a uniformiser at a Weierstrass point $P_0\in C(K_v)$ which reduces to a uniformiser modulo $v$, there exist unit power series $u_1(t), u_2(t)\in \OO_v[[t]]$ and a power series $w_2(t)\in t^2\OO_v[[t]]$ such that
\begin{align*}
T_1(2\iota(x(t),y(t))) &= \begin{cases}
2t^3u_1(t) & \text{if } P_0 =\infty;\\
2tu_1(t) & \text{otherwise.}
\end{cases}\\
 T_2(2\iota(x(t),y(t))) &= \begin{cases}
2tu_2(t) & \text{if } P_0 =\infty;\\
2x(0)t u_2(t) + w_2(t)& \text{otherwise.}
\end{cases}
\end{align*}
\end{lemma}

\begin{proof}
Recalling the formulae for $X_{ij}(2\iota(P))$ used in the proof of Lemma \ref{lemma:nu_v}, we see that 
\begin{align*}
X_{11}(2\iota(P)) = \frac{g_1(x)}{4y^2}, \qquad X(2\iota(P)) = \frac{g_3(x)}{4y^2}.
\end{align*}
Using the formulae for $X_{ijk}([P_1+P_2-2\infty])$ of \cite[(1.4)]{Grant1990}, we find, taking limits as $(x_1,y_1)$ and $(x_2,y_2)$ both go to $(x,y)$, that
\begin{equation*}
X_{111}(2\iota(P)) = \frac{g_2(x)}{8y^3}.
\end{equation*}
As far as the expansions in $t$ are concerned, it suffices to show the claim for one such local coordinate $t$, since any other can be obtained from $t$ by multiplication by a unit power series in $\OO_v[[t]]$. 

If $P_0 = \infty$, let $t= -\frac{x^2}{y}$. Then 
\begin{align*}
x(t) = t^{-2}  + O(1)\in \OO_v((t)),\qquad y(t) = -t^{-5} + O(t^{-3})\in \OO_v((t))\\
g_1(x(t)) =t^{-16}+ O(t^{-14}),\quad g_2(x(t)) = -t^{-24} + O(t^{-22}) \quad g_3(x(t)) = -t^{-18} + O(t^{-16}).
\end{align*}
If $P_0 = (x_0,y_0)$ with $y_0 = 0$, then $|f^{\prime}(x_0)|_v=1$, since $C$ has good reduction. If we let $t = \frac{y}{f^{\prime}(x_0)}$, we have
\begin{align*}
x(t) = x_0+ O(t^2)\in \OO_v[[t]],\qquad 
y(t) = f^{\prime}(x_0)t\\
g_1(x(t)) = f^{\prime}(x_0)^2 + O(t^2), \quad g_2(x(t)) = -f^{\prime}(x_0)^3 + O(t^2), \quad g_3(x(t)) = x_0f^{\prime}(x_0)^2 + O(t^2). 
\end{align*}
The result follows, upon noticing that the assumption that the equation for $C$ has good reduction implies that $v\nmid 2$. 
\end{proof}
\subsection{Sigma functions expansion}\label{app:sigma_expansion}
The expansion of $\sigma_v^{(c)}(T)$ up to terms of total degree at most $8$ is given by:
\begin{equation*}
\begin{aligned}
\hspace{-50pt}\sigma_v^{(c)}(T)& = T_{1} + \left(\frac{1}{2} b_{3} + \frac{1}{2} c_{11}\right) T_{1}^{3} + c_{12} T_{1}^{2} T_{2} + \left(\frac{1}{2} c_{22}\right) T_{1} T_{2}^{2}\\
 &+ \left(\frac{3}{8} b_{3}^{2} + \frac{5}{12} b_{2} b_{4} - \frac{5}{3} b_{1} b_{5} + \frac{7}{12} b_{3} c_{11} + \frac{1}{8} c_{11}^{2} + \frac{1}{3} b_{4} c_{12}\right) T_{1}^{5}\\
  &+ \left(\frac{5}{6} b_{3} c_{12} + \frac{1}{2} c_{11} c_{12} + \frac{1}{3} b_{4} c_{22} - \frac{10}{3} b_{5}\right) T_{1}^{4} T_{2} + \left(\frac{1}{2} c_{12}^{2} + \frac{1}{4} b_{3} c_{22} + \frac{1}{4} c_{11} c_{22} - \frac{1}{2} b_{4}\right) T_{1}^{3} T_{2}^{2}\\
  &+ \left(-\frac{2}{3} b_{1} c_{12} + \frac{1}{2} c_{12} c_{22} + \frac{1}{3} c_{11}\right) T_{1}^{2} T_{2}^{3} + \left(-\frac{2}{3} b_{1} c_{22} + \frac{1}{8} c_{22}^{2} - \frac{1}{12} b_{2} + \frac{1}{3} c_{12}\right) T_{1} T_{2}^{4}\\
   &+ \bigg(\frac{5}{16} b_{3}^{3} + \frac{391}{360} b_{2} b_{3} b_{4} + \frac{13}{30} b_{1} b_{4}^{2} + \frac{13}{30} b_{2}^{2} b_{5} - \frac{547}{90} b_{1} b_{3} b_{5} + \frac{439}{720} b_{3}^{2} c_{11} + \frac{73}{120} b_{2} b_{4} c_{11} - \frac{73}{30} b_{1} b_{5} c_{11}\\
    &+ \frac{11}{48} b_{3} c_{11}^{2} + \frac{1}{48} c_{11}^{3} + \frac{61}{90} b_{3} b_{4} c_{12} - \frac{3}{5} b_{2} b_{5} c_{12} + \frac{1}{6} b_{4} c_{11} c_{12} + \frac{1}{18} b_{4}^{2} c_{22} - \frac{113}{90} b_{4} b_{5}\bigg) T_{1}^{7} \\
    &+ \bigg(\frac{89}{120} b_{3}^{2} c_{12} + \frac{49}{60} b_{2} b_{4} c_{12} - \frac{79}{15} b_{1} b_{5} c_{12} + \frac{3}{4} b_{3} c_{11} c_{12} + \frac{1}{8} c_{11}^{2} c_{12} + \frac{1}{3} b_{4} c_{12}^{2} + \frac{17}{30} b_{3} b_{4} c_{22} - \frac{3}{5} b_{2} b_{5} c_{22} \\
    &+ \frac{1}{6} b_{4} c_{11} c_{22} + \frac{4}{5} b_{4}^{2} - \frac{181}{15} b_{3} b_{5} - \frac{14}{3} b_{5} c_{11}\bigg) T_{1}^{6} T_{2} + \bigg(-b_{1} b_{4} c_{12} + \frac{7}{12} b_{3} c_{12}^{2} + \frac{1}{4} c_{11} c_{12}^{2} + \frac{3}{16} b_{3}^{2} c_{22}\\
    & + \frac{5}{24} b_{2} b_{4} c_{22} - \frac{17}{6} b_{1} b_{5} c_{22} + \frac{7}{24} b_{3} c_{11} c_{22} + \frac{1}{16} c_{11}^{2} c_{22} + \frac{1}{2} b_{4} c_{12} c_{22} - \frac{3}{4} b_{3} b_{4} - \frac{5}{2} b_{2} b_{5} - \frac{1}{4} b_{4} c_{11} - \frac{34}{3} b_{5} c_{12}\bigg) T_{1}^{5} T_{2}^{2}\\
   &   + \bigg(-\frac{5}{9} b_{1} b_{3} c_{12} - \frac{1}{3} b_{1} c_{11} c_{12} + \frac{1}{6} c_{12}^{3} - \frac{11}{9} b_{1} b_{4} c_{22} \\
    &  + \frac{5}{12} b_{3} c_{12} c_{22} + \frac{1}{4} c_{11} c_{12} c_{22} + \frac{1}{6} b_{4} c_{22}^{2} - \frac{1}{9} b_{2} b_{4} + \frac{5}{18} b_{3} c_{11} + \frac{1}{6} c_{11}^{2} - \frac{43}{18} b_{4} c_{12} - \frac{20}{3} b_{5} c_{22}\bigg) T_{1}^{4} T_{2}^{3} \\
     & + \bigg(-\frac{2}{3} b_{1} c_{12}^{2} - \frac{1}{3} b_{1} b_{3} c_{22} - \frac{1}{3} b_{1} c_{11} c_{22} + \frac{1}{4} c_{12}^{2} c_{22} + \frac{1}{16} b_{3} c_{22}^{2} + \frac{1}{16} c_{11} c_{22}^{2} - \frac{1}{24} b_{2} b_{3} + \frac{1}{2} b_{1} b_{4} \\
      &- \frac{1}{24} b_{2} c_{11} - \frac{5}{6} b_{3} c_{12} + \frac{1}{2} c_{11} c_{12} - \frac{9}{4} b_{4} c_{22} - \frac{3}{2} b_{5}\bigg) T_{1}^{3} T_{2}^{4} \\
      &+ \bigg(\frac{6}{5} b_{1}^{2} c_{12} - b_{1} c_{12} c_{22} + \frac{1}{8} c_{12} c_{22}^{2} - \frac{3}{5} b_{1} c_{11}       - \frac{41}{60} b_{2} c_{12} + \frac{1}{3} c_{12}^{2} - b_{3} c_{22} + \frac{1}{6} c_{11} c_{22} - \frac{1}{5} b_{4}\bigg) T_{1}^{2} T_{2}^{5} \\
      &+ \left(\frac{64}{45} b_{1}^{2} c_{22} - \frac{1}{3} b_{1} c_{22}^{2} + \frac{1}{48} c_{22}^{3} + \frac{19}{90} b_{1} b_{2} - \frac{37}{45} b_{1} c_{12} - \frac{77}{120} b_{2} c_{22} + \frac{1}{6} c_{12} c_{22} - \frac{1}{15} b_{3} + \frac{1}{18} c_{11}\right) T_{1} T_{2}^{6} \\
      &+O(T_1,T_2)^9.
\end{aligned}
\end{equation*}

\bibliographystyle{amsalpha}
\bibliography{biblio}
\end{document}